\documentclass[11pt]{article}
\usepackage[T1]{fontenc}
\usepackage[margin=1in]{geometry}
\usepackage{lmodern}
\usepackage{microtype}
\usepackage{amsmath,amssymb,mathtools,amsthm}
\usepackage{graphicx}
\usepackage{enumitem}
\usepackage{subcaption}
\usepackage{algorithm,algorithmic}
\usepackage{stackengine}
\usepackage[hidelinks]{hyperref}
\usepackage[capitalize,nameinlink]{cleveref}

\numberwithin{equation}{section}
\numberwithin{algorithm}{section}
\newtheorem{theorem}{Theorem}[section]
\newtheorem{lemma}[theorem]{Lemma}
\newtheorem{corollary}[theorem]{Corollary}

\theoremstyle{definition}
\newtheorem{definition}[theorem]{Definition}
\newtheorem{assumption}{Assumption}
\theoremstyle{remark}
\newtheorem{remark}[theorem]{Remark}

\crefname{section}{section}{sections}
\crefname{subsection}{subsection}{subsections}
\crefname{assumption}{Assumption}{Assumptions}
\crefname{algorithm}{Algorithm}{Algorithms}
\allowdisplaybreaks

\newcommand{\email}[1]{\href{mailto:#1}{\texttt{#1}}}
\title{Adaptive Algorithms for Robust Phase Retrieval}
\author{Zhong Zheng\thanks{Dept. of Applied and Computational Mathematics and Statistics, University of Notre Dame (\email{zzheng2@nd.edu})}
\and Necdet Serhat Aybat\thanks{Department of Industrial and Manufacturing Engineering, The Pennsylvania State University (\email{nsa10@psu.edu})}
\and Shiqian Ma\thanks{Data Science and AI Institute, and Department of Applied Mathematics and Statistics, Johns Hopkins University (\email{sqma@jhu.edu})}
\and Lingzhou Xue\thanks{Department of Statistics, The Pennsylvania State University (\email{lzxue@psu.edu})}}
\date{}

\DeclareMathOperator*{\argmin}{arg\,min}
\newcommand{\be}{\begin{equation}}
\newcommand{\ee}{\end{equation}}
\newcommand{\fprod}[1]{\left\langle#1\right\rangle}
\newcommand{\norm}[1]{\|#1\|}
\newcommand{\reals}{\mathbb{R}}
\newcommand{\integers}{\mathbb{Z}}
\newcommand{\cI}{\mathcal{I}}
\newcommand{\cO}{\mathcal{O}}
\newcommand{\cond}{\texttt{Cond}}
\newcommand{\adasub}{\texttt{AdaSubGrad}}
\newcommand{\adaipl}{\texttt{AdaIPL}}
\newcommand{\pl}{\texttt{PL}}
\newcommand{\ipl}{\texttt{IPL}}
\newcommand{\psub}{\texttt{PSubGrad}}
\newcommand{\gsub}{\texttt{GSubGrad}}
\newcommand{\apd}{\texttt{APD}}
\newcommand{\apg}{\texttt{APG}}
\newcommand{\lac}{\textbf{(LAC)}}
\newcommand{\hac}{\textbf{(HAC)}}
\newcommand{\lace}{\textbf{(LAC-exact)}}
\newcommand{\hace}{\textbf{(HAC-exact)}}

\stackMath
\newcommand\tsup[2][2]{%
  \def\useanchorwidth{T}%
  \ifnum#1>1%
    \stackon[-.5pt]{\tsup[\numexpr#1-1\relax]{#2}}{\scriptscriptstyle\sim}%
  \else%
    \stackon[.5pt]{#2}{\scriptscriptstyle\sim}%
  \fi%
}

\hypersetup{
  pdftitle={Adaptive Algorithms for Robust Phase Retrieval},
  pdfauthor={Zhong Zheng, Necdet Serhat Aybat, Shiqian Ma, Lingzhou Xue}
}

\begin{document}

\maketitle

\begin{abstract}
This paper considers robust phase retrieval, which can be cast as a nonsmooth
and nonconvex optimization problem. We propose two first-order algorithms with adaptive step sizes: the subgradient algorithm (\adasub{}) and the inexact proximal linear algorithm (\adaipl{}). Our contribution lies in a novel design of adaptive step sizes based on quantiles of the absolute residuals. We analyze local linear convergence of both algorithms across different hyperparameter regimes under i.i.d. centered sub-Gaussian measurements, a stability condition linking the measurement distribution and the corruption level, and an additional uniform small-ball condition. Numerical experiments on synthetic datasets and image recovery also demonstrate that our methods are
competitive with existing methods in the literature that utilize predetermined (possibly impractical) step sizes, such as subgradient methods and the inexact proximal linear method.
\end{abstract}


\section{Introduction}\label{sec:intro}
Phase retrieval~(PR) aims to recover a signal from intensity- or magnitude-based measurements. It finds various applications in different fields, including X-ray crystallography \cite{miao1999extending}, optics \cite{millane1990phase}, diffraction and array imaging \cite{chai2010array},
 astronomy \cite{fienup1987phase}, and microscopy \cite{miao2008extending}. Mathematically, PR tries to find the true signal vectors $x_\star$ or $-x_\star$ in $\mathbb{R}^n$ from a set of magnitude measurements:
\begin{equation}\label{intro_noiseless_phase}
   b_{i}=    \fprod{a_i,x_\star}^2, \text{ for } i=1,2,\ldots,m,
\end{equation}
where $a_i\in \mathbb{R}^n$ and $b_i\geq 0, i=1,2,\ldots,m$. Directly solving the equations in \eqref{intro_noiseless_phase} is an NP-hard problem \cite{fickus2014phase}, and algorithms based on different designs of objective functions have been well studied in the literature, including Wirtinger flow \cite{candes2015phase}, truncated Wirtinger flow \cite{chen2017solving}, truncated amplitude flow \cite{wang2017solving}, and reshaped Wirtinger flow \cite{zhang2017nonconvex}.

In this paper, we focus on the robust phase retrieval (RPR) problem \cite{duchi2019solving}, which considers the case where $b_i$ might contain infrequent but arbitrary noise due to measurement errors, i.e.,
\begin{equation}\label{intro_corrupted_phase}
    b_{i} = \begin{cases}
\fprod{a_i,x_\star}^2, & i\in \mathcal{I}_1,\\
\xi_i, & i\in \mathcal{I}_2,
\end{cases}
\end{equation}
in which $\mathcal{I}_1\bigcup \mathcal{I}_2 = \{1,2\ldots,m\}$, $\mathcal{I}_1\cap\mathcal{I}_2 = \emptyset$, and $\xi_i$ denotes the noise that only exists in measurements in $\mathcal{I}_2$, and it follows an arbitrary distribution. We will refer to such noisy measurements as \textit{corrupted measurements} (this is the name used by \cite{duchi2019solving}). Our objective is to recover $x_\star$ or $-x_\star$
using $\{(a_i,b_i)\}_{i=1}^m$ without knowing $\mathcal{I}_1$ or $\mathcal{I}_2$. \cite{duchi2019solving} proposed to formulate RPR as
an optimization problem employing the nonsmooth $\ell_1$-loss:
\begin{equation}\label{duchi_l1_ori}
    \min_{x\in \mathbb{R}^{n}} F(x)\triangleq \frac{1}{m}\sum_{i=1}^{m}\left|
    \fprod{a_i,x}^2-b_i\right| = h(c(x)),
\end{equation}
where $h(z) \triangleq \frac{1}{m}\|z\|_1$, and $c(x) \triangleq  |Ax|^2-b$ is a smooth map in which $|\cdot|^2$ operates element-wise on $Ax = [\langle a_1, x\rangle, \ldots, \langle a_m, x\rangle]^\top$ and $b=[b_1,\ldots,b_m]^\top$. In the rest of the paper, $A\in\reals^{m\times n}$ such that $A=[a_1, a_2, \ldots, a_m]^\top$ denotes the measurement matrix.

In \cite{duchi2019solving}, Duchi and Ruan have demonstrated that under standard statistical assumptions on the data generation process \eqref{intro_corrupted_phase}, $F$ satisfies the sharpness condition (formalized later as Assumption \ref{ass:sharpness}) with high probability. Crucially, the sharpness property implies that the set of global minimizers of $F(\cdot)$ is exactly $\{x_\star, -x_\star\}$, where $x_\star$ is the true signal vector. Since we can only observe the magnitude of the measurements $a_i^\top x_\star$ for $i = 1,2,\ldots,m$, we can only recover $x_\star$ up to a potential sign flip. Thus, we provide the following definition.
\begin{definition}
\label[definition]{def:eps-optimal}
    ${x_\epsilon}\in\mathbb{R}^n$ is an $\epsilon$-optimal solution if $\Delta(x_\epsilon)\leq\epsilon$, where
\begin{align}
\label{eq:Delta}
    \Delta(x)\triangleq \min\{\|x-x_\star\|_2,\|x+x_\star\|_2\}.
\end{align}
\end{definition}
In \cite{duchi2019solving}, it is shown that \eqref{duchi_l1_ori} possesses better recoverability than the median truncated Wirtinger flow algorithm \cite{zhang2016provable} based on the $\ell_2$-loss.

\subsection{Existing Algorithms and Challenges}
In the literature, two kinds of algorithms have been proposed for solving \eqref{duchi_l1_ori}: \textit{(i)} subgradient-type algorithms, and \textit{(ii)} proximal-linear (PL)-type algorithms.

First, we review existing subgradient-type algorithms~\cite{davis2018subgradient,davis2020nonsmooth} which can handle \eqref{duchi_l1_ori} while avoiding (inexact) subproblem solves. Polyak subgradient descent (\psub) is investigated in \cite{davis2020nonsmooth}:
\begin{equation*}
    x^{k+1} = x^k - \left(F(x^k) - F(x_\star)\right)\xi_k/\|\xi_k\|_2^2,\quad \xi_k\in\partial F(x^k),\quad k\in\mathbb{N}.
\end{equation*}
Subgradient algorithms with geometrically decaying step sizes (\gsub) are proposed in \cite{davis2018subgradient}:
\begin{equation}\label{descent_subgradient_geome}
    x^{k+1} = x^k - \lambda_k \xi_k/\|\xi_k\|_2,\quad \xi_k\in\partial F(x^k),\quad \lambda_k = \lambda_0 q^k, \quad k\in\mathbb{N},
\end{equation}
where $\lambda_0>0$ and $q\in (0,1)$ are constant algorithm parameters.
 For both algorithms, local linear convergence has been shown.
 \psub{} only works for noiseless robust phase retrieval \eqref{intro_noiseless_phase} as it relies on the value of $F(x_\star)$,
 and knowing a lower bound is not sufficient to establish convergence for \psub. Moreover, as discussed in \cite{davis2020nonsmooth}, using fixed step sizes for subgradient algorithms only leads to suboptimal solutions. On the other hand, for \gsub, there is still no practical guidance for properly choosing the hyperparameters ($\lambda_0$ and $q$).
 More precisely, for any fixed $\gamma\in (0,1)$, the results in~\cite{davis2020nonsmooth} require setting $\lambda_0 = \gamma \lambda_s^2/(L B_\xi)$ and $q = \sqrt{1-(1-\gamma)(\lambda_s/B_\xi)^2}$, explicitly depending on the \textit{unknown} quantity $\lambda_s$ (see Assumption~\ref{ass:sharpness}), where $B_\xi \triangleq \sup\{\|\xi\|_2:\ \xi\in\partial F(x),\ x\in \mathcal{T}_\gamma\}$ and $\mathcal{T}_\gamma = \{x\in\mathbb{R}^n: \Delta(x)\leq \gamma\lambda_s/L\}$. In summary, for \eqref{duchi_l1_ori} with measurements as in~\eqref{intro_corrupted_phase}, no \textit{practical} hyperparameter choice is known for~\cite{davis2018subgradient,davis2020nonsmooth} with theoretical guarantees, and under improper step size choices, these algorithms might not perform well or even possibly fail to converge~\cite{davis2018subgradient, davis2020nonsmooth, zheng2023new}.

Next, we review proximal-linear-type algorithms. For any given $z,y\in\reals^n$ and $t>0$, let 
\begin{equation}\label{def-Ft}
    F(z;y) \triangleq h(c(y) + J_c(y)(z-y)),\quad
    F_t(z;y) \triangleq F(z;y) + \frac{1}{2t}\|z-y\|_2^2,
\end{equation}
where $J_c(\cdot){\in\reals^{m\times n}}$ denotes the Jacobian of $c(\cdot)$ and can be written explicitly as $J_c(y) = 2\mbox{diag}(Ay)A$. In one typical iteration of a PL-type algorithm, one \textit{inexactly} solves a subproblem of the form:
\be\label{PL}
x^{k+1} \approx \argmin_{x\in \mathbb{R}^{n}}\ F_{t_k}(x;x^k),
\ee
where $t_k>0$ is the chosen step size, and ``$\approx$" means that
$x^{k+1}$ is an ``inexact" solution to the subproblem in \eqref{PL}. PL-type methods have also been studied for general nonconvex composite optimization problems; see, e.g., \cite{drusvyatskiy2019efficiency,wang2022manifold,zheng2025new}.
Using fixed step sizes $t_k = L^{-1}$ for all $k\in\mathbb{N}$,
with $L\triangleq 2\|A\|_2^2/m$, \cite{duchi2019solving} has proposed the proximal linear algorithm. They establish a local quadratic convergence rate for the \pl{} method in terms of the iteration counter $k$, i.e., in the number of minimizations of the form $\argmin_x F_{t_k}(x;x^k)$ when the $k$-th subproblem in \eqref{PL} is solved to $\epsilon_k$-suboptimality in function values such that $\epsilon_k\leq 2^{-2^k}$ for $k\geq 0$.
However, the \textit{total complexity}, i.e., total number of gradient calls required for all inner iterations, for the \pl{} method remains unknown. This is because the cost of solving the subproblems to such high precision is not accounted for in their analysis. It is worth emphasizing that closed-form solutions to these subproblems are unavailable, and in practice, one cannot compute them exactly using iterative methods.
On the other hand, for the numerical experiments in~\cite{duchi2019solving}, each subproblem in the form of
\eqref{PL} was inexactly solved by the proximal operator graph splitting (\texttt{POGS}) method~\cite{parikh2014block}, terminated as suggested in~\cite{parikh2014block}, i.e., when the primal and dual residuals satisfy a predetermined threshold---that said, the convergence analysis for this strategy was not provided in \cite{duchi2019solving}.

To get better control over the cost of solving \eqref{PL}, under the same fixed step sizes, \cite{zheng2023new} has proposed the Inexact Proximal Linear (\ipl) algorithm that solves \eqref{PL} inexactly using one of the following \textit{inexact} termination conditions:
\begin{align}
\label{low-high-0}
{
    F_{t_k}(x^{k+1};x^k)-\min_{x\in \mathbb{R}^n}\ F_{t_k}(x;x^k)\leq
    \begin{cases}
        \rho_l\left(F(x^k)-F_{t_k}(x^{k+1};x^k)\right) &\mbox{(\textbf{LAC-exact})}\\
        \frac{\rho_h}{2t_k}\|x^{k+1}-x^k\|_2^2 &\mbox{(\textbf{HAC-exact})}
    \end{cases}}\end{align}
where $\rho_l>0$, $\rho_h\in (0,1/4)$ are given positive constants. Here, LAC is short for the low accuracy condition, and HAC is short for the high accuracy condition. Since $\min_x F_{t_k}(x;x^k)$ is not known in practice, to verify these conditions one needs to work with sufficient conditions for
\eqref{low-high-0} obtained by replacing
$\min_{x\in \mathbb{R}^n}F_{t_k}(x;x^k)$ with the dual function values of \eqref{PL}. In~\cite{zheng2023new}, \eqref{PL} is solved through applying the Accelerated Proximal Gradient algorithm (\apg{}) given in~\cite{tseng2008accelerated} to
the dual problem of \eqref{PL}. In~\cite{zheng2023new}, it is proven that \ipl{} can compute an $\epsilon$-optimal solution for \eqref{duchi_l1_ori} within $\cO(1/\epsilon)$ inner iterations in total for inexactly solving a sequence of subproblems in the form of \eqref{PL}, establishing the total complexity for the double-loop algorithm \ipl{}. Numerical experiments in \cite{zheng2023new} empirically show that \ipl{} enjoys better numerical performance in terms of CPU time compared to \pl{} method.
However, the efficiency of \ipl{} is still unsatisfying due to the sublinear convergence rate; that said, there is room for further improvement for \ipl{}, and it is one objective of this paper.

 Stochastic algorithms are also studied in \cite{duchi2018stochastic, davis2018stochastic, davis2019stochasticsub0, davis2019stochasticsub1, davis2019stochastic}, where the proximal-linear type and the subgradient-type methods are unified. In addition, \cite{davis2020nonsmooth} and \cite{zheng2024smoothed} also analyze the nonconvex landscape. These types of methods and their analysis are beyond the scope of our paper. 
We also note the algorithm \texttt{Robust-AM} proposed in the recent work ~\cite{kim2024robust}. They solve the robust phase retrieval problem via the optimization problem:
    $\min_{x\in\mathbb{R}^n}\frac{1}{m}\sum_{i=1}^m \left||a_i^\top x| - \sqrt{b_i}\right|.$
    The authors solve it by iteratively solving the following subproblem: $x^{k+1} = \argmin_{x\in\mathbb{R}^n} \frac{1}{m}\sum_{i=1}^m |(a_i^\top x) - \mbox{sign}(a_i^\top x^k)\sqrt{b_i}|$.
\cite{kim2024robust} prove the main-iteration local linear convergence. Although the subproblem can be solved exactly via linear programming, \cite{kim2024robust} use the Alternating Direction Method of Multipliers (ADMM) to solve it inexactly for better numerical efficiency.

\subsection{Proposed Algorithms: \adasub{} and \adaipl{}}
 In this paper, we propose to incorporate an adaptive step size strategy within subgradient and inexact-proximal linear algorithm frameworks for improving their convergence behavior both in theory and practice. Next, we give some definitions. Throughout we set $L\triangleq 2\|A\|_2^2/m$. For a $\tilde{p}\in (0,1)$ such that $m\tilde{p}\in\mathbb{N}_+$, let
 \begin{equation}\label{def_r_med}
     {r_i(\cdot) \triangleq } |\fprod{a_i,\cdot}^2 - b_i|,\quad {\forall}~i\in [m];\quad r^{\Tilde{p}}(\cdot) \triangleq \mbox{ the }\Tilde{p} \mbox{-th quantile of } \{r_i(\cdot)\}_{i=1}^m,
 \end{equation}
 i.e., $r^{\Tilde{p}}(x)$ denotes the $(m\Tilde{p})$-th order statistic of absolute residuals $\{r_i(x)\}_{i=1}^m $ at any given $x$. The quantile operator is known for its robustness to outliers \cite{koenker1978regression}; hence, it helps us design a step size strategy robust to the corrupted measurements for \eqref{duchi_l1_ori}, i.e., when the true measurement vector is corrupted by a
 sparse noise vector with non-zero entries having arbitrarily large magnitudes.

\subsubsection{\adasub{}}
 We propose a subgradient method with adaptive step sizes (\adasub). The step size $\alpha_k$ at iteration $k\in\mathbb{N}$ is chosen as follows:
 \begin{equation}\label{choice_alpha_subg}
     \alpha_k = G~r^{\Tilde{p}}(x^k),\quad \forall~k\in\mathbb{N},
 \end{equation}
 where $G>0$ is an algorithm parameter. The \adasub{} iterates are computed as
 \begin{equation}\label{iter_adasub}
 x^{k+1} = x^k - \alpha_k\xi^k/\|\xi^k\|_2^2,\quad \xi^k\in\partial F(x^k),\quad \forall~k\in\mathbb{N}.
 \end{equation}
 According to \cite[Theorem 10.6, Corollary 10.9]{RockWets98}, the subdifferential of weakly convex function $F$ has the form $\partial F(x)\triangleq [{J_c(x)}]^\top \partial h(c(x))$ for $x\in\mathbb{R}^n$, where $J_c(x)\in\reals^{m\times n}$ denotes the Jacobian of $c$ at $x$. A subgradient $\xi_k\in\partial F(x^k)$ can be computed as
 $\xi^k = \frac{2}{m}\sum_{i=1}^m \fprod{a_i,x^k}\mbox{sign}\Big(\fprod{a_i,x^k}^2-b_i\Big)a_i$ for any $k\in\mathbb{N}.$ Our proposed method \adasub{} is formally
 stated in Algorithm \ref{alg:adaptive-sub}.

 \begin{algorithm}[ht]
\caption{Subgradient Algorithm with Adaptive Step Sizes (\adasub{})}
\label{alg:adaptive-sub}
{\small
\begin{algorithmic}
\STATE \textbf{Input:} Initial point $x^0\in\mathbb{R}^n$, parameter $G>0,$ percentile $\tilde{p}\in (0,1)$ such that $m\tilde{p}\in\mathbb{N}_+$.
\FOR{$k = 0, 1, \ldots, $}
\STATE{Update $x^{k+1}$ using \eqref{iter_adasub} with $\alpha_k$ given in \eqref{choice_alpha_subg}.}
\ENDFOR
\end{algorithmic}}\end{algorithm}

Note that in our choice of the step size $\alpha_k$, we adopt the quantile design, i.e., $r^{\Tilde{p}}(x^k)$. The main motivation behind our choice is that under the data generation process discussed in \Cref{sec:data-gen} (see~Assumption~\ref{ass_data}), one can show that $\alpha_k = \Theta(F(x^k) - F(x_\star))$
for all $k\in \mathbb{N}$ with \textit{high probability}. This result relies in particular on the uniform small-ball condition in Assumption~\ref{ass_data}(c), which is new to the robust phase retrieval literature. Nevertheless, \Cref{sec:data-gen} shows that this condition holds for a broad class of continuous elliptical distributions, including centered Gaussian distributions with positive definite covariance. Therefore, \adasub{} will exhibit convergence behavior similar to that of the Polyak subgradient algorithm. More precisely, we prove that for sufficiently small $G>0$, \adasub{} enjoys a local linear convergence.

\subsubsection{\adaipl{}}
 Second, we introduce the inexact proximal linear algorithm with adaptive step sizes (\adaipl{}). Given some positive constant $G>0$, let
 \begin{equation}\label{choice_t_diminish}
t_k \triangleq \min\{L^{-1},G~ r^{\Tilde{p}}(x^k)\}.
\end{equation}
In \adaipl{}, iterates are computed by \textit{inexactly} solving \eqref{PL} with $t_k$ chosen as in \eqref{choice_t_diminish} such that for all $k\in \mathbb{N}$ either \textbf{(LAC-exact)}
or \textbf{(HAC-exact)} given in \eqref{low-high-0} holds. Since $\min_{x\in \mathbb{R}^n}\ F_{t_k}(x;x^k)$ appearing in \eqref{low-high-0} is not available in practice, in \adaipl{} we replace \eqref{low-high-0} with a practical one as described next. Given $t_k\leq \frac{1}{L} = \frac{m}{2\|A\|_2^2}$,
\eqref{PL} is convex and can be equivalently written as
\be\label{PL-rewrite-2}
\min_{z\in \mathbb{R}^n} \ H_k(z) \triangleq \frac{1}{2t_k}\|z\|_2^2 + \|B_kz-d_k\|_1,
\ee
{after the change of variables:
$z \triangleq x -x^k$ and setting}
$B_k \triangleq \frac{2}{m}\mbox{diag}(Ax^k)A$ and
$d_k \triangleq \frac{1}{m}\left(b - |Ax^k|^2\right)$. The problem in~\eqref{PL-rewrite-2} has the following min-max and dual forms:
\begin{subequations}
\begin{align}
    \min_{z\in \mathbb{R}^n}\max_{\lambda\in \mathbb{R}^m:~\|\lambda\|_\infty\leq 1} H_k(z,\lambda) &\triangleq \frac{1}{2t_k}\|z\|_2^2+\lambda^\top (B_kz-d_k),\label{sub_minmax}\\
    \max_{\lambda\in \mathbb{R}^m:~\|\lambda\|_\infty\leq 1} D_k(\lambda) &\triangleq -\frac{t_k}{2}\left\|\lambda^\top B_k\right\|_2^2-\lambda^\top d_k.\label{sub_dual}
\end{align}
\end{subequations}
Let
$z^k(\lambda) \triangleq -t_kB_{k}^\top \lambda$ and $\lambda^k(z) \triangleq \mbox{sign}\left(B_kz - d_k\right)$. Using $z^{k} \triangleq x^{k+1} - x^k$, we can rewrite \eqref{low-high-0} as
\begin{equation}
\label{eq:low-high}
{
\begin{aligned}
    {H_k(z^{k}) - \min_{z\in\mathbb{R}^n} H_k(z)\leq
    \begin{cases}
        \rho_l\left(H_k(0) - H_k(z^{k})\right), & \mbox{(\textbf{LAC-exact})}\\
        \frac{\rho_h}{2t_k}\|z^{k}\|_2^2, & \mbox{(\textbf{HAC-exact}),}
    \end{cases}}
\end{aligned}}\end{equation}
where $\rho_l>0$, $\rho_h\in (0,1/4)$ are given positive constants. As $\min_{z\in\mathbb{R}^n} H_k(z)$ may not be easily available, we provide \textit{sufficient} conditions for (\textbf{LAC-exact}) and (\textbf{HAC-exact}) conditions in \eqref{eq:low-high} that can be checked in practice. Due to weak duality, we have
\begin{equation*}
    H_k(z^{k}) - \min_{z\in\mathbb{R}^n} H_k(z)\leq H_k(z^{k}) - D_k(\lambda),\quad \forall~\lambda\in\mathbb{R}^m:\ \|\lambda\|_\infty\leq 1.
\end{equation*}
Consider a generic solver such that when initialized from an arbitrary $(z^{k}_0,\lambda^{k}_0)$, it generates a primal-dual iterate sequence $\{(z^{k}_j,\lambda^{k}_j)\}_{j=0}^\infty\subset \mathbb{R}^n\times \mathbb{R}^m$ for the $k$-th subproblem
satisfying
$\sup_{j\in\mathbb{N}}\|\lambda_j^k\|_\infty\leq 1$, and after
$j_k$ iterations it computes a primal-dual pair $(z^k,\lambda^k) =(z^{k}_{j_k},\lambda^{k}_{j_k})$ such that
\begin{equation}
\label{eq:low-high-practical}
{
\begin{aligned}
    H_k(z^{k}) - D_k(\lambda^k)\leq
    \begin{cases}
        \rho_l\left(H_k(0) - H_k(z^{k})\right), & \mbox{(\textbf{LAC})}\\
        \frac{\rho_h}{2t_k}\|z^{k}\|_2^2, & \mbox{(\textbf{HAC}).}
    \end{cases}
\end{aligned}}\end{equation}
Since \eqref{eq:low-high-practical} is a sufficient condition for \eqref{eq:low-high}, for \adaipl{} we can adopt the practical condition given in \eqref{eq:low-high-practical} instead of the condition in \eqref{eq:low-high}.
\begin{lemma}\label[lemma]{sufficiency_pd}
    \emph{(\textbf{LAC})} and \emph{(\textbf{HAC})}
    imply \emph{(\textbf{LAC-exact})} and \emph{(\textbf{HAC-exact})},
    respectively.
\end{lemma}

\begin{algorithm}[ht]
\caption{Inexact Proximal Linear Algorithm with Adaptive Step Sizes (\adaipl{})}
\label{alg:adaptive-IPL}
{\small
\begin{algorithmic}
\STATE \textbf{Input:} Initial point $x^0\in\reals^n$, $\rho_l>0$ or $\rho_h\in (0,1/4)$, percentile $\tilde{p}\in (0,1)$ s.t. $m\tilde{p}\in\mathbb{N}_+$.
\STATE {Set \texttt{Cond} to either \textbf{(LAC)} or to \textbf{(HAC)} in \eqref{eq:low-high-practical}}
\FOR{$k = 0, 1, \ldots, $}
\STATE {Compute $t_k$ as in \eqref{choice_t_diminish}}
\STATE {Compute $x^{k+1}$ by inexactly solving \eqref{PL} such that \texttt{Cond} holds}
\ENDFOR
\end{algorithmic}}\end{algorithm}

A pseudocode for \adaipl{} is given in Algorithm~\ref{alg:adaptive-IPL}. It is only a prototype algorithm because the method for inexactly solving \eqref{PL} is not fixed at this point. We will discuss this issue later in Section \ref{subsec:subproblem_solvers}. For \adaipl{}, the step size $t_k$ uses the quantile design such that under the data generation process introduced in Section \ref{sec:data-gen}, $t_k = \Theta(\Delta(x^k))$ for all $k\geq 0$ with high probability, where $\Delta(\cdot)$ is defined in~\eqref{eq:Delta}. Analysis in Section \ref{sec:conv_adaipl} will show that under adaptive step sizes, we can reach a better balance between main iteration complexity and subproblem iteration complexity compared to \ipl{} with fixed step sizes proposed in~\cite{zheng2023new}. We will show that for any choice of algorithm parameter $G>0$, \adaipl{} enjoys local linear convergence in terms of total iteration complexity associated with inexactly solving all the subproblems in the form of \eqref{PL}.
\subsubsection{Summary of Contributions}\label{sec:summary_contribution}
We propose \adasub{} and \adaipl{} with adaptive step sizes chosen based on the quantiles of absolute residuals. To the best of our knowledge, we are the first to use quantile-based adaptive step sizes for robust phase retrieval to design practical methods that do not require intensive hyperparameter tuning; moreover, unlike \cite{davis2018subgradient}, the convergence guarantees of our methods do not require the algorithm parameters to satisfy some conditions involving unknown problem constants. The proposed algorithms enjoy the following advantages over existing algorithms.
\renewcommand{\arraystretch}{1.2}
\begin{table}
    \centering
    {\small
    \begin{tabular}{|c|c|c|c|}
    \hline
    \textbf{Algorithm} & \textbf{Ideal Complexity} & \textbf{Step sizes} & \textbf{Tuning} \\
    \hline
    \pl{}\cite{duchi2019solving}  & Unknown & Fixed & N/A \\
       \hline
    \ipl{}\texttt{-LAC}\cite{zheng2023new} & $\cO(C_S^2\kappa_0^2\|x_\star\|_2/\epsilon)$ & Fixed & N/A\\
       \hline
    \ipl{}\texttt{-HAC} \cite{zheng2023new} & $\cO(C_S^2\kappa_0^3\|x_\star\|_2/\epsilon)$ & Fixed & N/A\\
    \hline
    \psub\cite{davis2020nonsmooth}& $\cO(\kappa_0^{2}\log \frac{1}{\epsilon})$ & Adaptive & Hard\\
       \hline
         \gsub \cite{davis2018subgradient} & $\cO(\kappa_0^{2}\log \frac{1}{\epsilon})$ & Predetermined & Hard\\
    \hline
    \adasub{} \textbf{(this work)} & $\cO(\kappa_0^2\log \frac{1}{\epsilon})$ & Adaptive & Easy\\
    \hline
    \adaipl{}\texttt{-LAC} \textbf{(this work)} & $\cO(C_S\kappa_0\log \frac{1}{\epsilon})$ & Adaptive & Easy\\
    \hline
    \adaipl{}\texttt{-HAC} \textbf{(this work)} & $\cO(C_S\kappa_0\log \frac{1}{\epsilon})$ & Adaptive & Easy\\
    \hline
    \end{tabular}}\caption{Comparison of algorithms for RPR. In the first column, for \ipl{} and \adaipl{}, ``\texttt{-LAC}" and ``\texttt{-HAC}" correspond to \lac{} and \hac{} conditions, respectively. The second column compares the total complexity under the ideal choices of hyperparameters for finding an $\epsilon$-optimal point when $\epsilon>0$ is small enough. $\kappa_0\geq 1$ denotes the condition number, $C_S{\geq 1}$ is a factor related to solving \eqref{PL}, and $\cO(\cdot)$ only hides numerical constants. The third column characterizes the types of step sizes for each algorithm. The fourth column summarizes the difficulty level of hyperparameter tuning. ``N/A" here means that \pl{} and \ipl{} that employ fixed step sizes do not need tuning---refer to the second bullet point in Section \ref{sec:summary_contribution} for explanations of ``Easy" and ``Hard".}\label{tab:perfect_situation_modified}
\end{table}
\begin{itemize}[leftmargin=*]
    \item Our algorithms enjoy the best theoretical convergence rates. For finding an $\epsilon$-optimal solution, we define the total complexity of subgradient-type algorithms as the number of iterations and the total complexity for proximal-linear-type algorithms as the total iterations used for inexactly solving all the subproblems \eqref{PL}.
    Under the ideal situation in terms of hyperparameters that will be explained in Section \ref{sec:conv_adaipl}, Table \ref{tab:perfect_situation_modified} summarizes the total complexity
    of all candidate algorithms for finding an $\epsilon$-optimal solution where $\epsilon$ is sufficiently small. Here, $\kappa_0\geq 1$ is the condition number of RPR that will be explained in Section \ref{sec:data-gen}, $C_S = \sqrt{m}\max_{i\in [m]} \|a_i\|_2/\|A\|_2$ is a constant factor related to the complexity of solving \eqref{PL}, and we treat $\rho_l,\rho_h$ in \lac{} and \hac{} in \eqref{eq:low-high-practical} as numerical constants.
    \adasub{} enjoys a local linear rate comparable to other subgradient algorithms, and \adaipl{} enjoys a better local
 linear rate in terms of the condition number.
    \item Our algorithms enjoy a linear rate even under imperfect choices of hyperparameters: \adaipl{} shows local linear convergence for any $G>0$, and \adasub{} enjoys local
 linear convergence when $G$ is sufficiently small. In contrast, \psub \cite{davis2020nonsmooth} relies on the value of $F(x_\star)$, and \gsub \cite{davis2018subgradient} only converges under their specific choice of parameters that depend on unknown constants for \eqref{duchi_l1_ori}. The difficulty of tuning hyperparameters is summarized in Table \ref{tab:perfect_situation_modified}.
    \item We conduct numerical tests comparing \adasub{} and \adaipl{}
    against the other state-of-the-art methods for solving the RPR problem. Empirical results show that both \adasub{} and \adaipl{} are robust to parameter selection and perform better than the others. \end{itemize}

\textbf{Notations.} For any $m_0\in\mathbb{N}_+$, we denote $[m_0] = \{1,2\ldots m_0\}$. $\mathbb{S}^{n-1} = \{x\in\mathbb{R}^n:\|x\|_2 = 1\}$. $\mathbf{1}[\cdot]$ is the indicator function that takes logical statements as its argument; it returns $1$ when its argument is true and $0$ otherwise. For $x\in\mathbb{R}$, we let $\mbox{sign}(x) = \mathbf{1}[x>0] - \mathbf{1}[x<0]$. We also adopt the Landau notation, i.e., for $f, g: \mathbb{R}_{+} \rightarrow \mathbb{R}_{+}$, we use $f=\mathcal{O}(g)$ and $f=\Omega(g)$ if there exist some $C^0>0$ and $n^0 \in \mathbb{R}_{+}$ such that $f(n) \leq C^0 g(n)$ and $f(n) \geq C^0 g(n)$, respectively, for all $n \geq n^0$; moreover, if $f=\mathcal{O}(g)$ and $f=\Omega(g)$, then we use $f=\Theta(g)$.

\textbf{Organization.}
Section \ref{sec:property_F} introduces basic properties of $F(\cdot)$. Section \ref{sec:data-gen} discusses the data generation process and key conditions for our adaptive step sizes, and in Section \ref{sec:proof_stat}, we give the proof of
\Cref{thm_assumptions_high_prob}, which provides a proper statistical foundation for our convergence analysis.

Section~\ref{sec:adasub} establishes the convergence rate of \adasub{}, and~\Cref{sec:conv_adaipl} shows the convergence rate of \adaipl{}. In \Cref{sec:proof_adaipl}, we provide the proofs of the results given in~\Cref{sec:conv_adaipl}. Finally, after the numerical experiments in \Cref{sec:main_numerical}, we conclude the paper with a brief discussion in~\Cref{sec:conclusions}.
\section{\texorpdfstring{Basic Properties of $F(\cdot)$}{Basic Properties of F(.)}}\label{sec:property_F}
In this section, we provide some basic results regarding the properties of $F(\cdot)$. Throughout this section, $L\triangleq 2\|A\|_2^2/m$.
\begin{lemma}[Lemma 6 in \cite{zheng2023new}, local Lipschitz
continuity]\label[lemma]{Lip_F}
	 For any $r\geq 0$,
  {\small
	$$\sup\Big\{\frac{|F(x) - F(y)|}{\|x-y\|_2}:\ x,y\in \mathbb{R}^n, \Delta(x)\leq r,~\Delta(y)\leq r,~x\neq y\Big\}\leq L(\|x_\star\|_2 + r),$$}where $\Delta(\cdot)$ is defined in~\eqref{eq:Delta}.
\end{lemma}
\begin{lemma}[Absolute deviation bound]\label[lemma]{lip_F1}
For all $x\in\mathbb{R}^n$, it holds that
\begin{equation}
\label{eq:F-F*-bound}
|F(x) - F(x_\star)|\leq \frac{L}{2}\|x-x_\star\|_2\|x+x_\star\|_2.
\end{equation}
\end{lemma}
\begin{proof}When $x\in \{x_\star,-x_\star\}$, the relationship holds; otherwise, we have
\begin{align*}
\lvert F(x) - F(x_\star) \rvert
\le \frac{1}{m}\sum_{i=1}^m
\left| (a_i^\top x)^2 - (a_i^\top x_\star)^2 \right| = \|x - x_\star\|_2 \, \|x + x_\star\|_2 \,
   \frac{1}{m}\sum_{i=1}^m
   \left| u^\top a_i a_i^\top v \right|.
\end{align*}
For $u \triangleq (x-x_\star)/\|x-x_\star\|_2$ and $v \triangleq (x+x_\star)/\|x+x_\star\|_2$, we get
{\small
\begin{equation*}
    \frac{1}{m}\sum_{i=1}^m |u^\top a_ia_i^\top v|\leq \left(u^\top\left(\frac{1}{2m}\sum_{i=1}^m a_ia_i^\top\right)u + v^\top\left(\frac{1}{2m}\sum_{i=1}^m a_ia_i^\top\right)v\right)\leq L/2.
\end{equation*}}The first inequality holds because $|u^\top a_ia_i^\top v|\leq \frac{1}{2}(a_i^\top u)^2 + \frac{1}{2}(a_i^\top v)^2 = \frac{1}{2}u^\top (a_ia_i^\top) u + \frac{1}{2}v^\top (a_ia_i^\top) v$. The second follows from $L = 2\|A\|_2^2/m = 2\|\sum_{i = 1}^m a_ia_i^\top\|_2/m$.
This completes the proof.
\end{proof}

\begin{lemma}[Local linear approximation~\cite{duchi2019solving}]
\label[lemma]{thm:gen_weak}
The inequality below holds for any $x,y\in \mathbb{R}^n$:
\begin{equation}\label{rel:gen_weak1}
    {\left|F(x)-F(x;y)\right|\leq \frac{1}{2t}\|x-y\|_2^2,}\quad
    {\forall~t\in(0, 1/L].}
\end{equation}
\end{lemma}

Next, similar to \cite{duchi2019solving}, \cite{davis2018subgradient}, and \cite{davis2020nonsmooth}, we make the following sharpness assumption.
\begin{assumption}[Condition C1 in \cite{duchi2019solving}]\label{ass:sharpness}
There exists $\lambda_s>0$ such that
\begin{equation}\label{ineq:sharpness}
F(x)-F(x_\star)\geq \lambda_s\Delta(x),\quad \forall x\in\mathbb{R}^n,
\end{equation}
where $F$ is defined in \eqref{duchi_l1_ori}, and $\Delta(\cdot)$ is defined in~\eqref{eq:Delta}.
\end{assumption}

\begin{lemma}[Lemma 7 in \cite{zheng2023new}]
\label[lemma]{lip_F2}
 Under Assumption~\ref{ass:sharpness}, for any $r\geq 0$, \begin{align*}
		\{x\in\mathbb{R}^n:\Delta(x)\leq E(r)\} \subseteq  \{x\in\mathbb{R}^n: F(x) - F(x_\star)\leq r\}
		\subseteq  \{x\in\mathbb{R}^n: \Delta(x)\leq r/\lambda_s\},
	\end{align*}
implying $E(r)\leq r/\lambda_s$, where $E(r){\triangleq} \left(\sqrt{L^2\|x_\star\|_2^2+4rL} - L\|x_\star\|_2\right)/(2L)$.
\end{lemma}

\begin{lemma}\label[lemma]{rel_lamL}
Assumption~\ref{ass:sharpness} implies that $\lambda_s\leq L\|x_\star\|_2/2$.
\end{lemma}
\begin{proof}Consider $x\in\mathbb{R}^n$ such that $\|x-x_\star\|_2=\Delta(x)$ and $\Delta(x)\neq 0$. Based on Lemma \ref{lip_F1} and Assumption \ref{ass:sharpness}, we have $\lambda_s\Delta(x)\leq F(x) - F(x_\star)\leq (L/2)\Delta(x)\|x+x_\star\|_2,$
    which implies that
    $\lambda_s\leq (L/2)\|x+x_\star\|_2.$
    Letting $x = \mathbf{0}$, we have $\lambda_s\leq L\|x_\star\|_2/2.$
\end{proof}
\paragraph{Differentiability properties of $F(\cdot)$}
Recall that $F(\cdot) = h(c(\cdot))$, where $h:\reals^m\to\reals$ and $c:\reals^n\to\reals^m$ are defined by $h(\cdot) = \|\cdot\|_1/m$ and $c(\cdot) = |A\cdot|^2 - b$; hence, $\partial F(x) = [J_c(x)]^\top \partial h(c(x))$ for $x\in\mathbb{R}^n$ as defined in \eqref{iter_adasub}.
For any fixed $\bar x\in\mathbb{R}^n$,
let $c_{\bar x}(x) \triangleq c(\bar x) + [J_c(\bar x)](x - \bar x)$ for all $x\in\reals^n$; hence, $c_{\bar x}(\cdot)$ is an affine function. Consider $F(\cdot;\bar x)$ defined in \eqref{def-Ft}. Since $F(\cdot;\bar x) = h(c_{\bar x}(\cdot))$, the function $F(x;\bar x)$ is convex in $x$, and $\partial F(x;\bar x)\mid_{x=\bar x}$ can be written as $[J_c(\bar x)]^\top \partial h(c_{\bar x}(\bar x)) = [J_c(\bar x)]^\top \partial h(c(\bar x))$. Thus, $\partial F(\bar x) = \partial F(x;\bar x)\vert_{x=\bar x}$ for any $\bar x\in\reals^n$. Next, based on this observation, we provide a useful inequality for $F(\cdot)$.

For any $\bar x\in\mathbb{R}^n$ and $v\in \partial F(\bar x)$, since we have $v\in \partial F(x;\bar x)\vert_{x=\bar x}$ and $F(\cdot;\bar x)$ is convex, we have $F(x;\bar x) - F(\bar x) = F(x;\bar x) - F(\bar x;\bar x)\geq \langle x-\bar x,v\rangle$. Together with  $F(x;\bar x)\leq F(x) + L\|x-\bar x\|_2^2/2$, which follows from \eqref{rel:gen_weak1}, we conclude that
\begin{equation}\label{eq_weakly_convex_useful}
{
\begin{aligned}
    F(x) - F(\bar x) - \langle x-\bar x,v\rangle + L\|x-\bar x\|_2^2/2\geq 0,\quad \forall~x,\bar x\in\mathbb{R}^n,\ \forall~v\in\partial F(\bar x),
\end{aligned}}\end{equation}which indeed implies that $F(\cdot)$ is weakly convex, i.e., $F(\cdot)+\frac{L}{2}\norm{\cdot}^2$ is convex.

\section{Data Generation Process and Key Conditions}
\label{sec:data-gen}
In this section, we provide the statistical background and some key conditions for RPR. We first introduce the assumption for the data generation process of RPR.
\begin{assumption}
\label{ass_data}
Suppose that the following conditions hold:
\begin{enumerate}
    \item[(a)] Given non-negative integers $m_1,m_2$, let $m=m_1+m_2$, and $a_{1},a_{2},\ldots ,{a_m}$ in $\mathbb{R}^n$ be independent and identically distributed random vectors, represented by a generic random vector $a\in\reals^n$, which follows a $\sigma^2$-subGaussian distribution with $\mathbf{0}$ as mean and $\Sigma$ as covariance matrix. The true signal vector $x_\star\in\mathbb{R}^n\backslash\{\mathbf{0}\}$ satisfies \eqref{intro_corrupted_phase} for some arbitrary $\cI_1,\cI_2\subset[m]$ such that $\mathcal{I}_1\bigcup \mathcal{I}_2 = \{1,2\ldots,m\}$, $\mathcal{I}_1\cap\mathcal{I}_2 = \emptyset$ with
    $|\cI_1|=m_1$,  $|\cI_2|=m_2$, and $\xi_i$ for $i\in\mathcal{I}_2$ is a non-negative random variable following an arbitrary distribution.
\item[(b)] {Suppose} $\kappa_{\mathrm{st}}{\triangleq}\inf_{u, v \in \mathbb{S}^{n-1}} \mathbb{E}[|\langle a, u\rangle\langle a, v\rangle|]$ and $p_{\mathrm{fail}}\triangleq \frac{m_2}{m_1+m_2}<\frac{1}{2}$ satisfy
$\kappa_{\mathrm{st}} - 2p_{\mathrm{fail}}\|\Sigma\|_2>0$.
\item[(c)] (Uniform small-ball condition) Given $\Tilde{p}\in(p_{\mathrm{fail}},~1-p_{\mathrm{fail}})$, suppose that there exists
$\kappa>0$ such that
{\small
\begin{equation}
\label{eq:p0}
    p_0\triangleq \inf_{u,v\in \mathbb{S}^{n-1}}\mathbb{P}(\min\{|\fprod{a,u}|,~|\fprod{a,v}|\}\geq \kappa) > {\frac{1-\tilde{p}}{1-p_{\mathrm{fail}}}}.
\end{equation}}
\end{enumerate}
\end{assumption}
First, Assumption \ref{ass_data}(a) is slightly weaker than the combination of Model M2 and Assumption 4 in \cite{duchi2019solving}. Let $m\triangleq m_1+m_2$ denote the total number of measurements, i.e., $|\mathcal{I}_1| = m_1$, $|\mathcal{I}_2| = m_2$; moreover, $p_{\mathrm{fail}}=\frac{m_2}{m_1+m_2}$ represents the proportion of corrupted measurements---when $m_2=p_{\mathrm{fail}}=0$, this model reduces to
the noiseless phase retrieval problem. We observe $\{(a_i,b_i)\}_{i=1}^m$ without knowing which measurements are corrupted. We also guarantee that $\{a_i: i\in {\mathcal{I}_1}\}$ and $\{a_i: i\in {\mathcal{I}_2}\}$ are mutually independent and
they follow a $\sigma^2$-subGaussian distribution. Next, we compare (a) with the assumptions in \cite{duchi2019solving}.
First, throughout \cite{duchi2019solving}, the vectors $\{a_i\}_{i\in [m]}$ are assumed to be independent and identically distributed copies of a random vector $a\in\reals^n$. Moreover, \cite[Assumption 4]{duchi2019solving} requires that $a\in\reals^n$ follows a subGaussian distribution. These assumptions are equivalent to our Assumption \ref{ass_data}(a).
    Moreover, in addition to the assumptions mentioned above, Model M2 in \cite{duchi2019solving} also requires the independence between $\{\xi_i:i\in\mathcal{I}_2\}$ and $\{a_i: i\in \mathcal{I}_1\}$.

    Below, we state an existing statistical guarantee for data generation
    based on Assumption \ref{ass_data}(a).
\begin{lemma}\label[lemma]{sharpness_upper}
(Lemma 3.1 in \cite{duchi2019solving}) Let Assumption \ref{ass_data}(a) hold. Then for all $t \geq 0$,
{\small
$$
\mathbb{P}\left(\left\|\frac{1}{m}\sum_{i=1}^m a_ia_i^\top-\Sigma\right\|_2 \geq 11 \sigma^2 \max \left\{\sqrt{\frac{4 n}{m}+t}, \frac{4 n}{m}+t\right\}\right) \leq \exp (-m t).
$$}\end{lemma}
\begin{remark}
\label[remark]{rem:L}
Lemma \ref{sharpness_upper} shows that when $m/n$ is large enough, $L = 2\|\frac{1}{m}\sum_{i=1}^m a_ia_i^\top\|_2$ is close to $2\|\Sigma\|_2$ with high probability.
\end{remark}

We next discuss Assumption \ref{ass_data}(b). First, we state a result related to the sharpness condition.
\begin{lemma}\label[lemma]{lemma_stat_sharpness}
(Proposition 4 in \cite{duchi2019solving}) Under (a) and (b) of Assumption \ref{ass_data},
there exist numerical constants $C, c_3>0$ (that do not depend on any other quantities) such that
{\small
\begin{align*}
F(x)-F(x_{\star}) \geq\left(\kappa_{\mathrm{st}}-2 p_{\text {fail }}\|\Sigma\|_2-{C} \sigma^2 \sqrt[3]{\frac{n}{m}}-{C} \sigma^2 t\right)\|x - x_\star\|_2\|x + x_\star\|_2,
\end{align*}}holds for all $x \in \mathbb{R}^n$ with probability at least $1-2 e^{-c_3 m}-2 e^{-m t^2}$ for any $t>0$. \end{lemma}
\begin{remark}
\label[remark]{rem:lambda-s}
Lemma \ref{lemma_stat_sharpness} implies that the sharpness condition in~\eqref{ineq:sharpness}
holds with high probability when $\kappa_{\mathrm{st}}>0, m/n$ is large and $p_{\mathrm{fail}}$ is sufficiently small. In this situation, for small $t>0$, the coefficient of $\|x - x_\star\|_2\|x + x_\star\|_2$ is close to $\kappa_{\mathrm{st}} - 2p_{\mathrm{fail}}\|\Sigma\|_2$. Next, we show that $\|x - x_\star\|_2\|x + x_\star\|_2\geq \|x_\star\|_2\Delta(x)$. Without loss of generality, we assume that $\|x - x_\star\|_2 = \Delta(x)\leq \|x + x_\star\|$. In this case, $\|x_\star\|_2\leq \|x + x_\star\|_2/2 + \|x - x_\star\|_2/2\leq \|x + x_\star\|_2$. This shows the relationship. Thus, $\lambda_s$ is close to $\|x_\star\|_2(\kappa_{\mathrm{st}} - 2p_{\mathrm{fail}}\|\Sigma\|_2)$. More importantly, \eqref{ineq:sharpness} implies that $\{x_\star,-x_\star\} = \argmin_{x\in\mathbb{R}^n} F(x)$.
\end{remark}

Using the local linear approximation in \Cref{thm:gen_weak} and the sharpness condition, we can define a \textit{condition number} $\kappa_0 \triangleq L\|x_\star\|_2/(2\lambda_s)$ for $F$ defined in \eqref{duchi_l1_ori}; moreover, under \eqref{ineq:sharpness}, Lemma \ref{rel_lamL} shows that  $\kappa_0\geq 1$. Finally,
since $\kappa_{\mathrm{st}}\leq \inf_{u\in\mathbb{S}^{n-1}} \mathbb{E} (a^\top u)^2$, we argue that $L\|x_\star\|_2/(2\lambda_s)$ is related to the condition number of $\Sigma$ and $p_{\mathrm{fail}}$. Indeed, from \Cref{rem:lambda-s,rem:L}, for $\kappa_{\mathrm{st}}>0$, $m/n$ large and $p_{\mathrm{fail}}$ sufficiently small, 
\begin{equation}
\label{eq:kappa0}
{
\begin{aligned}
  \kappa_0 \triangleq L\|x_\star\|_2/(2\lambda_s)\approx \frac{1}{\kappa_{\mathrm{st}}/\|\Sigma\|_2 - 2p_{\mathrm{fail}}}.
  \end{aligned}}\end{equation}
Therefore, $\kappa_0$ represents how ill-conditioned the RPR problem is.

Assumption~\ref{ass_data}(c) is new to the RPR literature and is the key to guarantee that $t_k$ and $\alpha_k$ are proportional to $\Theta(\Delta(x^k))$. Indeed, consider  $\Tilde{p}\in(p_{\mathrm{fail}},~1-p_{\mathrm{fail}})$, which can always be set to $\tilde p=\frac{1}{2}$ as $\frac{1}{2}$ clearly belongs to this interval, and recall that we use $\tilde p$-th percentile of the residuals at $x^k$ to define $\alpha_k$ and $t_k$ as in \eqref{choice_alpha_subg} and \eqref{choice_t_diminish}, respectively, e.g., for $\tilde p=\frac{1}{2}$, both $\alpha_k$ and $t_k$ are set based on the median value of $\{r_i(x^k)\}_{i=1}^m$. Statistically, it means that for any $\|u\|_2=1$, the distribution of $\fprod{a,u}$ should not concentrate too close to $0\in\reals$.

\textbf{A sufficient condition for Assumption~\ref{ass_data}(c)}. Since $\mathbb{P}(\min\{|\fprod{a,u}|,~|\fprod{a,v}|\}\geq \kappa) \geq  \mathbb{P}(|\fprod{a,u}|\geq \kappa) + \mathbb{P}(|\fprod{a,v}|\geq \kappa) - 1$, a sufficient condition for \eqref{eq:p0} is 
    \begin{equation}\label{response_sufficient}
    {
    \begin{aligned}
        \inf_{u\in \mathbb{S}^{n-1}}\mathbb{P}(|\fprod{a,u}|\geq \kappa) > 1/2 + \frac{1-\tilde{p}}{2(1-p_{\mathrm{fail}})}.
    \end{aligned}}
    \end{equation}
    Note that $1/2 + \frac{1-\tilde{p}}{2(1-p_{\mathrm{fail}})}<1$ since $\tilde p > p_{\rm fail}$, e.g., $\tilde p=1/2$. Compared to \eqref{eq:p0}, the condition in \eqref{response_sufficient} is easier to verify as it only involves the marginal distribution.

\textbf{A class of distributions satisfying Assumption~\ref{ass_data}(c).} A broad class of continuous elliptical distributions satisfies \eqref{eq:p0}. Consider the density function of the random vector $a\in\reals^n$,
    $f_a(v) = c_a g(v^\top \Sigma^{-1} v)$ for $v\in\mathbb{R}^n$,
    where $\Sigma$ is symmetric positive definite. We require $\int_0^\infty t^{n/2}g(t)\,dt = n\int_0^\infty t^{n/2-1}g(t)\,dt<\infty$, so that $\operatorname{Cov}(a)=\Sigma$. Here, $c_a>0$ is the normalizing constant, and $g(\cdot)$ is a nonnegative Lebesgue-integrable function on $\mathbb{R}_+\bigcup \{0\}$ satisfying $0<\int_0^\infty t^{n/2-1}g(t)\,dt<\infty$. We further require that there exists $\epsilon>0$ such that $g(t)>0$ for $t\in (0,\epsilon)$. This class includes centered Gaussian distributions, i.e., $a\sim\mathcal{N}(\mathbf{0},\Sigma)$ with a symmetric positive definite $\Sigma$. Next, we prove that there exists $\kappa>0$ such that \eqref{response_sufficient} holds for any distribution with the properties stated above. For any $u\in\mathbb{S}^{n-1}$, letting $s_u = \sqrt{u^\top \Sigma u}$, the density function of $a_u = a^\top u$ is given, for $n\geq 2$, by
    $f_{a_u}(z) = \frac{1}{s_u}\hat{g}((z/s_u)^2)$, where $\hat{g}(t) = c_a|\Sigma|^{1/2}\frac{2\pi^{(n-1)/2}}{\Gamma\left((n-1)/2\right)}\int_0^\infty r^{n-2} g(t+r^2)\,dr$ for all $t\geq 0$.
    For $n=1$, the same density representation holds with $\hat{g}(t)=c_a|\Sigma|^{1/2}g(t)$. It is easy to verify that there exists $\epsilon'>0$ such that $\hat{g}(t)>0$ holds for $0<t\leq \epsilon'$. Furthermore, since $\mathbb{P}(|\fprod{a,u}|\geq \kappa) = \int_{|t|\geq \kappa/s_u} \hat{g}(t^2)dt$,
    we have
    $\inf_{u\in \mathbb{S}^{n-1}}\mathbb{P}(|\fprod{a,u}|\geq \kappa) = \int_{|t|\geq \kappa/\sqrt{\lambda_{\min}(\Sigma)}} \hat{g}(t^2)dt$,
    which further implies that
    $\lim_{\kappa\rightarrow 0+} \inf_{u\in \mathbb{S}^{n-1}}\mathbb{P}(|\fprod{a,u}|\geq \kappa) = 1$. Recalling that $1/2 + \frac{1-\tilde{p}}{2(1-p_{\mathrm{fail}})}<1$, we can conclude that there exists $\kappa>0$ such that \eqref{response_sufficient} holds, which further implies \eqref{eq:p0} holds.

Next, we provide an example that (a), (b), and (c) of Assumption \ref{ass_data} hold simultaneously. Let $\tilde p=\frac{1}{2}$ and $a\sim N(\mathbf{0}, I_n)$, which indicates that $\kappa_{\mathrm{st}} = 2/\pi$ (see \cite[Example 4]{duchi2019solving}) and $\Sigma = \mathcal{I}_n.$ Thus, when $0 < p_{\mathrm{fail}} \leq 1/(2\pi)$, we have that $\kappa_{\mathrm{st}} - 2p_{\mathrm{fail}}\|\Sigma\|_2\geq 1/\pi > 0$.
Moreover, for $p_0(w)\triangleq\inf_{u,v\in \mathbb{S}^{n-1}}\mathbb{P}(\min\{|\fprod{a,u}|,~|\fprod{a,v}|\}\geq w)$, it holds that $\lim_{w\rightarrow 0} p_0(w) = 1$; and we also have $\frac{1-\tilde{p}}{1-p_{\mathrm{fail}}} \leq 1/(2-1/\pi) < 1$. Therefore, we can conclude that $\kappa>0$ satisfying Assumption~\ref{ass_data}(c) exists.

Next, we argue that $r^{\Tilde{p}}(x)= \Theta(F(x) - F(x_\star))$ with high probability.
\begin{theorem}\label{thm_assumptions_high_prob}
    Suppose that Assumption \ref{ass_data} holds and $\Tilde{p}\in (p_{\mathrm{fail}},1-p_{\mathrm{fail}}),m\Tilde{p}\in\mathbb{N}$. There exist positive constants $u_L,u_H,\rho_1,\rho_2,\rho_3$, depending on $\Sigma, \sigma,\kappa_{\mathrm{st}}, \kappa, p_0, p_{\mathrm{fail}}, \tilde{p}$, such that when $m\geq \rho_1n$, the following statistical event, 
    
    \begin{equation}\label{condition_med_prop}
    {
    \begin{aligned}
    r^{\Tilde{p}}(x) \in \Big[u_L(F(x) - F(x_\star)),\  u_H(F(x) - F(x_\star))\Big],\quad \forall x\in\mathbb{R}^n, \end{aligned}}\end{equation}holds with probability at least $1-\rho_2\exp(-m\rho_3).$
\end{theorem}
The detailed proof of Theorem \ref{thm_assumptions_high_prob} is provided in Section \ref{sec:proof_stat}. 
Finally, using \eqref{ineq:sharpness} and \eqref{condition_med_prop} implies the following conclusions for \adasub{} and \adaipl{}.
\begin{corollary}
    \label[corollary]{prop_alpha_and_tk}
Under the statistical event \eqref{condition_med_prop}, which holds with high probability under Assumption \ref{ass_data}, the following conclusions hold.

(a) When $0< G < 2/u_H$, the \emph{\adasub{}} step size $\alpha_k$ in \eqref{choice_alpha_subg} satisfies
{\small
\begin{equation*}
\alpha_k\in \Big[c_1(F(x^k) - F(x_\star)),\ c_2(F(x^k) - F(x_\star))\Big],\quad \forall k\in\mathbb{N},
\end{equation*}}where $c_1 \triangleq u_LG$, $c_2 \triangleq u_HG$ and they satisfy $0<c_1\leq c_2<2$.

(b) Suppose that Assumption~\ref{ass:sharpness} also holds. For any $k\in\mathbb{N}$, if $\Delta(x^k)\leq \|x_\star\|_2,$ then there exists $g_k\in [g_L,g_H]$ such that the \emph{\adaipl{}} step size $t_k$ in \eqref{choice_t_diminish} satisfies $t_k = \min\{L^{-1},g_k\Delta(x^k)\}$,
where $g_L \triangleq G\lambda_s u_L$ and $g_H \triangleq 3G L\|x_\star\|_2 u_H/2$.
\end{corollary}
\begin{proof}(a) is a conclusion of the statistical event in \eqref{condition_med_prop}. For (b), since $t_k = \min\{L^{-1},G~ r^{\Tilde{p}}(x^k)\}$, we only need to prove that for any $x\in\mathbb{R}^n$ with $\Delta(x)\leq \|x_\star\|$, we have
    \begin{equation}\label{lower_upper_r_tilde0}
        (g_L/G) \Delta(x) \leq r^{\Tilde{p}}(x)\leq (g_H/G) \Delta(x).
    \end{equation}
   For the first inequality above, by the first inequality of the statistical event \eqref{condition_med_prop}, $r^{\Tilde{p}}(x) \geq u_L(F(x) - F(x_\star))$. Noticing that $F(x) - F(x_\star)\geq \lambda_s\Delta(x)$ by \eqref{ineq:sharpness}, we get the first inequality. For the second one above, the statistical event in \eqref{condition_med_prop} implies that $r^{\Tilde{p}}(x)\leq u_H(F(x) - F(x_\star))$.
    By Lemma \ref{lip_F1}, $|F(x) - F(x_\star)|\leq \frac{L}{2}\|x-x_\star\|_2\|x+x_\star\|_2$. In addition, we have that
    $\frac{L}{2}\|x-x_\star\|_2\|x+x_\star\|_2\leq \frac{L}{2}\Delta(x) (\Delta(x) + 2\|x_\star\|_2)\leq 3L\|x_\star\|_2\Delta(x)/2 = g_H\Delta(x)/(G u_H).$ Here, the first inequality holds because when one of $\|x-x_\star\|_2$ and $\|x+x_\star\|_2$ equals $\Delta(x)$, the other one can be bounded by $\Delta(x) + 2\|x_\star\|_2$ by the triangle inequality. The second inequality holds because $\Delta(x)\leq \|x_\star\|_2$.
    Thus, the second inequality of \eqref{lower_upper_r_tilde0} is also proved.
\end{proof}
Note that larger choices for $G$ lead to larger coefficients $c_1,c_2,g_L,g_H$.
\section{Proof of Theorem \ref{thm_assumptions_high_prob}}\label{sec:proof_stat}
Fix $x\in\reals^n$. First, we recall an upper bound on $F(x) - F(x_\star)$. According to Lemma~\ref{lip_F1}, one has $F(x) - F(x_\star)\leq \frac{L}{2}\|x-x_\star\|_2\|x+x_\star\|_2$.
Moreover, Lemma~\ref{sharpness_upper} shows that $L$ is close to $2\|\Sigma\|_2$ when $m/n$ is large, where $\Sigma$ is given in Assumption~\ref{ass_data} and $\|\Sigma\|_2 = \sup_{u\in \mathbb{S}^{n-1}}\mathbb{E} \fprod{a,u}^2$. Throughout the proof, we adopt the notation $$\tilde{\Delta}(x) {\triangleq} \|x-x_\star\|_2\|x+x_\star\|_2$$ for brevity.
Next, we begin the proof. It contains two main components.

\subsection{\texorpdfstring{$r^{\tilde{p}}(x) = {\Theta}(\tilde{\Delta}(x))$ with high probability}{r-quantile(x) = Theta(Delta-tilde(x)) with high probability}}
 We review the definitions in \eqref{choice_alpha_subg} and \eqref{choice_t_diminish} regarding our choice of $\alpha_k$ and $t_k$. Let $r_i(x) = |{\fprod{a_i,x}^2} - b_i|$ for all $i = 1,2\ldots,m$. \begin{definition}
     Given $x\in\reals^n$, for any $i\in [m]$, let $r_{(i)}(x)$ denote the $i$-th order statistic based on
     $\{r_i(x)\}_{i=1}^m$,
     which indicates that $r_{(1)}(x)\leq\ldots \leq r_{(m)}(x)$. Similarly, for any $i\in [m_1]$, let $\Tilde{r}_{(i)}(x)$ denote the $i-$th order statistic based on
     $\{r_i(x)\}_{i\in\mathcal{I}_1}$. Moreover, for any given percentile $\Tilde{p}\in (0,1)$ such that $m\Tilde{p}\in\mathbb{N}$, we define $r^{\tilde{p}}(x) {\triangleq} r_{(m\Tilde{p})}(x)$.
 \end{definition}\begin{remark}
    Recall that $p_{\mathrm{fail}}=(m-m_1)/m$ is the fraction of corrupted measurements, and by definition $m p_{\mathrm{fail}}=m-m_1\in\integers_+$. We assume that $p_{\mathrm{fail}}<\frac{1}{2}$ and $\tilde p\in (p_{\mathrm{fail}}, 1-p_{\mathrm{fail}})$. For the sake of simplicity of notation, we assume that  $m\tilde p\in\integers_+$; hence, $m(\tilde p-p_{\mathrm{fail}})\in\integers$, and note that $m \tilde p< m(1-p_{\mathrm{fail}})=m_1$. Thus, the order statistics $\tilde r_{(m(\tilde p-p_{\mathrm{fail}}))}$, $\tilde r_{(m\tilde p)}$ and $r_{(m\tilde p)}$ are all well defined.
 \end{remark}
 Next, we compare $r_{(m\tilde{p})}(x)$ to the mean and quantiles of the noiseless samples.
\begin{lemma}\label[lemma]{bound_quantile}
If $p_{\mathrm{fail}}<\Tilde{p}<1-p_{\mathrm{fail}}$, then for any $x\in \mathbb{R}^n$, $\Tilde{r}_{\left(m(\Tilde{p} - p_{\mathrm{fail}})\right)}(x)\leq r_{(m\Tilde{p})}(x)\leq \Tilde{r}_{\left(m\Tilde{p}\right)}(x)$; moreover,
{\small
$$\Tilde{r}_{\left(m\Tilde{p}\right)}(x)\leq \frac{1}{m(1 - p_{\mathrm{fail}} - \Tilde{p})}\sum_{i=m\Tilde{p}+1}^{m_1} \Tilde{r}_{(i)}(x)\leq \frac{1-p_{\mathrm{fail}}}{(1 - p_{\mathrm{fail}} - \Tilde{p})}\frac{1}{m_1}\sum_{i{\in\mathcal{I}_1}} r_i(x).$$}\end{lemma}
\begin{proof}We denote $A_1 \triangleq \{i\in [m]:r_i(x) \leq r_{(m\tilde{p})}(x)\},A_2 \triangleq \{i\in {\cI_1}:r_i(x) \leq r_{(m\tilde{p})}(x)\}$, $A_3 \triangleq \{i\in [m]:r_i(x) \leq \tilde{r}_{(m\tilde{p})}(x)\}$ and $A_4 \triangleq \{i\in {\mathcal{I}_1}:r_i(x) \leq \tilde{r}_{(m\tilde{p})}(x)\}$.
From the definition of order statistics, we know that $|A_1|\geq m\tilde{p}$---
it might happen that $r_i(x) = r_j(x)$ for some $i,j\in [m]$ such that $i\neq j$; hence, it is possible that $|A_1|> m\tilde{p}$. Furthermore, note that
$A_2=A_1\backslash {\mathcal{I}_2}$.
Thus, $|A_2| = |A_1\backslash {\mathcal{I}_2}|\geq |A_1| - {m_2}\geq m(\tilde{p} - p_{\mathrm{fail}})$,
which indicates the first inequality. We can find that $A_4\subseteq A_3$, which indicates that $|A_3|\geq |A_4|\geq m\tilde{p}$. Thus, the second inequality holds.
The third inequality holds because $\Tilde{r}_{(i)}(x)\geq \tilde{r}_{(m\Tilde{p})}(x)$ when $i\geq m\Tilde{p}$, and we also have $m(1-p_{\mathrm{fail}} - \tilde{p}) = m_1 - m\tilde{p}$. The fourth one holds because
$\sum_{i{\in\mathcal{I}_1}} r_i(x)\geq \sum_{i=m\Tilde{p}+1}^{m_1} \Tilde{r}_{(i)}(x)$
and $\frac{1-p_{\mathrm{fail}}}{(1 - p_{\mathrm{fail}} - \Tilde{p})}\frac{1}{m_1} = \frac{1}{m(1 - p_{\mathrm{fail}} - \Tilde{p})}$.
\end{proof}

Lemma \ref{bound_quantile} implies that 
\begin{equation}\label{eq_proof_stat_lu_summary}
{
\begin{aligned}
    \Tilde{r}_{\left(m(\Tilde{p} - p_{\mathrm{fail}})\right)}(x)\leq r_{(m\Tilde{p})}(x)\leq \frac{1-p_{\mathrm{fail}}}{(1 - p_{\mathrm{fail}} - \Tilde{p})}\frac{1}{m_1}\sum_{i{\in\mathcal{I}_1}}r_i(x),
\end{aligned}}\end{equation}
where both bounds are quantities that are only related to the noiseless measurements.
Next, we argue that
$\frac{1}{m_1}\sum_{i\in{\mathcal{I}_1}} r_i(x)\leq {\cO}(\tilde{\Delta}(x))$
for all $x\in\reals^n$ with high probability.
\begin{lemma}\label[lemma]{lemma_upper_bound_mean_r_noiseless}
    Suppose that Assumption \ref{ass_data} holds
for $\Tilde{p}\in (0,1)$ such that $m\Tilde{p}\in\mathbb{N}$ and $p_{\mathrm{fail}}<\Tilde{p}<1-p_{\mathrm{fail}}$. Then whenever $m\geq {8}n/(u_1(1-p_{\mathrm{fail}}))$, we have
\begin{equation}\label{proof_stat_event1}
{
\begin{aligned}
\mathbb{P}\left(\frac{1}{m_1}\sum_{i{\in\mathcal{I}_1}} r_i(x)\leq \frac{3}{2}\|\Sigma\|_2\tilde{\Delta}(x),\quad\forall~ x\in\mathbb{R}^n\right)\geq 1-\exp\Big(-\frac{mu_1}{2}(1-p_{\mathrm{fail}})\Big)
\end{aligned}}\end{equation}
where $u_1 = \min\{1,~\|\Sigma\|_2/(22\sigma^2)\}\|\Sigma\|_2/(22\sigma^2)$.
\end{lemma}
\begin{proof}
    For $m_2=0$, i.e., $m=m_1$, we have $F(x_\star) = 0$ and Lemma~\ref{lip_F1} implies
    $\frac{1}{m_1}\sum_{i{\in\mathcal{I}_1}} r_i(x) = F(x) - F(x_\star)  \leq \|\frac{1}{m_1}\sum_{i{\in\mathcal{I}_1}} a_ia_i^\top\|_2\tilde{\Delta}(x)$, for all $x\in\mathbb{R}^n$.
    In this setting, letting $m_1\geq {8}n/u_1$ and $t = u_1/2$ in Lemma \ref{sharpness_upper}, we have that
    $\|\frac{1}{m_1}\sum_{i{\in\mathcal{I}_1}} a_ia_i^\top - \Sigma\|_2\leq \frac{1}{2} \|\Sigma\|_2$
    holds with probability at least $1-\exp(-m_1u_1/2)$. Thus, under this event, we have
    $\frac{1}{m_1}\sum_{i{\in\mathcal{I}_1}} r_i(x) \leq \frac{3}{2}\|\Sigma\|_2\tilde{\Delta}(x)$ for all $x\in\mathbb{R}^n$.
    Next, for $m_2>0$, using $m_1 = (1-p_{\mathrm{fail}})m$ within the above inequality completes the proof.
\end{proof}

Next, we argue that
    $\Tilde{r}_{\left(m(\Tilde{p} - p_{\mathrm{fail}})\right)}(x) \geq \Omega(\tilde{\Delta}(x))$ holds for all $x\in\reals^n$
with high probability.
\begin{lemma}\label[lemma]{lemma-med_error-lower}
Suppose that Assumption \ref{ass_data} holds
for $\Tilde{p}\in (0,1)$ such that $m\Tilde{p}\in\mathbb{N}$ and $p_{\mathrm{fail}}<\Tilde{p}<1-p_{\mathrm{fail}}$. There exists $c_0>0$ such that for all $x\in\mathbb{R}^n$, $\Tilde{r}_{\left(m(\Tilde{p} - p_{\mathrm{fail}})\right)}(x)\geq \kappa^2\tilde{\Delta}(x)$ holds w.p.~at least
 $1 - 2\exp\left(-\frac{m(1-p_{\mathrm{fail}})(p_0 - (1-\tilde{p})/(1-p_{\mathrm{fail}}))^2}{32}\right)$ whenever $m\geq \frac{c_0^2 n}{1-p_{\mathrm{fail}}}\left(\frac{p_0 - (1-\tilde{p})/(1-p_{\mathrm{fail}})}{4}\right)^{-2}$.
\end{lemma}

\begin{proof}
 When $\Delta(x) = 0$, the conclusion obviously holds; otherwise, let $u = (x-x_\star)/\|x-x_\star\|_2$, $v = (x+x_\star)/\|x+x_\star\|_2$, and for all
 $i\in {\mathcal{I}_1}$,
 we define
 \begin{equation*}
    h_i(u,v) = \mathbf{1}[|\fprod{a_i,u}|\geq \kappa]\ \mathbf{1}[|\fprod{a_i,v}|\geq \kappa].
 \end{equation*}The indicators defining $h_i(u,v)$ form a VC class of
dimension at most $Cn$ for a numerical constant $C>0$,
since their defining sets are intersections of two unions
of two halfspaces.
Combining the VC empirical-process bound and the bounded
differences inequality
\cite[Theorems~8.3.23 and~2.9.1]{vershynin2018high},
we obtain that there exists a numerical constant $c_0>0$
such that for every $\kappa>0$ and any $t\geq 0$,
\begin{equation}\label{eq_duchi_VC_concentration}
{
\begin{aligned}
    \mathbb{P}\left(\sup _{u, v \in \mathbb{S}^{n-1}} \left|\frac{1}{m_1} \sum_{i\in\mathcal{I}_1} h_i(u, v)-\mathbb{E}\left[h_i(u, v)\right]\right| \geq c_0 \sqrt{\frac{n}{m_1}}+t\right) \leq 2 \exp \left(-\frac{m_1 t^2}{2}\right).
\end{aligned}}\end{equation}
 The value of $c_0$ is independent of any problem parameters and data.
 In addition, for all
 $i\in {\mathcal{I}_1}$,
\begin{equation}\label{eq_proof_stat_lower_indicator}
     r_i(x) = |\fprod{a_i,u}|\ |\fprod{a_i,v}|\ \tilde{\Delta}(x)\geq \kappa^2 h_i(u,v)\tilde{\Delta}(x).
 \end{equation}
Next, we consider the event
$\frac{1}{m_1}\sum_{i{\in\mathcal{I}_1}}h_i(u,v)>\frac{1-\tilde{p}}{1-p_{\mathrm{fail}}}$.
Knowing that $h_i(u,v)$ takes values in $\{0,1\}$, the event indicates that the number of zeros in
$\{h_i(u,v)\}_{i{\in\mathcal{I}_1}}$
is strictly less than $m_1 - m_1(1-\tilde{p})/(1-p_{\mathrm{fail}}) = m(\tilde{p} - p_{\mathrm{fail}})$. Together with \eqref{eq_proof_stat_lower_indicator}, this event indicates that $\Tilde{r}_{\left(m(\Tilde{p} - p_{\mathrm{fail}})\right)}(x)\geq\kappa^2 \tilde{\Delta}(x)$. To conclude,
\begin{equation}\label{eq_proof_stat_lower_indicator2}
\begin{aligned}
    \Tilde{r}_{\left(m(\Tilde{p} - p_{\mathrm{fail}})\right)}(x)\geq \kappa^2 \tilde{\Delta}(x)~\mathbf{1}\left[\frac{1}{m_1}\sum_{i{\in\mathcal{I}_1}}h_i(u,v)>\frac{1-\tilde{p}}{1-p_{\mathrm{fail}}}\right].
\end{aligned}
\end{equation}
 Based on (c) in Assumption \ref{ass_data}, we have
 $\mathbb{E}h_i(u,v)\geq p_0>\frac{1-\tilde{p}}{1-p_{\mathrm{fail}}}$ for all $u,v\in\mathbb{S}^{n-1}$ and $i{\in\mathcal{I}_1}$
 Together with \eqref{eq_duchi_VC_concentration}, if $m\geq \frac{c_0^2 n}{1-p_{\mathrm{fail}}}\left(\frac{p_0 - (1-\tilde{p})/(1-p_{\mathrm{fail}})}{4}\right)^{-2}$ and $t = \frac{p_0 - (1-\tilde{p})/(1-p_{\mathrm{fail}})}{4}$, then w.p. at least
 $1 - 2\exp\left(-\frac{m(1-p_{\mathrm{fail}})(p_0 - (1-\tilde{p})/(1-p_{\mathrm{fail}}))^2}{32}\right)$, for all $u,v\in\mathbb{S}^{n-1}$ we have
{\small
 \begin{equation*}
     \frac{1}{m_1}\sum_{i{\in\mathcal{I}_1}}h_i(u,v)\geq \mathbb{E}\left[h_i(u, v)\right] - c_0 \sqrt{\frac{n}{m_1}}-t\geq \frac{p_0+\frac{1-\tilde{p}}{1-p_{\mathrm{fail}}}}{2}>\frac{1-\tilde{p}}{1-p_{\mathrm{fail}}}.
\end{equation*}}Together with \eqref{eq_proof_stat_lower_indicator2}, this indicates that
 $\Tilde{r}_{\left(m(\Tilde{p} - p_{\mathrm{fail}})\right)}(x)\geq \kappa^2\tilde{\Delta}(x)$ for all $x\in\mathbb{R}^n$.
\end{proof}

Combining the event in Lemma \ref{lemma_upper_bound_mean_r_noiseless} and the event in Lemma \ref{lemma-med_error-lower}, applying \eqref{eq_proof_stat_lu_summary}, we can derive that, when $m\geq u_2 n$ where $u_2 \triangleq \frac{8}{1-p_{\mathrm{fail}}}\max\Big\{\frac{1}{u_1}, 2c_0^2\left(p_0 - \frac{1-\tilde{p}}{1-p_{\mathrm{fail}}}\right)^{-2}\Big\}$, w.p. at least
$1-3\exp(-u_3m)$ where $u_3 \triangleq \frac{1-p_{\mathrm{fail}}}{2}\min\Big\{\frac{1}{16}\left(p_0 - \frac{1-\tilde{p}}{1-p_{\mathrm{fail}}}\right)^2,~u_1\Big\}$, the following event holds:
\begin{equation}\label{proof_stat_event2}
    u_4\tilde{\Delta}(x)\leq r_{(m\Tilde{p})}(x)\leq u_5\tilde{\Delta}(x),\quad \forall~x\in\mathbb{R}^n,
\end{equation}
where $u_4 \triangleq \kappa^2$ and $u_5 \triangleq \frac{3}{2}\|\Sigma\|_2\Big(1-\frac{\Tilde{p}}{1 - p_{\mathrm{fail}}}\Big)^{-1}$.

\subsection{\texorpdfstring{$F(x) - F(x_\star) = {\Theta}(\tilde{\Delta}(x))$ with high probability}{F(x) - F(x-star) = Theta(Delta-tilde(x)) with high probability}}
First, we provide an upper bound for $F(x) - F(x_\star)$.
Letting $m\geq 8n/u_1$ and $t = u_1/2$ in Lemma \ref{sharpness_upper}, we have that $\|\frac{1}{m}\sum_{i=1}^{m} a_ia_i^\top - \Sigma\|_2\leq \frac{1}{2} \|\Sigma\|_2$ holds with probability at least $1-\exp(-mu_1/2)$. Thus, under this event, together with Lemma \ref{lip_F1}, and noticing that $L = 2\left\|\frac{1}{m}\sum_{i=1}^ma_ia_i^\top\right\|_2$,
we have $F(x)  - F(x_\star) \leq \frac{3}{2}\|\Sigma\|_2\tilde{\Delta}(x)$.

Next, we discuss a lower bound for $F(x) - F(x_\star)$.
From Lemma \ref{lemma_stat_sharpness}, for $m\geq \left(\frac{\kappa_{\mathrm{st}} - 2p_{\mathrm{fail}}\|\Sigma\|_2}{4{C}\sigma^2}\right)^{-3} n$ and $t = \frac{\kappa_{\mathrm{st}} - 2p_{\mathrm{fail}}\|\Sigma\|_2}{4{C}\sigma^2}$, it holds with probability at least $1-4\exp\Big(-m\min\{c_3, \big(\frac{\kappa_{\mathrm{st}} - 2p_{\mathrm{fail}}\|\Sigma\|_2}{4{C}\sigma^2}\big)^2\}\Big)$ that
$F(x) - F(x_\star)\geq \frac{\kappa_{\mathrm{st}} - 2p_{\mathrm{fail}}\|\Sigma\|_2}{2}\tilde{\Delta}(x)$ for all $x\in\reals^n$.
Thus, for $m\geq u_6n$ with
{\small
$$u_6 \triangleq \max\Big\{\frac{8}{u_1},~\left(\frac{4C\sigma^2}{\kappa_{\mathrm{st}} - 2p_{\mathrm{fail}}\|\Sigma\|_2}\right)^{3}\Big\},$$}it holds w.p. at least $1-5\exp(-m u_7)$ that
\begin{equation}\label{proof_stat_event3}
   u_8\tilde{\Delta}(x) \leq F(x) - F(x_\star)\leq u_9\tilde{\Delta}(x),\quad \forall x\in\mathbb{R}^n,
\end{equation}
where $u_7 \triangleq \min\Big\{c_3, \Big(\frac{\kappa_{\mathrm{st}} - 2p_{\mathrm{fail}}\|\Sigma\|_2}{4C\sigma^2}\Big)^2,\frac{u_1}{2}\Big\}$, $u_8 \triangleq \frac{\kappa_{\mathrm{st}} - 2p_{\mathrm{fail}}\|\Sigma\|_2}{2}$ and $u_9 \triangleq \frac{3}{2}\|\Sigma\|_2$. Finally, we finish the proof of Theorem \ref{thm_assumptions_high_prob} by combining \eqref{proof_stat_event2} and \eqref{proof_stat_event3}; indeed, whenever $m\geq \rho_1 n$ for $\rho_1 = \max\{u_2, u_6\}$, \eqref{condition_med_prop} holds for $u_L = u_4/u_9$ and $u_H = u_5/u_8$ w.p. at least $1-\rho_2\exp(-m\rho_3)$ with $\rho_2 = 8$ and $\rho_3 = \min\{u_3,u_7\}$.

\section{Convergence Rate of \adasub{}}\label{sec:adasub}

In this section,
we
establish the local linear convergence of \adasub{} (Algorithm \ref{alg:adaptive-sub}).
\begin{theorem}[Convergence Rate of \adasub]\label{thm:convergence_ada_subgrad}
Suppose that Assumption \ref{ass:sharpness} holds. Under the statistical event in \eqref{condition_med_prop}, which holds with high probability under Assumption \ref{ass_data}, if $x^0$ satisfies $\Delta(x^0)\leq \lambda_s(1-\frac{c_2}{2})/L$, then
the \emph{\adasub{}} sequence $\{x^k\}_{k=0}^\infty$ satisfies
\begin{equation}
\label{eq:adasubgrad-rate}
{
\begin{aligned}
\Delta(x^k)\leq \left(\sqrt{1-\frac{2\lambda_s^2(1-c^2)}{9L^2\|x_\star\|_2^2}}\right)^k\Delta(x^0){\triangleq R_k},\quad \forall k\in\mathbb{N},
\end{aligned}}\end{equation}
where $c \triangleq \max\{|1-c_1|,|1-c_2|\}$, and $c_1,c_2$ are constants in \Cref{prop_alpha_and_tk}.
\end{theorem}
\begin{proof}Define $e_k \triangleq \alpha_k/(F(x^k) - F(x_\star))$ for $k\geq 0$. From \Cref{prop_alpha_and_tk}, $e_k\in [c_1,c_2]\subseteq (0,2)$ and $e_k(2-e_k)\geq 1-c^2$ for $k\in\mathbb{N}$. Note that $1-\frac{2\lambda_s^2(1-c^2)}{9L^2\|x_\star\|_2^2}\in (0,1)$ from Lemma~\ref{rel_lamL}. Let $\hat{x}\in\{x_\star,-x_\star\}$ such that $\Delta(x^0) = \|x^0 - \hat{x}\|_2$. In the rest we establish \eqref{eq:adasubgrad-rate} by showing $\norm{x^k-\hat{x}}_2\leq R_k$ for $k\geq 0$ using induction. Note that the base case of the induction for $k=0$ trivially holds, i.e., $\|x^0 - \hat{x}\|_2\leq R_0=\Delta(x^0)$. Next, we assume that $\norm{x^k-\hat{x}}_2\leq R_k$ holds for some $k\in\integers_+$, and we show that it also holds for $k+1$.  From the update rule in \eqref{iter_adasub}, we have
 \begin{equation}
     \label{proof_subgrad_eq1}
     {
	\begin{aligned}
		\left\|x^{k+1}-\hat{x}\right\|_2^{2} &=\left\|x^{k}-\hat{x}\right\|_2^{2}+2\left\langle x^{k}-\hat{x}, x^{k+1}-x^{k}\right\rangle+\left\|x^{k+1}-x^{k}\right\|_2^{2} \\
		&=\left\|x^{k}-\hat{x}\right\|_2^{2}+\frac{2e_k\left(F(x^k)-F(\hat{x})\right)}{\left\|\xi^{k}\right\|_2^{2}} \cdot\left\langle\xi^{k}, \hat{x}-x^{k}\right\rangle+e_k^2\frac{\left(F(x^k)-F(\hat{x})\right)^{2}}{\left\|\xi^{k}\right\|_2^{2}}.
	\end{aligned}}\end{equation}
\eqref{eq_weakly_convex_useful} with $x=\hat{x},\bar x = x^k,v=\xi^k$ implies $\left\langle\xi^{k}, \hat{x}-x^{k}\right\rangle \leq F(\hat{x})-F(x^k)+\frac{L}{2}\left\|x^{k}-\hat{x}\right\|^{2}.$
	Applying it to \eqref{proof_subgrad_eq1}, we can further find that
 \begin{equation}\label{proof_subgradient_eq2}
 {
 \begin{aligned}
     \|x^{k+1}-\hat{x}\|_2^{2} \leq\|x^{k}-\hat{x}\|_2^{2}+\frac{e_k\left(F(x^k)-F(\hat{x})\right)}{\left\|\xi^{k}\right\|_2^{2}}\left(L\left\|x^{k}-\hat{x}\right\|_2^{2}-(2-e_k)\left(F(x^k)-F(\hat{x})\right)\right).
 \end{aligned}}\end{equation}
Next, we discuss three quantities appearing in \eqref{proof_subgradient_eq2}.
 \paragraph{Term 1: $\|x^{k}-\hat{x}\|_2^{2}$} By the induction hypothesis, we have $\norm{x^k-\hat{x}}_2\leq R_k
  \leq \Delta(x^0)$. Thus, ${\norm{x^k-\hat{x}}_2}\leq \Delta(x^0)\leq \lambda_s(2-c_2)/(2L)\leq \lambda_s/L$, where the second inequality follows from the hypothesis and the last one from $c_2>0$. Moreover, the last inequality and Lemma \ref{rel_lamL} imply that $\|x^k-\hat{x}\|_2\leq \|x_\star\|_2/2$. Thus, $\|x^k+\hat{x}\|_2\geq 2\|\hat{x}\|_2 - \|x^k-\hat{x}\|_2 = 2\|x_\star\|_2- \|x^k-\hat{x}\|_2\geq 3\|x_\star\|_2/2\geq \|x^k-\hat{x}\|_2$, which means that $x^k$ is closer to $\hat{x}$ compared to $-\hat{x}$ so that $\Delta(x^k) = \|x^k-\hat{x}\|_2$. To conclude, \begin{equation}\label{eq_upper_delta_xk}
     \Delta(x^k) = \|x^k-\hat{x}\|_2\leq \Delta(x^0)\leq \|x_\star\|_2/2.
 \end{equation}

 \paragraph{Term 2: $\|\xi^k\|$} For any $\epsilon>0$ and arbitrary $\xi^k\in \partial F(x^k)$ such that $\xi^k\neq \mathbf{0}$, since $x^k$ is an interior point of $S_{\epsilon}^k\triangleq \{x\in\mathbb{R}^n:\Delta(x)\leq \Delta(x^k) + \epsilon\}$, there exists $t_0>0$ such that $x^k + t\xi^k\in S_{\epsilon}^k$ for $t\in [0,t_0]$. In \eqref{eq_weakly_convex_useful}, setting $x = \tilde{x}^k \triangleq x^k + t\xi^k$ for some $t\in (0,t_0]$, $\bar x = x^k$ and $v=\xi^k$, we get
 $F(\tilde{x}^k) - F(x^k) + Lt^2\|\xi^k\|_2^2/2\geq t\|\xi^k\|_2^2.$
 Thus, noticing that $\|\tilde{x}^k-x^k\|_2 = t\|\xi^k\|_2$,
 we can obtain $\|\xi^k\|_2\leq \frac{F(\tilde{x}^k) - F(x^k)}{\|\tilde{x}^k-x^k\|_2} + Lt\|\xi^k\|_2/2$. Moreover, $\frac{F(\tilde{x}^k) - F(x^k)}{\|\tilde{x}^k-x^k\|_2}\leq \sup\Big\{ \frac{|F(x) - F(y)|}{\left\| x-y\right\|_2}:\ x,y\in {S_{\epsilon}^k},~x\neq y\Big\}$ since $\tilde{x}^k,x^k\in S_{\epsilon}^k$. Therefore, letting $t\rightarrow 0$,
 we have that 
 {\small
 \begin{equation*}
     \|\xi^k\|_2\leq \sup_{x,y\in {S_{\epsilon}^k},~x\neq y} \frac{|F(x) - F(y)|}{\left\| x-y\right\|_2}{\triangleq B_\epsilon}.
\end{equation*}}By Lemma \ref{Lip_F} and \eqref{eq_upper_delta_xk},
 we have that
 \begin{equation}\label{eq_subgrad_upper}
     \|\xi^k\|_2\leq {\lim_{\epsilon\to 0}B_\epsilon\leq}\lim_{r\rightarrow \Delta(x^k)} L(\|x_\star\|_2+r) = L(\|x_\star\|_2+\Delta(x^k))\leq \frac{3}{2}L\|x_\star\|_2.
 \end{equation}
Since \eqref{eq_subgrad_upper} trivially holds when $\xi^k = \mathbf{0}$, we have $\|\xi^k\|_2\leq \frac{3}{2}L\|x_\star\|_2$ for all $\xi_k\in \partial F(x^k)$.
\paragraph{Term 3: $L\left\|x^{k}-\hat{x}\right\|_2^{2}-(2-e_k)\left(F(x^k)-F(\hat{x})\right)$} Due to \eqref{eq_upper_delta_xk} and \eqref{ineq:sharpness} and $e_k\in (0,2)$, we have $(2-e_k)(F(x^k)-F(\hat{x})) \geq \lambda_s(2-e_k)\|x^k - \hat{x}\|_2$. Thus,
\begin{equation*}
    L\left\|x^{k}-\hat{x}\right\|_2^{2}-(2-e_k)\left(F(x^k)-F(\hat{x})\right)\leq \|x^k - \hat{x}\|_2\Big(L\|x^{k}-\hat{x}\|_2-\lambda_s(2-e_k)\Big).
\end{equation*}
Recall that by the hypothesis, $x^0$ satisfies $\Delta(x^0)\leq \lambda_s(1-\frac{c_2}{2})/L$; therefore, it follows from \eqref{eq_upper_delta_xk} and $e_k\in [c_1,c_2]\subset (0,2)$ that $L\|x^k-\hat{x}\|_2\leq L\Delta(x^0)\leq \lambda_s(2-c_2)/2\leq \lambda_s(2-e_k)/2$. Thus, we get $\|x^k - \hat{x}\|_2(L\|x^{k}-\hat{x}\|_2-\lambda_s(2-e_k))\leq {-}\lambda_s(2-e_k)\|x^k - \hat{x}\|_2/2$. To conclude, we have
\begin{equation}\label{eq_compare_first_second}
    L\left\|x^{k}-\hat{x}\right\|_2^{2}-(2-e_k)\left(F(x^k)-F(\hat{x})\right)\leq -\lambda_s(2-e_k)\|x^k-\hat{x}\|_2/2.
\end{equation}
Thus, using \eqref{eq_subgrad_upper} and \eqref{eq_compare_first_second} within \eqref{proof_subgradient_eq2}, we have
 \begin{equation*}
     \|x^{k+1} - \hat{x}\|_2^2\leq \|x^{k} - \hat{x}\|_2^2 - \frac{2\lambda_se_k(2-e_k)}{9L^2\|x_\star\|_2^2}(F(x^k) - F(\hat{x}))\|x^k - \hat{x}\|_2.
 \end{equation*}
 From \eqref{ineq:sharpness}, we have $F(x^k) - F(\hat{x})\geq \lambda_s\|x^k - \hat{x}\|_2$, which combined with the above inequality implies
 \begin{equation}\label{proof_subgrad_eq0}
 {
 \begin{aligned}
    \|x^{k+1} - \hat{x}\|_2\leq \sqrt{1-\frac{2\lambda_s^2e_k(2-e_k)}{9L^2\|x_\star\|_2^2}}\|x^k - \hat{x}\|_2,\quad \forall k\in\mathbb{N}.
\end{aligned}}\end{equation}
Using $e_k(2-e_k)\geq 1-c^2$ and the inductive hypothesis, i.e., $\norm{x^k-\hat{x}}_2\leq R_k$, within this inequality
 implies that $\|x^{k+1} - \hat{x}\|_2\leq R_{k+1}$ completing the induction, which also implies \eqref{eq:adasubgrad-rate} since $\Delta(x^k)\leq \norm{x^k-\hat{x}}_2$ for all $k\geq 0$.
\end{proof}
\begin{remark}\label[remark]{remark_init_subgradiant}
    The initialization requirement $\Delta(x^0)\leq \lambda_s(1-c_2/2)/L$ in Theorem \ref{thm:convergence_ada_subgrad} can be satisfied by the output of \cite[Algorithm 3]{duchi2019solving} with high probability when $m/n$ is large enough and $p_{\mathrm{fail}}$ is small enough. In fact, under some regularity assumptions, \cite[Theorem 3]{duchi2019solving} claims that for any $C_{\mathrm{init}}>0$, $x^0$ returned by Algorithm 3 in \cite{duchi2019solving} satisfies $\Delta(x^0)\leq C_{\mathrm{init}}$ when $p_{\mathrm{fail}} < 1/4$ and $m/n\geq \Omega((\|x_\star\|_2/C_{\mathrm{init}})^2)$.
\end{remark}
Theorem \ref{thm:convergence_ada_subgrad} shows that Algorithm \ref{alg:adaptive-sub} takes $\cO(\frac{\kappa_0^2}{1-c^2}\log \frac{1}{\epsilon})$ iterations to find an $\epsilon$-optimal solution for any $\epsilon\in(0,\Delta(x^0))$. According to \Cref{prop_alpha_and_tk}(a), this complexity result is guaranteed to hold with high probability when the parameter $G>0$ appearing in the step-size choice in~\eqref{choice_alpha_subg} is sufficiently small such that $G<2/u_H=\frac{2}{3}(\kappa_{\mathrm{st}}/\|\Sigma\|_2 - 2p_{\mathrm{fail}})(1-\frac{\tilde p}{1-p_{\rm fail}})\approx\frac{2}{3}\kappa_0^{-1}(1-\frac{\tilde p}{1-p_{\rm fail}})$ when $m/n$ is large and $p_{\mathrm{fail}}$ is sufficiently small---see \eqref{eq:kappa0}.
The factor $1-c^2$ represents how well our adaptive step size $\alpha_k$ in \eqref{choice_alpha_subg} approximates $F(x^k) - F(x_\star)$, i.e., $1-c^2$ measures how close \adasub{} update rule in \eqref{iter_adasub} mimics
\psub{}; indeed, in the ideal situation where $c=0$, which might be impossible in practice, \psub{} and \adasub{} coincide. When $G$ is overly small such that $c_2 = u_H G\leq 1$, we have $c=1-c_1$ since $0<c_1\leq c_2<1$; hence, $1-c^2 = c_1(2-c_1)\geq c_1$ and the number of required iterations is $\cO(\frac{\kappa_0^2}{c_1}\log \frac{1}{\epsilon})$. Since $c_1 = u_LG$, the complexity bound in terms of $G$ and $\epsilon$ becomes $\cO(\frac{1}{G}\log \frac{1}{\epsilon})$.

\section{Convergence Rate of \adaipl{}}\label{sec:conv_adaipl}
In this section, we analyze the convergence properties of \adaipl{} stated in Algorithm~\ref{alg:adaptive-IPL}.
\subsection{Convergence Behaviours for Outer Iterations} In this subsection, we analyze a prototype of \adaipl{} where we do not require $t_k$ to be selected based on \eqref{choice_t_diminish}; instead, we establish convergence of \adaipl{} for $\{t_k\}$ satisfying a more general condition.
More precisely, in the result below, we provide a one-step analysis of \adaipl{} when $t_k\in (0,L^{-1}]$ and $\Delta(x^k)$ is small.
\begin{lemma}\label[lemma]{lemma_traj_gen_tk}
    Under Assumption~\ref{ass:sharpness}, if $x^k$ satisfies $\Delta(x^k)\leq \lambda_s/(4L)$ and $t_k\in (0,L^{-1}]$ for some $k\geq 0$, then the following conclusions hold for Algorithm \ref{alg:adaptive-IPL}.
    \begin{itemize}
        \item[(a)] When \textbf{(LAC)} in \eqref{eq:low-high-practical} holds with $\rho_l\geq 0$, we have that
        \begin{equation*}
            F(x^{k+1}) - F(x_\star)\leq \left(1 - \frac{5\min\{\lambda_st_k/(2\Delta(x^k)),1\}}{8(1+\rho_l)}\right)(F(x^k) - F(x_\star)).
        \end{equation*}
        \item[(b)] When \textbf{(HAC)} in \eqref{eq:low-high-practical} holds with $\rho_h\in [0,1/4)$, we have that
        \begin{equation*}
            F(x^{k+1}) - F(x_\star)\leq \left(1 - \frac{5\min\{\lambda_st_k/(2\Delta(x^k)),1\}}{8(1+\frac{2\rho_h}{1-4\rho_h})}\right)(F(x^k) - F(x_\star)).
        \end{equation*}\end{itemize}
\end{lemma}

Lemma \ref{lemma_traj_gen_tk} indicates that in \adaipl{} with $\cond$ stated in \Cref{alg:adaptive-IPL} set to either (\textbf{LAC}) or (\textbf{HAC})  in \eqref{eq:low-high-practical}, if $\Delta(x^0)$ is small enough, the number of main iterations required for computing an $\epsilon-$optimal point is
$\cO(\log(1/\epsilon))$. Specifically, consider the $t_k$ selection rule in \eqref{choice_t_diminish},
when $G>0$ is chosen sufficiently \textit{small} so that $g_H\triangleq 3 G L\norm{x_\star}_2u_H/2\leq 2/\lambda_s$, according to part (b) of \Cref{prop_alpha_and_tk} we have $t_k=g_k\Delta(x^k)\leq g_H \Delta(x^k)$ for all $k\geq 0$ with high probability. Note that $\lambda_s t_k/(2\Delta(x^k))=\lambda_s g_k/2\leq \lambda_s g_H/2\leq 1$; therefore, the contraction rate is smaller than $1-\frac{5\lambda_sg_L}{16}(1+\rho_l)^{-1}$ and $1-\frac{5\lambda_sg_L}{16}(1+\frac{2\rho_h}{1-4\rho_h})^{-1}$ for \lac{} and \hac{}, respectively, where we used the fact that $g_k\geq g_L$ for all $k\geq 0$ with high probability (see part (b) of \Cref{prop_alpha_and_tk}). Hence, the number of main iterations for both \lac{} and \hac{} is $\cO(\frac{1}{\lambda_s g_L}\log \frac{1}{\epsilon})$; and in terms of $\epsilon$ and $G$, we get $\cO(\frac{1}{G}\log(\frac{1}{\epsilon}))$.
On the other hand, when $G>0$ is chosen sufficiently \textit{large} so that $g_L \triangleq G\lambda_s u_L> 2/\lambda_s$, the number of main iterations for both \lac{} and \hac{} will become $\cO(\log \frac{1}{\epsilon})$ without $G$ explicitly appearing in the denominator. Indeed, when $g_L>2/\lambda_s$, the discussion above implies that $\lambda_s t_k/(2\Delta(x^k))\geq 1$ for all $k\geq 0$; therefore, the contraction rate is smaller than $1-\frac{5}{8}(1+\rho_l)^{-1}$ and $1-\frac{5}{8}(1+\frac{2\rho_h}{1-4\rho_h})^{-1}$ for \lac{} and \hac{}, respectively. This implies that the $\cO(1)$ constant for the outer iteration complexity does not explicitly depend on $t_k$; hence, $G$ does not appear in the bound. That said, the above complexity bound for \hac{} case is not tight. Indeed,
for \hac{} whenever $G$ is selected \textit{large} so that $g_L\geq 2/\lambda_s$, we next establish in Lemma \ref{lemma_path_traj_high} that Algorithm \ref{alg:adaptive-IPL} has a better complexity bound than $\cO(\log(\frac{1}{\epsilon}))$ implied by Lemma \ref{lemma_traj_gen_tk}(b).
In the result below, we provide a one-step analysis of \adaipl{} when $t_k\in (2\Delta(x^k)/\lambda_s,L^{-1}]$ and $\Delta(x^k)$ is small---for this result, similar to the above discussion, we also do not assume that $t_k$ is selected based on \eqref{choice_t_diminish}.

\begin{lemma}\label[lemma]{lemma_path_traj_high}
 Under Assumption~\ref{ass:sharpness}, for any $k\in\mathbb{N}$, if \hac{} in \eqref{eq:low-high-practical} holds with $\rho_h\in [0,1/4)$ and  $\Delta(x^k)\leq \min\{\|x_\star\|_2/\sqrt{M_0},\lambda_s/(4L)\}$, then 
 {\small
 \begin{equation*}
     \lambda_s\Delta(x^{k+1})\leq \left(4\rho_h/t_k+4L\right)\Delta^2(x^{k}),\quad \forall~t_k\in[2\Delta(x^k)/\lambda_s,~L^{-1}],
\end{equation*}}where {\small $M_0 \triangleq \left(2+\sqrt{2}\rho_h^{3/4}/(1-2\rho_h^{1/2})\right)/(1-\sqrt{2}\rho_h^{1/4})$}.
\end{lemma}
Lemma \ref{lemma_path_traj_high} shows that when $\Delta(x^0)$ is small enough, \adaipl{} with $\cond$ set to (\textbf{HAC}) in \eqref{eq:low-high-practical} shows linear or super-linear convergence in terms of main iterations. Specifically, consider the $t_k$ selection rule in \eqref{choice_t_diminish},
when $G>0$ is chosen sufficiently \textit{large} such that $g_L> 2/\lambda_s$, the number of main iterations is $\cO\Big(\log(\frac{1}{\epsilon})/\log (g_L\lambda_s)\Big)$; thus, in terms of $\epsilon$ and $G$, we get $\cO(\log(\frac{1}{\epsilon})/\log(G))$.
In the extreme situation with $G = \infty$, we get $t_k = L^{-1}$, and choosing a constant step size $1/L$ corresponds to the \ipl{} algorithm in~\cite{zheng2023new}, whose convergence rate in terms of main iterations $(k)$ is quadratic.
\subsection{Subproblem Solvers and Computational Complexity}\label{subsec:subproblem_solvers}
In this subsection, we introduce a class of algorithms for \textit{inexactly} solving the \adaipl{} subproblem in \eqref{PL} and their computational complexity. We consider two alternative solvers that
can guarantee the suboptimality $H_k(z^{k}_j) - D_k(\lambda^{k}_j)=\cO(1/j^2)$ for all $k\geq 0$. The first one is \cite[Algorithm 1]{tseng2008accelerated}, and throughout this paper, we refer to it as the Accelerated Proximal Gradient (\apg{}) method. The second one is the Accelerated Primal-Dual
(\apd{}) algorithm introduced in \cite[Algorithm 4]{chambolle2016ergodic} and \cite[Algorithm 2.2]{hamedani2021primal}. In the following result, we state the complexity result for both methods when applied to the $k$-th \adaipl{} subproblem in \eqref{PL-rewrite-2}.
\begin{theorem}[Corollary 1(b) in \cite{tseng2008accelerated} and Theorem 2.2 in \cite{hamedani2021primal}]
\label{rate_subproblem_updated}
For any $k\geq 0$, consider either \apg{} or \apd{} applied to \eqref{PL-rewrite-2}. Let $\{(z^{k}_j,\lambda^{k}_j)\}_{j=0}^\infty$ be the iterate sequence generated. Then, $\sup_{j\in\mathbb{N}}\|\lambda^{k}_j\|_\infty\leq 1$ and there exists a  constant\footnote{$C_0$ is dimension-free and does not depend on any problem or algorithm parameters.}
$C_0\geq 2$ such that 
\begin{equation}\label{ineq_sub_rate}
{
\begin{aligned}
    H_k(z^{k}_j) - D_k(\lambda^{k}_j)\leq \frac{t_kC_0m\|B_k\|_2^2}{(j+1)^2},\quad \forall j\in\mathbb{N}_+.
    \end{aligned}}\end{equation}
\end{theorem}
To the best of our knowledge, for solving \eqref{PL-rewrite-2} when the $B_k$ matrix is arbitrary, $\cO(1/j^2)$ is the best rate we can get on the duality gap of the primal-dual iterate sequence.

The next result shows that $\|B_k\|_2$ on the r.h.s. of \eqref{ineq_sub_rate} can be uniformly bounded.
\begin{lemma}[Lemma 8 in \cite{zheng2023new}]
\label[lemma]{bound_b2}
If $\sup_{k\in\mathbb{N}}\Delta(x^k)\leq r$, then
\begin{equation*}
    \sup_{k\in\mathbb{N}} \|B_k\|_2\leq B(r) \triangleq \frac{2}{m}\|A\|_2(\|x_\star\|_2+r)\max_{i\in [m]}\|a_i\|_2.
\end{equation*}
\end{lemma}

In the rest, we discuss the complexity of the \adaipl{} when \eqref{PL-rewrite-2} is solved inexactly by some algorithm that can guarantee \eqref{ineq_sub_rate} and terminated according to \eqref{eq:low-high-practical}. For any given $\epsilon>0$, let $K_\epsilon$ denote the number of main (outer) iterations required by \adaipl{} to compute an $\epsilon$-optimal solution, i.e.,
$K_\epsilon \triangleq \inf\{k\in\mathbb{N}_+:\Delta(x^k)\leq \epsilon\}$.
Given $x^k$ for any $k\geq 0$, let $N_k\in\mathbb{N}_+$ denote the number of iterations required by the solver, i.e., inner iterations of \adaipl{}, to inexactly solve the $k$-th subproblem---in order to compute $x^{k+1}$ within the $k$-th outer iteration of \adaipl{}. For \lac{} in \eqref{eq:low-high-practical}, $N_k = \inf\{j\in\mathbb{N}:H_k(z^k_j) - D_k(\lambda^k_j)\leq \rho_l(H_k(\mathbf{0}) - H_k(z^k_j))\}$. For \hac{} in \eqref{eq:low-high-practical}, $N_k = \inf\{j\in\mathbb{N}:H_k(z^k_j) - D_k(\lambda^k_j)\leq \rho_h\|z^k_j\|_2^2/(2t_k)\}$. Therefore, the overall complexity of \adaipl{} for computing an $\epsilon$-optimal solution is thus given by
{\small
$N(\epsilon) {\triangleq} \sum_{k=0}^{K_\epsilon - 1}N_k.$}
Next, we provide an upper bound for $N_k$; similar to Lemmas \ref{lemma_traj_gen_tk} and \ref{lemma_path_traj_high}, we do not assume that $t_k$ is selected based on \eqref{choice_t_diminish}; instead, we show the result for any $t_k\in(0, 1/L]$. First, in Lemma \ref{iteration_low_sub}, we give a bound under \lac{} in \eqref{eq:low-high-practical}.
\begin{lemma}[$N_k$ bound under \lac{}]\label[lemma]{iteration_low_sub}
Suppose Assumption~\ref{ass:sharpness} holds, and a subproblem solver satisfying
\eqref{ineq_sub_rate} for some constant $C_0\geq 2$ is given. For any $k\in\mathbb{N}$, if $\Delta(x^k)\leq \lambda_s/(4L)$ and $0< t_k\leq L^{-1}$, then for any $\rho_l>0$, \lac{} in~\eqref{eq:low-high-practical} holds within
{\small
\begin{equation*}
N_k\leq M_1\sqrt{\frac{t_k}{\lambda_s\Delta(x^k)\min\{1,\lambda_st_k/(2\Delta(x^k))\}}}\cdot\sqrt{\frac{1+\rho_l}{\rho_l}}{\triangleq M^{\bf LAC}_k},
\end{equation*}}inner iterations in which $M_1 \triangleq \sqrt{{1.6} C_0mB^2\big(\lambda_s/(4L)\big)}$ and $B(\cdot)$ is defined in \Cref{bound_b2}.
\end{lemma}
Second, in Lemma~\ref{iteration_high_sub}, we give another bound considering \hac{} in \eqref{eq:low-high-practical}.
\begin{lemma}[$N_k$ bound under \hac{}]\label[lemma]{iteration_high_sub}
Under the premise of Lemma~\ref{iteration_low_sub}, for any $\rho_h\in (0,1/4)$, \hac{} in~\eqref{eq:low-high-practical} holds within 
{\small
\begin{equation*}
    N_k\leq M_2\frac{t_k}{\Delta(x^k)\min\{1,\lambda_st_k/(2\Delta(x^k))\}}\sqrt{\frac{1+\rho_h}{\rho_h}}{\triangleq M^{\bf HAC}_k},
\end{equation*}}inner iterations in which $M_2 \triangleq \sqrt{16C_0mB^2(\lambda_s/(4L))}$.
\end{lemma}

These two results
indicate that for \adaipl{} with $t_k$ selected as in \eqref{choice_t_diminish}, if $\Delta(x^0)$ is small enough, then the number of inner iterations
needed to satisfy \eqref{eq:low-high-practical} is
$\cO(1)$. More precisely, it follows from the discussion below Lemma~\ref{lemma_traj_gen_tk} that
when $G>0$ is chosen sufficiently \textit{small} so that $g_H\triangleq 3 G L\norm{x_\star}_2u_H/2\leq 2/\lambda_s$, we would have $\lambda_s t_k/(2\Delta(x^k))
\leq \lambda_s g_H/2\leq 1$ holding for all $k\geq 0$ with high probability. Thus, for both \lac{} and \hac{} in \eqref{eq:low-high-practical}, we get $N_k
{=}\cO(\sqrt{m}B(\lambda_s/(4L))/\lambda_s)$. According to \Cref{rel_lamL}, $\lambda_s/(4L)\leq \|x_\star\|_2/8$; therefore, using the  definition of $B(\cdot)$ in \Cref{bound_b2},  \begin{equation}\label{eq_temp_B_discussion}
    B\Big(\frac{\lambda_s}{4L}\Big) = \frac{2}{m}\|A\|_2\max_{i\in [m]} \|a_i\|_2\Big(\|x_\star\|_2 + \frac{\lambda_s}{4L}\Big){=}\cO(C_S\|x_\star\|_2\|A\|_2^2m^{-3/2}),
\end{equation} where $C_S\triangleq \sqrt{m} \max_{i\in [m]} \|a_i\|_2/\|A\|_2.$ Here, $C_S$ is a factor related to the complexity of solving the subproblem in~\eqref{PL-rewrite-2}, i.e., equivalently \eqref{PL}---see~\Cref{rate_subproblem_updated} and \Cref{bound_b2}. Note $C_S\geq 1$ as $\max_{i\in [m]}\|a_i\|_2^2\geq \frac{1}{m}\sum_{i=1}^m \|a_i\|_2^2 = \|A\|_F^2/m\geq \|A\|_2^2/m$. Thus, using $L\triangleq 2\|A\|_2^2/m$, we get
$$N_k{=}\cO(C_S\|x_\star\|_2\|A\|_2^2/(m\lambda_s)) = \cO(C_S\kappa_0),$$
in which $\kappa_0\triangleq L\|x_\star\|_2/(2\lambda_s)\geq 1$ is the condition number for \eqref{duchi_l1_ori} introduced in Section \ref{sec:data-gen} and Lemma \ref{rel_lamL}. It is crucial to note that this upper bound holds for all sufficiently small $G$ and does not increase as $G$ is chosen smaller.

On the other hand, when $G>0$ is chosen sufficiently \textit{large} so that $g_L \triangleq G\lambda_s u_L> 2/\lambda_s$, then according to part (b) of \Cref{prop_alpha_and_tk}, with high probability we have $\lambda_s t_k/(2\Delta(x^k))=\lambda_s g_k/2 \geq 1$ (since $g_k\geq g_L$) for all $k\geq 0$. Thus,
Lemma \ref{iteration_low_sub} and Lemma \ref{iteration_high_sub} together with
\eqref{eq_temp_B_discussion} imply that
{\small
\begin{equation*}
    N_k\leq
    \begin{cases}
        M_1\sqrt{\frac{t_k}{\lambda_s\Delta(x^k)}}\cdot\sqrt{\frac{1+\rho_l}{\rho_l}}{=} \cO(C_S\kappa_0\sqrt{g_H\lambda_s}) & \text{ for \lac{},}\\
        M_2\frac{t_k}{\Delta(x^k)}\sqrt{\frac{1+\rho_h}{\rho_h}}{=} \cO(C_S\kappa_0g_H\lambda_s) & \text{ for \hac{}.}
    \end{cases}
\end{equation*}}Since $g_H\triangleq 3 G L\norm{x_\star}_2u_H/2$, in terms of the dependency on $G$, we have $N_k = \cO(\sqrt{G})$ under \lac{} in \eqref{eq:low-high-practical}, and $N_k = \cO(G)$ under \hac{} in \eqref{eq:low-high-practical}.

\subsection{Overall Complexity}
Finally, we are ready to provide the overall iteration complexity for \adaipl{}. Indeed, in \Cref{overall_complexity_diminish_stepsize_gen} we establish a bound on $N(\epsilon)$ for \textit{any} given $G>0$ when $t_k$ is chosen according to \eqref{choice_t_diminish}.
\begin{theorem}\label{overall_complexity_diminish_stepsize_gen}
Suppose Assumption~\ref{ass:sharpness} holds and a subproblem solver satisfying
\eqref{ineq_sub_rate} is given. Under the statistical event in \eqref{condition_med_prop}, which holds with high probability under Assumption \ref{ass_data},
if $\Delta(x^0)\leq \min\{E(\frac{\lambda_s^2}{4L}),E(\frac{\lambda_s}{g_H L})\}$, where
$ E(r) \triangleq \frac{1}{2L}\left(\sqrt{L^2\|x_\star\|_2^2+4rL} - L\|x_\star\|_2\right)$ defined in Lemma \ref{lip_F2} and $g_H\geq g_L>0$ are constants defined in \Cref{prop_alpha_and_tk}, then the following conclusions hold with $C(\rho)\triangleq 1 - \frac{5\min\{\lambda_sg_L,~2\}}{16(1+\rho)}$.

(a) When \emph{$\cond$} in \emph{\adaipl{}}, stated in \Cref{alg:adaptive-IPL}, is set to (\textbf{LAC}) in \eqref{eq:low-high-practical} for some $\rho_l>0$, for any $\epsilon\in (0,\Delta(x^0))$, it holds that
{\small
\begin{equation*}
    N(\epsilon)\leq M_1\cdot\frac{\sqrt{\frac{g_H}{\lambda_s\min\{1,\lambda_sg_H/2\}}}\cdot\sqrt{\frac{1+\rho_l}{\rho_l}}\cdot\log\left(\frac{2\Delta(x^0)L\|x_\star\|_2}{\lambda_s C(\rho_l)
}\cdot\frac{1}{\epsilon}\right)}{\log\left(1/C(\rho_l)
\right)}.
\end{equation*}}
(b) When \emph{$\cond$} in \emph{\adaipl{}}, stated in \Cref{alg:adaptive-IPL}, is set to (\textbf{HAC}) in \eqref{eq:low-high-practical} for some $\rho_h\in (0,1/4)$, for any $\epsilon\in (0,\Delta(x^0))$, it holds for ${\bar\rho} \triangleq 2\rho_h/(1-4\rho_h)$ that
{\small
\begin{equation*}
    N(\epsilon)\leq M_2\cdot \frac{\frac{g_H}{\min\{1,\lambda_sg_H/2\}}\sqrt{\frac{1+\rho_h}{\rho_h}}\cdot\log\left(\frac{2\Delta(x^0)L\|x_\star\|_2}{\lambda_s C(\bar\rho)
}\cdot\frac{1}{\epsilon}\right)}{\log\left(1/C(\bar\rho)
\right)}.
\end{equation*}}\end{theorem}
Next, we discuss that whenever $\Delta(x^0)$ is sufficiently small and $G$ in \eqref{choice_t_diminish} is sufficiently large, using \hac{} in \eqref{eq:low-high-practical} leads to a better complexity bound when compared to that of Theorem \ref{overall_complexity_diminish_stepsize_gen}(b).
\begin{theorem}\label{overall_complexity_diminish_stepsize}
Suppose Assumption~\ref{ass:sharpness} holds and a subproblem solver satisfying
\eqref{ineq_sub_rate} is given. Moreover, we assume that the probabilistic event in
\eqref{condition_med_prop} holds. If constants defined in \Cref{prop_alpha_and_tk} satisfy $g_H\geq g_L\geq 2/\lambda_s$, \emph{$\cond$} in \emph{\adaipl{}}, stated in \Cref{alg:adaptive-IPL}, is set to (\textbf{HAC}) in \eqref{eq:low-high-practical} for some $\rho_h\in (0,1/4)$, and $\Delta(x^0)\leq \min\left\{\frac{\rho_h}{2Lg_L},\|x_\star\|_2/\sqrt{M_0},(g_HL)^{-1}\right\},$ 
then for any $\epsilon\in (0,\Delta(x^0))$, it holds that 
{\small
\begin{equation*}
    N(\epsilon)\leq M_2\cdot\frac{g_H\cdot \sqrt{\frac{1+\rho_h}{\rho_h}}\cdot \log \left(\frac{\lambda_s\Delta(x^0)g_L}{6\rho_h}\cdot\frac{1}{\epsilon}\right)}{\log\left(\lambda_sg_L/(6\rho_h)\right)}
\end{equation*}}\end{theorem}
\begin{remark}
    Initialization requirements on $\Delta(x^0)$ in both Theorem \ref{overall_complexity_diminish_stepsize_gen} and Theorem \ref{overall_complexity_diminish_stepsize} can be satisfied by Algorithm 3 in \cite{duchi2019solving} with high probability when $m/n$ is large enough and $p_{\mathrm{fail}}$ is small enough. Please refer to Remark \ref{remark_init_subgradiant} for more explanations.
\end{remark}

Next, we explain the results in Theorems \ref{overall_complexity_diminish_stepsize_gen} and \ref{overall_complexity_diminish_stepsize}.
Indeed, when $\Delta(x^0)$ is sufficiently small, the total complexity for reaching an $\epsilon-$optimal point is $\cO(\log \frac{1}{\epsilon})$ for any $G>0$ and $\epsilon < \Delta(x^0)$. Specifically, when $G>0$ is chosen sufficiently \textit{small} so that $g_H\leq 2/\lambda_s$, we have $\min\{1,\lambda_sg_H/2\} = \lambda_sg_H/2$ and $\min\{1,\lambda_sg_L/2\} = \lambda_sg_L/2$. Therefore, \Cref{overall_complexity_diminish_stepsize_gen} implies that $N(\epsilon) {=} \cO(\frac{C_S\kappa_0}{g_L\lambda_s}\log \frac{1}{\epsilon})$ for both \lac{} and \hac{} in \eqref{eq:low-high-practical}. Moreover, since $g_L = G\lambda_s u_L$, the complexity bound is $\cO\Big(\frac{1}{G}\log(\frac{1}{\epsilon})\Big)$ in terms of $G$ and $\epsilon$. On the other hand, when $G>0$ is chosen sufficiently \textit{large} so that $g_L > 2/\lambda_s$, we have $\min\{1,\lambda_sg_H/2\} = \min\{1,\lambda_sg_L/2\} = 1$. Thus, \Cref{overall_complexity_diminish_stepsize_gen}(a) and \Cref{overall_complexity_diminish_stepsize} imply that
$N(\epsilon) {=} \cO(C_S\kappa_0\sqrt{g_H\lambda_s}\log \frac{1}{\epsilon})$ for \lac{} in \eqref{eq:low-high-practical} and $N(\epsilon) {=}\cO(\frac{C_S\kappa_0 g_H\lambda_s}{\log (g_L\lambda_s)}\log \frac{1}{\epsilon})$ for \hac{} in \eqref{eq:low-high-practical}, respectively. Recalling that $g_L \triangleq G\lambda_s u_L$ and $g_H \triangleq 3G L\|x_\star\|_2 u_H/2$, the complexity depends on $G$ and $\epsilon$ as $\cO(\sqrt{G}\log \frac{1}{\epsilon})$ for \lac{} in \eqref{eq:low-high-practical} and as $\cO(\frac{G}{\log G}\log \frac{1}{\epsilon})$ for \hac{} in \eqref{eq:low-high-practical}. To conclude, if we were to set $t_k = 2\Delta(x^k)/\lambda_s$ for any $k\in\mathbb{N}$, \adaipl{} would achieve a total complexity of $\cO(C_S\kappa_0\log\frac{1}{\epsilon})$ under both \lac{} and \hac{}. However, since setting $g_k = 2/\lambda_s$ for $k\in\mathbb{N}$ may be impractical, the corresponding bound can be interpreted as the ideal total complexity bound.
\subsection{Comparison of \adasub{} and \adaipl{} complexities}
\label{sec:complexity-comparison}
Here we compare the total complexity of \adasub{} and \adaipl{} for reaching an $\epsilon-$optimal solution of \eqref{duchi_l1_ori} with the existing deterministic algorithms summarized in Section \ref{sec:intro}. The total complexity refers to total \textit{inner iteration} numbers for \pl{}, \ipl{}, and \adaipl{} and refers to iteration numbers for \psub, \gsub, and \adasub{}. The comparison is
fair as the computational complexity of either a subgradient-type iteration or an inner iteration for solving the subproblem in~\eqref{PL} using either \apg{} or \apd{} is dominated by matrix-vector multiplications involving matrices of size $m\times n$.

In Table~\ref{tab:perfect_situation_modified}
provided in the introduction we summarized the total complexities for all the methods under the ideal situations for each of them. The total complexity for \pl{} is unknown because it requires solving \eqref{PL} \textit{exactly}, which is not practical. \ipl{} is only guaranteed to achieve sublinear rate leading to $\cO(1/\epsilon)$ complexity. In the ideal scenario, the convergence of \psub, \gsub, and \adasub{} is linear, leading to $\cO(\kappa_0^2\log \frac{1}{\epsilon})$ complexity. Here, the ideal scenario means that $F(x_\star)$ is known for \psub, $\lambda_0$ and $q$ are set exactly as in \cite[Theorem 5.1]{davis2018subgradient} for \gsub{} (which is not practical), and $c = 0$ in Theorem \ref{thm:convergence_ada_subgrad} for \adasub{}. The total complexity of \adaipl{}
is $\cO(C_S\kappa_0\log \frac{1}{\epsilon})$ under the ideal situation that $g_L = g_H = 2/\lambda_s$ in \Cref{prop_alpha_and_tk}(b). It is the best in terms of the condition number but contains an additional factor $C_S$ that results from the iteration complexity associated with inexactly solving the subproblems in \eqref{PL}. Thus, \adaipl{} is expected to show greater efficiency than \adasub{} when $\kappa_{\mathrm{st}}/\|\Sigma\|_2 - 2p_{\mathrm{fail}}$ is small (this quantity is discussed in Section \ref{sec:data-gen} and is approximately equal to $1/\kappa_0$), i.e., when the condition number for \eqref{duchi_l1_ori} is \textit{large}.

Next, aiming to compute an $\epsilon$-optimal solution to \eqref{duchi_l1_ori} for a given sufficiently small $\epsilon>0$, in Table \ref{tab:imperfect_situation}, we
discuss the robustness of \adaipl{} and \adasub{} to the algorithm parameter $G$ in terms of main (outer) iterations ($K_\epsilon$), the largest number of inner iterations per subproblem solve ($\sup_{k\in\mathbb{N}}N_k$), and the total complexity ($N_\epsilon$)---we consider the effects of overly large or overly small choices of the step size parameter $G>0$ (see the stepsize rules in~\eqref{choice_alpha_subg}
and \eqref{choice_t_diminish}).
Here, we treat each \adasub{} iteration as an outer iteration requiring only one inner iteration for each outer iteration, i.e., $N(\epsilon)=K_\epsilon$.
Table \ref{tab:imperfect_situation} highlights the advantages of \adasub{} and \adaipl{} in hyperparameter tuning over \psub{} and \gsub{}: \adasub{} tolerates overly small $G$ values and \adaipl{} works with any $G>0$ while \psub{} and \gsub{} rely on some particular choice of hyperparameter values to guarantee convergence.

\renewcommand{\arraystretch}{1.3}
    \begin{table}[ht]
        \centering
        {\small
    \begin{tabular}{|c|c|c|c|c|}
    \hline
        \textbf{Algorithm} & $G$ & $K_\epsilon$ & $\sup_{k\in\mathbb{N}} N_k$ & $N(\epsilon)$\\
          \hline
        \adasub{} & Large ($c_1\geq 2$) & diverge & diverge & diverge\\
          \hline
        \adasub{} & Small ($c_2 < 1$) & $\tilde{\cO}\left(\kappa_0^2/c_1\right)$ & 1 & $\tilde{\cO}\left(\kappa_0^2/c_1\right)$\\
          \hline
         \adaipl{}\texttt{-LAC} & Large ($g_L\geq \frac{2}{\lambda_s}$) & $\tilde{\cO}(1)$ & $\cO(C_S\kappa_0\sqrt{g_H\lambda_s})$ & $\tilde{\cO}(C_S\kappa_0\sqrt{g_H\lambda_s})$\\
           \hline
          \adaipl{}\texttt{-HAC} & Large ($g_L\geq \frac{2}{\lambda_s}$) & $\tilde{\cO}(1/\log (g_L\lambda_s))$ & $\cO(C_S\kappa_0 g_H\lambda_s)$ & $\tilde{\cO}\Big(\frac{C_S\kappa_0g_H\lambda_s}{\log (g_L\lambda_s)}\Big)$\\
            \hline
          \adaipl{}\texttt{-LAC} & Small ($g_H < \frac{2}{\lambda_s}$) & $\tilde{\cO}(1/(g_L\lambda_s))$ & $\cO(C_S\kappa_0)$ & $\tilde{\cO}(\frac{C_S\kappa_0}{g_L\lambda_s})$\\
          \hline
        \adaipl{}\texttt{-HAC} &Small ($g_H < \frac{2}{\lambda_s}$) & $\tilde{\cO}(1/(g_L\lambda_s))$ & $\cO(C_S\kappa_0)$ & $\tilde{\cO}(\frac{C_S\kappa_0}{g_L\lambda_s})$\\
            \hline
    \end{tabular}}\caption{{Robustness of \adasub{} and \adaipl{} to parameter choice in terms of
    number of outer iterations ($K_\epsilon$), the maximum number of inner iterations per subproblem solve ($\sup_{k\in\mathbb{N}} N_k$), and the total complexity ($N(\epsilon)$). ``\emph{\texttt{-LAC}}" and ``\emph{\texttt{-HAC}}" indicate whether (\textbf{LAC}) or (\textbf{HAC}) is used for \adaipl{}.} Here $\tilde{\cO}(\cdot)$ hides $\log \frac{1}{\epsilon}$.}\label{tab:imperfect_situation}
\end{table}

\section{Proofs for Section \ref{sec:conv_adaipl}}\label{sec:proof_adaipl}
We first introduce some notation. For all $k\geq 0$, let
\begin{align*}
S_{t_k}(x^k) \triangleq \argmin_{x\in\mathbb{R}^n} \ F_{t_k}(x;x^k),\quad
\varepsilon_{t_k}(x;x^k) \triangleq F_{t_k}(x;x^k) - F_{t_k}(S_{t_k}(x^k);x^k),
\end{align*}
where $F_t(\cdot;\cdot)$ is defined in \eqref{def-Ft}. Here, the minimizer $S_{t_k}(x^k)$ is unique since $F_{t_k}(\cdot;x^k)$ is strongly convex
with modulus $1/t_k$ for all $k\geq 0$. According to \Cref{sufficiency_pd}, \lac{} and \hac{} in \eqref{eq:low-high-practical} provide sufficient conditions for \lace{} and \hace{} in \eqref{eq:low-high} to hold, respectively.

\subsection{\texorpdfstring{Proof of \Cref{lemma_traj_gen_tk}}{Proof of Lemma~\ref{lemma_traj_gen_tk}}}
Within this proof, without loss of generality, we assume that $\Delta(x^k) = \|x^k - x_\star\|_2$. We split the proof into four parts.
\subsubsection{\texorpdfstring{\eqref{PL} is exactly solved and the step size belongs to $[2\Delta(x^k)/\lambda_s,L^{-1}]$}{(\ref{PL}) is exactly solved and the step size belongs to [2 Delta(x-k)/lambda-s, 1/L]}}\label{sec:proof_lemma61_part1}
Let $\tsup[1]{x}^{k+1} \triangleq S_{\tilde{t}_k}(x^k)$ for some arbitrary  $\tsup[1]{t}_{k} \in [2\Delta(x^k)/\lambda_s,L^{-1}]$---the interval is not empty because we assume $\Delta(x^k)\leq \lambda_s/(4L)$ in Lemma \ref{lemma_traj_gen_tk}; hence,
\begin{align}\label{eq00_proof_gen_descent_tk}
     F(\Tilde{x}^{k+1}) \leq F_{\tilde{t}_k}(\Tilde{x}^{k+1};x^k)\leq F(x_\star;x^k) + \frac{1}{2\tilde t_k}\|x_\star - x^k\|_2^2\leq F(x_\star) + \left(\frac{1}{2\tilde t_k}+\frac{L}{2}\right)\|x_\star - x^k\|_2^2
 \end{align}in which the first and the third inequalities hold because of \eqref{rel:gen_weak1} and the second one holds because $\Tilde{x}^{k+1}$ is the minimizer of $F_{\tilde t_k}(\cdot;x^k)$. Due to the fact that $\Delta(x^k) = \|x^k-x_\star\|_2\leq \lambda_s/(4L)$ and $\tilde t_k\geq 2\Delta(x^k)/\lambda_s$, we have
$\left(\frac{1}{2\tilde t_k}+\frac{L}{2}\right)\|x_\star - x^k\|_2^2\leq \frac{\lambda_s}{4}\Delta(x^k) +\frac{L}{2}\frac{\lambda_s}{4L}\Delta(x^k) = \frac{3}{8}{\lambda_s}\Delta({x^k});$hence, using \eqref{ineq:sharpness}, we get $\left(\frac{1}{2\tilde{t}_k}+\frac{L}{2}\right)\|x_\star - x^k\|_2^2\leq 3(F(x^k) - F(x_\star))/8$.
Applying it to \eqref{eq00_proof_gen_descent_tk}, we have
\begin{equation}\label{eq11_proof_gen_descent_tk}
    F(x^k) - F(\Tilde{x}^{k+1})\geq F(x^k) - F_{\tilde{t}_k}(\Tilde{x}^{k+1};x^k)\geq \frac{5}{8}(F(x^k) - F(x_\star)).
\end{equation}
\subsubsection{\texorpdfstring{\eqref{PL} is exactly solved and the step size is less than $2\Delta(x^k)/\lambda_s$}{(\ref{PL}) is exactly solved and the step size is less than 2 Delta(x-k)/lambda-s}}
Denote $\tsup{x}^{k+1} \triangleq S_{\tsup{t}_k}(x^k)$ for some $\tsup{t}_{k}\in (0, 2\Delta(x^k)/\lambda_s)$. Consider the result in \Cref{sec:proof_lemma61_part1} for $\tsup[1]{t}_k = 2\Delta(x^k)/\lambda_s$ and let $\bar{x}^{k+1} \triangleq x^k + \frac{\tsup{t}_k}{\tsup[1]{t}_k}(\Tilde{x}^{k+1} - x^k)$. Hence, 
{\small
\begin{align*}
    F_{\tilde{t}_k}(\bar{x}^{k+1};x^k)\leq \frac{\tsup{t}_k}{\tilde{t}_k}F_{\tilde{t}_k}(\Tilde{x}^{k+1};x^k) + \left(1- \frac{\tsup{t}_k}{\tilde{t}_k}\right)F(x^k) - \frac{(\frac{\tsup{t}_k}{\tilde{t}_k})(1-\frac{\tsup{t}_k}{\tilde{t}_k})}{2\tilde{t}_k}\|x^k - \Tilde{x}^{k+1}\|_2^2,
\end{align*}}by strong convexity of $F_{\tsup[1]{t}_k}(\cdot;x^k)$. Thus, we have
{\small
\begin{align*}
&F(x^k) - F_{\tsup{t}_k}(\bar{x}^{k+1};x^k) = F(x^k) - F_{\tilde{t}_k}(\bar{x}^{k+1};x^k) +
\Big(\frac{1}{2\tilde{t}_k}-\frac{1}{2\tsup{t}_k}\Big)\|\bar{x}^{k+1} - x^k\|_2^2\\ \nonumber
&\geq \frac{\tsup{t}_k}{\tilde{t}_k}(F(x^k) - F_{\tilde{t}_k}(\Tilde{x}^{k+1};x^k)) + \frac{(\frac{\tsup{t}_k}{\tilde{t}_k})(1-\frac{\tsup{t}_k}{\tilde{t}_k})}{2\tilde{t}_k}\|x^k - \Tilde{x}^{k+1}\|_2^2 + \Big(\frac{1}{2\tilde{t}_k}-\frac{1}{2\tsup{t}_k}\Big)\|\bar{x}^{k+1} - x^k\|_2^2 \\
&= \frac{\tsup{t}_k}{\tilde{t}_k}(F(x^k) - F_{\tilde{t}_k}(\Tilde{x}^{k+1};x^k)){\geq \frac{5\tsup{t}_k}{8\tilde{t}_k}(F(x^k) - F(x_\star)),}
\end{align*}}where in the last inequality we used \eqref{eq11_proof_gen_descent_tk}.
Since $\tsup{x}^{k+1}$ is the minimizer of $F_{\tsup{t}_k}(\cdot;x^k)$, we have $F_{\tsup{t}_k}({\tsup{x}^{k+1}};x^k)\leq F_{\tsup{t}_k}(\bar{x}^{k+1};x^k)$, which implies $F(x^k) - F_{\tsup{t}_k}({\tsup{x}^{k+1}};x^k)\geq \frac{5\tsup{t}_k}{8\tilde{t}_k}(F(x^k) - F(x_\star))$; therefore, 
\begin{equation}\label{eq22_proof_gen_descent_tk}
{
\begin{aligned}
    F(x^k) - F({\tsup{x}^{k+1}})\geq F(x^k) - F_{\tsup{t}_k}({\tsup{x}^{k+1}};x^k)\geq \frac{5\tsup{t}_k}{8\tilde{t}_k}(F(x^k) - F(x_\star)).
\end{aligned}}
\end{equation}
Here, the first inequality is from \eqref{rel:gen_weak1}. This finishes part 2.
\subsubsection{Proof of part (a)} Combining \eqref{eq11_proof_gen_descent_tk} and \eqref{eq22_proof_gen_descent_tk}, we can derive that, for any $t_k\in (0,L^{-1}]$, whenever $\Delta(x^k)\leq \lambda_s/(4L)$ holds, one also has
\begin{equation}\label{eq_proof_eq_low_decrease}
\begin{aligned}
    F(x^k) - F_{t_k}(S_{t_k}(x^k);x^k)\geq \frac{5\min\{\lambda_st_k/(2\Delta(x^k)),1\}}{8}(F(x^k) - F(x_\star)).
\end{aligned}
\end{equation}
Furthermore, we have
$F(x^k) - F(x^{k+1})\geq F(x^k) - F_{t_k}(x^{k+1};x^k)\geq \frac{1}{1+\rho_l}(F(x^k) - F_{t_k}(S_{t_k}(x^k);x^k))$, where the first inequality is due to  \eqref{rel:gen_weak1}, and the second one originates from \lace{} in \eqref{low-high-0}, which is implied by \lac. Finally, combining the last inequality with \eqref{eq_proof_eq_low_decrease}, we complete the proof.
\subsubsection{Proof of part (b)} For part (b), we first establish some connections between \lace{} and \hace{} in \eqref{low-high-0}.
We know that
{\small
\begin{align*}
    &F_{t_k}(x^k;x^k) - F_{t_k}(x^{k+1};x^k) + \varepsilon_{t_k}(x^{k+1};x^k) = F_{t_k}(x^k;x^k) - F_{t_k}(S_{t_k}(x^k);x^k) \\
    &\geq \frac{1}{2t_k}\|x^k-S_{t_k}(x^k)\|_2^2\geq \frac{1}{4t_k}\|x^k-x^{k+1}\|_2^2-\frac{1}{2t_k}\|x^{k+1}-S_{t_k}(x^k)\|_2^2,
\end{align*}}where the first inequality is due to the strong convexity of $F_{t_k}(\cdot;x^k)$, and the second one is from the Cauchy-Schwarz inequality. Using the fact that
$\frac{1}{2t_k}\|x^{k+1}-S_{t_k}(x^k)\|_2^2\leq \varepsilon_{t_k}(x^{k+1};x^k)$
due to strong convexity of $F_{t_k}(\cdot;x^k)$, we further obtain 
\begin{equation}\label{eq_proof_high_temp}
{
\begin{aligned}
    F_{t_k}(x^k;x^k) - F_{t_k}(x^{k+1};x^k) + 2\varepsilon_{t_k}(x^{k+1};x^k)\geq \frac{1}{4t_k}\|x^k-x^{k+1}\|_2^2.
\end{aligned}}\end{equation}
Note that the inexact subproblem termination condition \hac{} implies \hace{} in \eqref{low-high-0}, which further implies that $\varepsilon_{t_k}(x^{k+1};x^k)\leq \rho_h\|x^k-x^{k+1}\|_2^2/(2t_k)$; therefore,
it follows from \eqref{eq_proof_high_temp} that 
\begin{equation}\label{eq_proof_high_temp2}
\begin{aligned}
    \frac{2\rho_h}{1-4\rho_h}\left(F_{t_k}(x^k;x^k) - F_{t_k}(x^{k+1};x^k)\right)\geq \frac{\rho_h}{2t_k}\|x^k-x^{k+1}\|_2^2\geq {\varepsilon_{t_k}(x^{k+1};x^k)}.
\end{aligned}\end{equation}
This takes the same form as \lace{} in \eqref{low-high-0} when $\rho_l = \frac{2\rho_h}{1-4\rho_h}$. Therefore, it directly follows from the arguments in the proof of part (a) that one can simply replace $\rho_l$ in (a) with $\frac{2\rho_h}{1-4\rho_h}$. 

\subsection{\texorpdfstring{Proof of \Cref{lemma_path_traj_high}}{Proof of Lemma~\ref{lemma_path_traj_high}}}
We first show an auxiliary result for one-step behavior of \adaipl{}.
\begin{lemma}\label[lemma]{descent_general}
Given $\beta\in(0,1]$ and arbitrary $t_k>0$, for any \emph{\adaipl{}} iterate $x^k$,
{\small
\begin{align}\label{ineq:descent_general}
        F({x}) &- F(x_\star)+\frac{1-\beta}{2t_k}\|{x} - x_\star\|_2^2\leq \left(\frac{1}{2t_k}+\frac{L}{2}\right)\|x^k - x_\star\|_2^2+\beta^{-1}\varepsilon_{t_k}({x},x^k) + \left(\frac{L}{2} - \frac{1}{2t_k}\right)\|{x}-x^k\|_2^2,
\end{align}}for all $x\in\reals^n$. The above inequality is also valid if we replace $x_\star$ with $-x_\star$ in \eqref{ineq:descent_general}.
\end{lemma}
\begin{proof}For arbitrary $\beta\in(0,1]$ and any $t_k>0$, we have the following inequalities for all $x\in\reals^n$:
{\small
\begin{align}\label{proof-lemma4-eq-1}
   -\frac{1}{2t_k}\|x_\star-S_{t_k}(x^k)\|_2^2
 &\leq \frac{\beta-1}{2t_k}\|x-x_\star\|_2^2+\frac{\beta^{-1}-1}{2t_k}\|x - S_{t_k}(x^k)\|_2^2\\
 & \leq\frac{\beta-1}{2t_k}\|x-x_\star\|_2^2+(\beta^{-1}-1)\varepsilon_{t_k}(x;x^k), \nonumber
\end{align}}in which the first inequality holds because of the Cauchy–Schwarz inequality, and the second inequality holds since $F_{t_k}(\cdot;x^k)$ is $1/t_k-$strongly convex. We then have
{\small
\begin{align*}
    \MoveEqLeft F(x)-\varepsilon_{t_k}(x;x^k)\leq
    F_{t_k}(S_{t_k}(x^k);x^k) + \left(\frac{L}{2} - \frac{1}{2t_k} \right)\|x-x^k\|_2^2 \\
    & \leq F(x_\star;x^k)+\frac{1}{2t_k}\|x^k-x_\star\|_2^2-\frac{1}{2t_k}\|x_\star-S_{t_k}(x^k)\|_2^2 + \frac{Lt_k-1}{2t_k} \|x-x^k\|_2^2 \\
    &\leq F(x_\star)+\frac{Lt_k+1}{2t_k}\|x^k-x_\star\|_2^2-\frac{1}{2t_k}\|x_\star-S_{t_k}(x^k)\|_2^2 + \frac{Lt_k-1}{2t_k} \|x-x^k\|_2^2,
\end{align*}}where the first and the last inequalities use \eqref{rel:gen_weak1}, and the second inequality is from the strong convexity of $F_{t_k}(\cdot;x^k)$.
 Combining this inequality with \eqref{proof-lemma4-eq-1} yields \eqref{ineq:descent_general}.
 It is easy to verify that the proof still holds when we replace $x_\star$ with $-x_\star$.
\end{proof}
Next, we provide another useful result to be used in the proof of \Cref{lemma_path_traj_high}.
\begin{lemma}\label[lemma]{lemma_same_closeness}
    Suppose Assumption \ref{ass:sharpness} holds. For any $k\in\mathbb{N}$ with $t_k\in(0, 1/L]$, if \hace{} in \eqref{low-high-0} holds with $\rho_h\in [0,1/4)$ and $\Delta(x^k)\leq \|x_\star\|_2/\sqrt{M_0}$, then
\begin{equation}\label{proof42_same_min}
{
\begin{aligned}
    \mathbf{1}\left[\Delta(x^k) = \|x^k - x_\star\|_2\right] = \mathbf{1}\left[\Delta(x^{k+1}) = \|x^{k+1} - x_\star\|_2\right],
\end{aligned}}\end{equation}
where {\small $M_0 \triangleq \left(2+\sqrt{2}\rho_h^{3/4}/(1-2\rho_h^{1/2})\right)/(1-\sqrt{2}\rho_h^{1/4})$} and $\mathbf{1}[\cdot]$ is the indicator function.
\end{lemma}
\begin{proof}Without loss of generality, we assume that $\Delta(x^k) = \|x^k-x_\star\|_2.$ We will split the proof into two cases. 
    \textbf{CASE 1: $\rho_h\in (0,1/4)$.} Consider \eqref{ineq:descent_general} with $\beta\in (0,1]$ and $x = x^{k+1}$. Since \hace{} in \eqref{low-high-0} holds, we have $\beta^{-1}\varepsilon_{t_k}({x}^{k+1},x^k)\leq \beta^{-1}\rho_h\|x^k - x^{k+1}\|_2^2/(2 t_k)$. Moreover, noticing $L/2+1/(2t_k)\leq 1/t_k$, $L/2-1/(2t_k)\leq 0$ and $F(x^{k+1}) - F(x_\star)\geq 0$, \eqref{ineq:descent_general} implies that
    $\frac{1-\beta}{2t_k}\|x^{k+1} - x_\star\|_2^2\leq \frac{1}{t_k}\|x^k - x_\star\|_2^2 + \frac{\beta^{-1}\rho_h}{2t_k}\|x^k - x^{k+1}\|_2^2$. Then, for any $u'\in (0,1)$, the Cauchy-Schwarz inequality implies $(1-\beta)\|x^{k+1} - x_\star\|_2^2\leq 2\|x^k - x_\star\|_2^2$ $+ \beta^{-1}\rho_h\left(\frac{\|x^{k+1} - x_\star\|_2^2}{u'} + \frac{\|x^{k} - x_\star\|_2^2}{1-u'}\right)$, which further leads to
    \begin{equation}\label{proof-lemma-exact-eq1}
    {
    \begin{aligned}
       \Big(1-\beta -\frac{\rho_h}{\beta u'}\Big)\|x^{k+1} - x_\star\|_2^2 \leq \Big(2+\frac{\rho_h}{\beta (1-u')}\Big)\|x^k - x_\star\|_2^2.
    \end{aligned}}\end{equation}
Hence, setting $u' = \sqrt{4\rho_h}\in (0,1)$ and $\beta = \sqrt{\rho_h/u'} = (\rho_h/4)^{1/4}\in (0,1)$, we have that
$1-\beta -\rho_h/(\beta u') = {1-\sqrt{2}\rho_h^{1/4}}\in (0,1)$ as $\rho_h\in (0,1/4)$.
Thus, \eqref{proof-lemma-exact-eq1} indicates that
$\|x^{k+1} - x_\star\|_2^2 \leq \frac{2+\beta^{-1}\rho_h/(1-u')}{1-\beta-\rho_h/(\beta u')}\|x^k - x_\star\|_2^2 = M_0\|x^k - x_\star\|_2^2\leq \|x_\star\|_2^2.$ Thus, $\|x^{k+1} + x_\star\|_2\geq 2\|x_\star\|_2 - \|x^{k+1} - x_\star\|_2\geq \|x^{k+1} - x_\star\|_2$. This means that $x^{k+1}$ is closer to $x_\star$ than $-x_\star$. 
\textbf{CASE 2: $\rho_h=0$.} Note $\rho_h=0$ implies that $x^{k+1} = S_{t_k}(x^k)$ and $\varepsilon_{t_k}(x^{k+1};x^k) = 0$. Thus, setting $x = x^{k+1}$ in \eqref{ineq:descent_general} and using the facts that $L/2+1/(2t_k)\leq 1/t_k, L/2-1/(2t_k)\leq 0, F(x^{k+1}) - F(x_\star)\geq 0$, we have
$\frac{1-\beta}{2t_k}\|x^{k+1} - x_\star\|_2^2\leq \frac{1}{t_k}\|x^k - x_\star\|_2^2$ for any $\beta\in (0,1]$. Letting $\beta\rightarrow {0^+}$, we have $\|x^{k+1} - x_\star\|_2^2\leq 2\|x^k - x_\star\|_2^2\leq \|x_\star\|_2^2$ since $M_0=2$. Thus, $\|x^{k+1} + x_\star\|_2\geq 2\|x_\star\|_2 - \|x^{k+1} - x_\star\|_2\geq \|x^{k+1} - x_\star\|_2$. This means that $x^{k+1}$ is closer to $x_\star$ than $-x_\star$. \end{proof}
\begin{remark}
    Later, when we invoke \Cref{lemma_same_closeness} with $\rho_h = 0$, we use ${M_0} = {2}$.
\end{remark}

Now we are ready to prove \Cref{lemma_path_traj_high}.
\begin{proof}[Proof of Lemma \ref{lemma_path_traj_high}]
Without loss of generality, we assume $\Delta(x^k) = \|x^k-x_\star\|_2$. By \eqref{proof42_same_min} in Lemma \ref{lemma_same_closeness}, we know that $\Delta(x^{k+1}) = \|x^{k+1} - x_\star\|_2$. Consider \eqref{ineq:descent_general} with $x = x^{k+1}$ and $\beta=\frac{1}{2}\in (0,1]$. Since \hace{} in \eqref{low-high-0} holds, we have $\beta^{-1}\varepsilon_{t_k}(x^{k+1};x^k)\leq \beta^{-1}\rho_h\|x^k - x^{k+1}\|_2^2/(2 t_k)$. Using this inequality within \eqref{ineq:descent_general} together with $F(x^{k+1}) - F(x_\star)\geq \lambda_s\Delta(x^{k+1})$ due to \eqref{ineq:sharpness}, we get 
{\small
\begin{equation*}\lambda_s\Delta(x^{k+1})+\frac{1}{4t_k}\Delta^2(x^{k+1})\leq \left(\frac{1}{2t_k}+\frac{L}{2}\right)\Delta^2(x^k)+\overbrace{\left(\frac{L}{2}-\frac{1-2\rho_h}{2t_k}\right)\|x^k-x^{k+1}\|_2^2}^{\triangleq\Gamma_k}.
\end{equation*}}Note that
${\Gamma_k} = \frac{L}{2}\|(x^k-x_\star) - (x^{k+1} - x_\star)\|_2^2 - \frac{1-2\rho_h}{2t_k}\|(x^k-x_\star) - (x^{k+1} - x_\star)\|_2^2$.
From \eqref{proof42_same_min}, we also have that
$\frac{L}{2}\|(x^k-x_\star) - (x^{k+1} - x_\star)\|_2^2\leq \frac{L}{2}(\Delta(x^k)+\Delta(x^{k+1}))^2$
and
$ - \frac{1-2\rho_h}{2t_k}\|(x^k-x_\star) - (x^{k+1} - x_\star)\|_2^2\leq -\frac{1-2\rho_h}{2t_k}(\Delta(x^k)-\Delta(x^{k+1}))^2$.
Thus,
\begin{equation}\label{eq2_proof_high}
{
\begin{aligned}
    \Big(\lambda_s - \Big(L + \frac{1-2\rho_h}{t_k}\Big)\Delta(x^k)\Big)\Delta(x^{k+1}) +{\left(\frac{3-4\rho_h}{4t_k} - \frac{L}{2}
    \right)}\Delta^2(x^{k+1})\leq \left(\frac{\rho_h}{t_k}+L\right)\Delta^2(x^{k}).
\end{aligned}}\end{equation}
From the hypothesis we have $\Delta(x^k)\leq \lambda_s/(4L)$, $t_k\in [2\Delta(x^k)/\lambda_s,~1/L]$ and $\rho_h\in (0,1/4)$; therefore,
we get
$\lambda_s - L\Delta(x^k) - \frac{1-2\rho_h}{t_k}\Delta(x^k)\geq \lambda_s - \lambda_s/4 - \Delta(x^k)/t_k\geq \lambda_s/4$
and
${\frac{3-4\rho_h}{4t_k}} - \frac{L}{2} = \left(\frac{1}{2t_k} - \frac{L}{2}\right) + \left(\frac{1-4\rho_h}{4t_k}\right)\geq 0.$
Using these relations within \eqref{eq2_proof_high}, we obtain
$\lambda_s\Delta(x^{k+1})\leq \left(\frac{4\rho_h}{t_k}+4L\right)\Delta^2(x^{k})$; hence, the proof is complete.
\end{proof}

\subsection{\texorpdfstring{Proofs of \Cref{iteration_low_sub} and \Cref{iteration_high_sub}}{Proofs of Lemma~\ref{iteration_low_sub} and Lemma~\ref{iteration_high_sub}}}
We first give an auxiliary result to discuss the scenario where the subproblems in~\eqref{PL} are solved exactly.
\begin{lemma}\label[lemma]{lemma_path_traj}
 Suppose that Assumption \ref{ass:sharpness} holds. Let $k\in\mathbb{N}$.

 (a) If $\Delta(x^k)\leq \lambda_s/(2L)$ and $x^{k+1} = S_{t_k}(x^k)$, then
 \begin{equation}\label{eq_exact_a1}
 {
 \begin{aligned}
     \Delta(x^{k+1})\leq (2L/\lambda_s)\Delta^2(x^k),\quad \forall~t_k\in[2\Delta(x^k)/\lambda_s, L^{-1}].
\end{aligned}}\end{equation}

 (b) If $\Delta(x^{k})\leq \lambda_s/(4L)$, then for all $t_k\in(0, L^{-1}]$,
\begin{subequations}
{
\begin{align}
    F_{t_{k}}(x^{k};x^{k}) - F_{t_{k}}(S_{t_{k}}(x^{k});x^{k}) &\geq \frac{5}{8}\lambda_s\min\Big\{1,~\frac{\lambda_st_k}{2\Delta(x^k)}\Big\}\Delta(x^{k}),\label{eq_update_F}\\
     \|x^{k} - S_{t_{k}}(x^{k})\|_2 &\geq \frac{1}{2}\min\Big\{1,~\frac{\lambda_st_k}{2\Delta(x^k)}\Big\}\Delta(x^{k}). \label{eq_update_x}
 \end{align}}\end{subequations}
\end{lemma}
\begin{proof}(\textbf{Part a}) Without loss of generality, we assume that $\Delta(x^k) = \|x^k-x_\star\|_2$.
Since $\Delta(x^k)\leq \lambda_s/(2L)$, \Cref{rel_lamL} implies that $\Delta(x^k)\leq \|x_\star\|_2/\sqrt{2}$. Moreover, as $x^{k+1} = S_{t_k}(x^k)$, invoking Lemma \ref{lemma_same_closeness} with $\rho_h = 0$, for which case $M_0=2$ (hence, we have $\Delta(x^k)\leq \|x_\star\|_2/\sqrt{M_0}$), \eqref{proof42_same_min} implies that $\mathbf{1}\left[\Delta(x^k) = \|x^k - x_\star\|_2\right] = \mathbf{1}\left[\Delta(x^{k+1}) = \|x^{k+1} - x_\star\|_2\right]$. Therefore, we have $\Delta(x^{k+1}) = \|x^{k+1}-x_\star\|_2$.
Using $\Delta(x^k) = \|x^k-x_\star\|_2$, $\Delta(x^{k+1}) = \|x^{k+1}-x_\star\|_2$, and $\frac{L}{2} - \frac{1}{2t_k}\leq 0$, we get
{\small
\begin{align*}
    \left(\frac{L}{2} - \frac{1}{2t_k}\right)\|x^k - x^{k+1}\|_2^2
    \leq \left(\frac{L}{2} - \frac{1}{2t_k}\right)\left(\Delta(x^k) - \Delta(x^{k+1})\right)^2.
\end{align*}}Next we use it within \eqref{ineq:descent_general} for $x=x^{k+1}$. Note that $\varepsilon_{t_k}(x^{k+1};x^k) = 0$ since $x^{k+1} = S_{t_k}(x^k)$, and we also have $F(x^{k+1}) - F(x_\star)\geq \lambda_s\Delta(x^{k+1})$ due to \eqref{ineq:sharpness}; thus,
$\lambda_s\Delta(x^{k+1}) + \frac{1-\beta}{2t_k}\Delta^2(x^{k+1})\leq (\frac{L}{2}+\frac{1}{2t_k})\Delta^2(x^k) + (\frac{L}{2} - \frac{1}{2t_k})(\Delta(x^k) - \Delta(x^{k+1}))^2$. Letting $\beta\rightarrow 0$, we further get
{\small
$$(t_k^{-1}-L/2)\Delta^2(x^{k+1})+\left(\lambda_s-(t_k^{-1} - L)\Delta(x^{k})\right)\Delta(x^{k+1}) - L\Delta^2(x^k)\leq 0.$$}Note that $t_k^{-1} - L/2\geq t_k^{-1}/2$ and $(t_k^{-1} - L)\Delta(x^k)\leq \lambda_s/2$ since $L^{-1}\geq t_k\geq 2\Delta(x^k)/\lambda_s$; therefore, we have
$t_k^{-1}\Delta^2(x^{k+1})+\lambda_s\Delta(x^{k+1}) - 2L\Delta^2(x^k)\leq 0$, which also implies that
{\small
$$\Delta(x^{k+1})\leq 4L\Delta^2(x^{k})/\left(\sqrt{\lambda_s^2+8t_k^{-1}L\Delta^2(x^k)}+\lambda_s\right)\leq 2L\Delta^2(x^k)/\lambda_s.$$}
(\textbf{Part b}) Recall that \eqref{ineq:sharpness} implies $F(x^k) - F(x_\star)\geq \lambda_s \Delta(x^k)$. Moreover, \eqref{eq_proof_eq_low_decrease} holds since we assume $\Delta(x^k)\leq \lambda_s/(4L)$. These two relations directly lead to \eqref{eq_update_F} since $F_{t_k}(x^k;x^k)=F(x^k)$.
Next, we focus on the proof of \eqref{eq_update_x}. Without loss of generality, we assume that $\|x^k - x_\star\|_2 = \Delta(x^k)$.

First, we consider $t_k\in [2\Delta(x^k)/\lambda_s,L^{-1}]$. From \Cref{rel_lamL}, we get $\Delta(x^k) \leq \lambda_s/(4L)\leq \|x_\star\|_2/8$. Therefore, \Cref{lemma_same_closeness}
with $\rho_h = 0$ and $M_0=2$ implies that $\Delta(S_{t_k}(x^k))=\|S_{t_k}(x^k) - x_\star\|_2$ since $\Delta(x^k)\leq \|x_\star\|_2/\sqrt{M_0}$. In addition, since $\Delta(x^k) \leq \lambda_s/(4L)$, \eqref{eq_exact_a1} implies
$\Delta(x^{k+1})\leq \frac{2L}{\lambda_s}\cdot\frac{\lambda_s}{4L}\Delta(x^k) = \Delta(x^k)/2$; hence,
\begin{equation}\label{eq00_lower_bound_quant1}
    \|x^k - S_{t_k}(x^k)\|_2 \geq \|x^k - x_\star\|_2 - \|S_{t_k}(x^k) - x_\star\|_2\geq \Delta(x^k)/2.
\end{equation}
Next, we consider $t_k\in(0, t')$ where ${t'} \triangleq 2\Delta(x^k)/\lambda_s$. By strong convexity of $F_{t'}(\cdot;x^k)$ and $F_{t_k}(\cdot;x^k)$ defined in \eqref{def-Ft}, $F_{t'}(S_{t_k}(x^k);x^k) - F_{t'}(S_{t'}(x^k);x^k)\geq \frac{1}{2 t'}\|S_{t_k}(x^k) - S_{t'}(x^k)\|_2^2$
and
$F_{t_k}(S_{t'}(x^k);x^k) - F_{t_k}(S_{t_k}(x^k);x^k)\geq \frac{1}{2t_k}\|S_{t_k}(x^k) - S_{t'}(x^k)\|_2^2$; hence,
for $u \triangleq x^k - S_{t'}(x^k)$ and $v \triangleq x^k - S_{t_k}(x^k)$, {\small
$$F(S_{t_k}(x^k);x^k) - F(S_{t'}(x^k);x^k)+\frac{1}{2t'}(\|v\|_2^2 - \|u\|_2^2)\geq \frac{1}{2t'}\|u-v\|_2^2,$$
$$F(S_{t'}(x^k);x^k) - F(S_{t_k}(x^k);x^k) + \frac{1}{2t_k}(\|u\|_2^2 - \|v\|_2^2)\geq \frac{1}{2t_k}\|u-v\|_2^2.$$}Adding these two inequalities, we have
\begin{equation}\label{eq_comp_update_stepsize}
{
\begin{aligned}
    \left(\frac{1}{2t_k} - \frac{1}{2t'}\right)(\|u\|_2^2 - \|v\|_2^2)\geq \left(\frac{1}{2t_k} + \frac{1}{2t'}\right)\|u - v\|_2^2,
\end{aligned}}\end{equation}
which further implies that $\|u\|_2\geq \|v\|_2$. Next, we consider two cases: $\norm{u}_2=\norm{v}_2$ (\textbf{case 1}), and $\norm{u}_2>\norm{v}_2$ (\textbf{case 2}). First, we consider (\textbf{case 1}), i.e.,
$\|x^k - S_{t'}(x^k)\|_2 = \|u\|_2 = \|v\|_2 = \|x^k - S_{t_k}(x^k)\|_2$. Note that \eqref{eq00_lower_bound_quant1} implies that $\|x^k - S_{t'}(x^k)\|_2\geq \Delta(x^k)/2$ since $t'=2\Delta(x^k)/\lambda_s$; therefore, 
\begin{equation}\label{eq_temp_size_update}
    \|x^k - S_{t_k}(x^k)\|_2 = \|x^k - S_{t'}(x^k)\|_2\geq \Delta(x^k)/2,
\end{equation}
implying that \eqref{eq_update_x} holds. Second, we consider (\textbf{case 2}). Since $\|u\|_2 > \|v\|_2$, \eqref{eq_comp_update_stepsize} implies
$(t'-t_k)(\|u\|_2^2 - \|v\|_2^2)\geq (t'+t_k)\|u - v\|_2^2\geq (t'+t_k)(\|u\|_2 - \|v\|_2)^2$. Thus,
$$(t'-t_k)(\|u\|_2 + \|v\|_2)\geq (t'+t_k)(\|u\|_2 - \|v\|_2),$$ which gives $\norm{v}_2\geq \frac{t_k}{t'}\norm{u}_2$, i.e.,
\begin{equation}
    \label{eq00_lower_bound_quant2}
    \|x^k - S_{t_k}(x^k)\|_2\geq \frac{1}{2}t_k\lambda_s\|x^k - S_{t'}(x^k)\|_2/\Delta(x^k),
\end{equation}
where we used $t' = 2\Delta(x^k)/\lambda_s$. Similar to (\textbf{case 1}), we have $\|x^k - S_{t'}(x^k)\|_2\geq \Delta(x^k)/2$ due to \eqref{eq00_lower_bound_quant1}. Therefore, using this relation within \eqref{eq00_lower_bound_quant2}, we get $\|x^k - S_{t_k}(x^k)\|_2\geq \frac{1}{4}t_k\lambda_s$.
Combining this bound with \eqref{eq00_lower_bound_quant1} and \eqref{eq_temp_size_update} completes the proof of \eqref{eq_update_x}.
\end{proof}

Next, we provide the proofs for Lemmas \ref{iteration_low_sub} and \ref{iteration_high_sub}.

\begin{proof}[Proof of Lemma \ref{iteration_low_sub}]Given $k\geq 0$ such that $\Delta(x^k)\leq\lambda_s/(4L)$,
    \Cref{bound_b2} shows that $\|B_k\|_2\leq B(\frac{\lambda_s}{4L})$. Let $j\in\mathbb{N}$ such that
    $j\geq \max\left\{0,\left\lceil {M^{\bf LAC}_k}\right\rceil-2\right\}$.
    From \eqref{ineq_sub_rate}, {\small
    $$H_k(z_{j+1}^k) - D_k(\lambda_{j+1}^k)\leq \frac{C_0t_km\|B_k\|_2^2}{(j+2)^2}\leq \frac{\rho_l}{1+\rho_l}\frac{5}{8}\lambda_s\min\{1,~\lambda_st_k/(2\Delta(x^k))\}\Delta(x^k).$$}By \eqref{eq_update_F}, we have $H_k(z_{j+1}^k) - D_k(\lambda_{j+1}^k)\leq \frac{\rho_l}{1+\rho_l}\left(H_k(\mathbf{0}) - \min_{z\in\mathbb{R}^n} H_k(z)\right)$. Thus, {\small
\begin{align*}
H_k(z_{j+1}^k) - D_k(\lambda_{j+1}^k)
&\leq  \rho_l\Big(-H_k(z^k_{j+1}) +D_k(\lambda^k_{j+1}) + H_k(\mathbf{0}) - \min_{z\in\mathbb{R}^n} H_k(z)\Big),
\end{align*}}which further implies that $H_k(z_{j+1}^k) - D_k(\lambda_{j+1}^k)\leq  \rho_l\Big(H_k(\mathbf{0}) - H_k(z^k_{j+1})\Big)$,
where we used weak duality, i.e., $D_k(\lambda_{j+1}^k)- \min_{z\in\mathbb{R}^n} H_k(z)\leq 0$. Therefore, \lac{} in \eqref{eq:low-high-practical} holds within
$N_k$ inner iterations and $N_k \leq \max\{1,{M^{\bf LAC}_k}\}$. To verify ${M^{\bf LAC}_k}\geq 1$, note that $\sqrt{m}\max_{i\in[m]}\|a_i\|_2\geq\|A\|_F\geq\|A\|_2$. Thus, the definition of $B(\cdot)$, $L=2\|A\|_2^2/m$, and Lemma~\ref{rel_lamL} imply $\sqrt{m}B(\lambda_s/(4L))\geq L(\|x_\star\|_2+\lambda_s/(4L))\geq 2\lambda_s$. Moreover, for $\Delta(x^k)>0$, we have $\min\{1,\lambda_st_k/(2\Delta(x^k))\}\leq\lambda_st_k/(2\Delta(x^k))$. Combining these inequalities with the definition of $M^{\bf LAC}_k$, we obtain $(M^{\bf LAC}_k)^2\geq \frac{16C_0mB^2(\lambda_s/(4L))}{5\lambda_s^2}\frac{1+\rho_l}{\rho_l}\geq\frac{64C_0}{5}\frac{1+\rho_l}{\rho_l}>1$, where we used $C_0\geq2$ and $\rho_l>0$. Therefore, $\mathbb{N}_+\ni N_k\leq M^{\bf LAC}_k$, completing the proof. \end{proof}

\begin{proof}[Proof of Lemma \ref{iteration_high_sub}]
Given $k\geq 0$ such that $\Delta(x^k)\leq\lambda_s/(4L)$,
    \Cref{bound_b2} shows that $\|B_k\|_2\leq B(\frac{\lambda_s}{4L})$. Let $j\in\mathbb{N}$ such that
    $j\geq \max\left\{0,\left\lceil {M^{\bf HAC}_k}\right\rceil-2\right\}$.
    From \eqref{ineq_sub_rate}, {\small
    $$H_k(z_{j+1}^k) - D_k(\lambda_{j+1}^k)\leq \frac{C_0t_km\|B_k\|_2^2}{(j+2)^2} \leq \frac{\rho_h}{1+\rho_h}\frac{1}{4t_k}\left(\frac{1}{2}\min\{1,~\lambda_st_k/(2\Delta(x^k))\}~{\Delta(x^k)}\right)^2.$$}By \eqref{eq_update_x}, we have
\begin{equation}\label{eq_proof_iter_high_temp}
{
\begin{aligned}
         H_k(z_{j+1}^k) - D_k(\lambda_{j+1}^k)\leq \frac{\rho_h}{1+\rho_h}\frac{1}{4t_k}\|x^k-S_{t_k}(x^k)\|_2^2.
\end{aligned}}\end{equation}
For $z^k_{\star} \triangleq \argmin_{z\in\mathbb{R}^n} H_k(z)$, we have $\|z^k_{\star}\|_2 = \|x^k - S_{t_k}(x^k)\|_2$; moreover, using $\frac{1}{t_k}$-strong convexity of $H_k(\cdot)$ and weak duality together, we obtain
$\frac{1}{2t_k}\|z^k_{j+1} - z^k_{\star}\|_2^2\leq H_k(z^k_{j+1}) - \min_{z\in\mathbb{R}^n} H_k(z)\leq H_k(z^k_{j+1}) - D_k(\lambda_{j+1}^k)$.
Therefore, \eqref{eq_proof_iter_high_temp} implies that $\|z^k_{j+1} - z^k_{\star}\|_2^2\leq \rho_h/\left(2(1+\rho_h)\right)\|z^k_{\star}\|_2^2.$
By the Cauchy-Schwarz inequality, we further have
{\small
\begin{align*}
    \frac{\rho_h}{2t_k}\|z^k_{j+1}\|_2^2  \geq \frac{\rho_h}{4t_k}\|z^k_{\star}\|_2^2- \frac{\rho_h}{2t_k}\|z^k_{j+1} - z^k_{\star}\|_2^2
    \geq \frac{\rho_h}{4t_k}\|z^k_{\star}\|_2^2- \frac{\rho_h^2}{4t_k(1+\rho_h)}\|z^k_{\star}\|_2^2 = \frac{\rho_h}{4t_k(1+\rho_h)}\|z^k_{\star}\|_2^2.
\end{align*}}Therefore, \hac{} in \eqref{eq:low-high-practical} holds within
$N_k \leq \max\{1,{M^{\bf HAC}_k}\}$ inner iterations. Finally, the preceding proof establishes that $\sqrt{m}B(\lambda_s/(4L))\geq2\lambda_s$. Since $t_k/(\Delta(x^k)\min\{1,\lambda_st_k/(2\Delta(x^k))\})\geq2/\lambda_s$ for $\Delta(x^k)>0$, the definition of $M^{\bf HAC}_k$ yields $M^{\bf HAC}_k\geq\frac{8\sqrt{C_0m}B(\lambda_s/(4L))}{\lambda_s}\sqrt{\frac{1+\rho_h}{\rho_h}}\geq16\sqrt{C_0}\sqrt{\frac{1+\rho_h}{\rho_h}}>1$. Therefore, $\mathbb{N}_+\ni N_k\leq M^{\bf HAC}_k$, completing the proof.
\end{proof}

Finally, we are ready to provide the proofs for Theorems \ref{overall_complexity_diminish_stepsize_gen} and \ref{overall_complexity_diminish_stepsize}.
\subsection{\texorpdfstring{Proof of~\Cref{overall_complexity_diminish_stepsize_gen}}{Proof of~Theorem~\ref{overall_complexity_diminish_stepsize_gen}}}
(\textbf{Part a}) By Lemma~\ref{sufficiency_pd}, \lac{} implies that \lace{} in \eqref{low-high-0} also holds for all $k\in\mathbb{N}$. As $\Delta(x^0)> \epsilon$, we have $\mathbb{N}_+\ni K_\epsilon\geq 1$.
Next, we use induction to show the following relations hold simultaneously for $k\in\mathbb{N}$:
\begin{subequations}
\label{eq:induction-LAC}
\begin{align}
    F(x^{k}) - F(x_\star)&\leq \left(C(\rho_l)\right)^k\big(F(x^0) - F(x_\star)\big),\label{proof_thm_46_eq0}\\
\Delta(x^k)&\leq \min\{\lambda_s/(4L),(g_HL)^{-1}\},\label{proof_thm_46_eq00}\\
\exists~g_k&\in[g_L,g_H]:\ t_k = g_k\Delta(x^k), \label{proof_thm_46_eq000}
\end{align}
\end{subequations}
where $C(\rho)=1 - \frac{5\min\{\lambda_sg_L/2,1\}}{8(1+\rho)}$ for $\rho>0$. For $k=0$, \eqref{proof_thm_46_eq0} holds. Moreover, by \Cref{lip_F2}, we have
$F(x^0) - F(x_\star)\leq \min\{\lambda_s^2/(4L),\lambda_s(g_HL)^{-1}\}$ and
$\Delta(x^0)\leq \min\{\lambda_s/(4L),(g_HL)^{-1}\}$. Therefore, \eqref{proof_thm_46_eq00} holds for $k=0$, which also implies that $g_0\Delta(x^0)\leq g_H\Delta(x^0)\leq L^{-1}$ since $g_0\leq g_H$. Recall that according to \Cref{prop_alpha_and_tk}(b), we have $t_0=\min\{g_0\Delta(x^0),L^{-1}\}$ since $\Delta(x^0)\leq \|x_\star\|_2$ due to Lemma \ref{rel_lamL}. Thus, \eqref{proof_thm_46_eq000} holds for $k=0$ as well. Next, we assume that \eqref{eq:induction-LAC} holds for some $k\geq 0$. Together with \eqref{proof_thm_46_eq00}, \eqref{proof_thm_46_eq000} and ${g_{k}}\geq g_L$, \Cref{lemma_traj_gen_tk}(a) implies $F(x^{k+1}) - F(x_\star)\leq \left(1 - \frac{5\min\{\lambda_sg_L/2,1\}}{8(1+\rho_l)}\right)\big(F(x^k) - F(x_\star)\big)$; hence, using it with \eqref{proof_thm_46_eq0} establishes \eqref{proof_thm_46_eq0} for $k+1$. Since
$F(x^{k+1}) - F(x_\star)\leq F(x^0) - F(x_\star)\leq \min\{\lambda_s^2/(4L),\lambda_s/(g_HL)\}$ holds, using \eqref{ineq:sharpness} we get
\eqref{proof_thm_46_eq00} for $k+1$. Finally, \Cref{prop_alpha_and_tk}(b) implies that \eqref{proof_thm_46_eq000} holds for $k+1$ as well because $g_{k+1}\Delta(x^{k+1})\leq g_H\Delta(x^{k+1})\leq L^{-1}$ and $\Delta(x^{k+1})\leq \|x_\star\|_2$ due to Lemma \ref{rel_lamL}. This completes the induction.

By Lemma \ref{iteration_low_sub} together with \eqref{proof_thm_46_eq00} and \eqref{proof_thm_46_eq000}, we have
\begin{equation}\label{eq_low_iter_upper}
{
\begin{aligned}
    \sup_{k\in\mathbb{N}} N_k\leq \sup_{k\in\mathbb{N}} M_1\sqrt{\frac{g_k(1+\rho_l)}{\lambda_s\rho_l\min\{1,\lambda_sg_k/2\}}}\leq M_1\sqrt{\frac{g_H(1+\rho_l)}{\lambda_s\rho_l\min\{1,\lambda_sg_H/2\}}}.
    \end{aligned}}\end{equation}
W.L.O.G suppose $\Delta(x^0)=\norm{x^0-x_\star}_2$. Due to \Cref{rel_lamL}, $\Delta(x^0)\leq \lambda_s/(4L)\leq \|x_\star\|_2$; hence, using \Cref{Lip_F} with $r=\norm{x_\star}_2$, we get $F(x^0) - F(x_\star)\leq 2\Delta(x^0)L\|x_\star\|_2$ since $\max\{\Delta(x^0),~\Delta(x_\star)\}\leq \norm{x_\star}_2$. Recall that $K_\epsilon=\inf\{k\in\mathbb{N}_+:~\Delta(x^k)\leq \epsilon\}$; therefore, from \eqref{ineq:sharpness}, $F(x^{K_\epsilon-1}) - F(x_\star)\geq \lambda_s\Delta(x^{K_\epsilon-1})\geq \lambda_s\epsilon$. Thus, by \eqref{proof_thm_46_eq0},
$\left(C(\rho_l)\right)^{K_\epsilon - 1}\geq \lambda_s \epsilon/\left(2\Delta(x^0)L\|x_\star\|_2\right)$, implying
\begin{equation}
\label{eq:K-epsilon}
{
\begin{aligned}
    {K_\epsilon\leq \log\Big(\frac{2\Delta(x^0)L\|x_\star\|_2}{\lambda_s\epsilon C(\rho_l)}\Big)/\log\Big(\frac{1}{C(\rho_l)}\Big).}
\end{aligned}}\end{equation}
Since $N(\epsilon)\leq K_\epsilon \sup_{k\in\mathbb{N}} N_k$, using \eqref{eq_low_iter_upper}, we get the desired result for part (a).

(\textbf{Part b}) By Lemma~\ref{sufficiency_pd}, \hac{} implies that \hace{} in \eqref{low-high-0} also holds for all $k\in\mathbb{N}$.
Note that \eqref{eq_proof_high_temp2}
implies that \lace{} in \eqref{low-high-0} holds for any $k\in\mathbb{N}$ for $\rho_l = \bar{\rho} \triangleq 2\rho_h/(1-4\rho_h)$. Thus, similar to the proof of (\textbf{Part a}), \eqref{proof_thm_46_eq0}, \eqref{proof_thm_46_eq00} and \eqref{proof_thm_46_eq000} with $\rho_l = \bar{\rho}$ hold for any $k\in\mathbb{N}$, which implies that \eqref{eq:K-epsilon} holds for $\rho_l=\bar{\rho}$.
Moreover, from Lemma \ref{iteration_high_sub} together with \eqref{proof_thm_46_eq00} and \eqref{proof_thm_46_eq000}, we get
{\small
$$\sup_{k\in\mathbb{N}} N_k\leq \sup_{k\in\mathbb{N}} \frac{M_2g_k}{\min\{1,\lambda_sg_k/2\}}\sqrt{\frac{1+\rho_h}{\rho_h}}\leq \frac{M_2g_H}{\min\{1,\lambda_sg_H/2\}}\sqrt{\frac{1+\rho_h}{\rho_h}}.$$}Since $N(\epsilon)\leq K_\epsilon \sup_{k\in\mathbb{N}} N_k$, we get the desired result for part (b).
\subsection{\texorpdfstring{Proof of \Cref{overall_complexity_diminish_stepsize}}{Proof of Theorem~\ref{overall_complexity_diminish_stepsize}}}
By Lemma~\ref{sufficiency_pd}, \hac{} implies that \hace{} in \eqref{low-high-0} also holds for all $k\in\mathbb{N}$. We use induction to show that the following relations hold simultaneously for $k\in\mathbb{N}$:
\begin{subequations}
\label{eq:induction-HAC}
\begin{align}
    \Delta(x^{k}) &\leq (6\rho_h/(\lambda_sg_L))^k \Delta(x^0),\label{proof_thm47_eq0}\\
    \Delta(x^k)&\leq \min\{\rho_h/(2Lg_L),~\|x_\star\|_2/\sqrt{M_0},~1/(g_HL)\},\label{proof_thm47_eq000}\\
    \exists g_k &\in[g_L,g_H]:\ t_k = g_k\Delta(x^k).\label{proof_thm47_eq00}
\end{align}
\end{subequations}
For $k=0$, \eqref{proof_thm47_eq0} trivially holds and \eqref{proof_thm47_eq000} is true due to hypothesis. Moreover, since $\rho_h\in (0,1/4)$ and $g_L\geq 2/\lambda_s$, we have $\Delta(x^0)\leq \rho_h/(2Lg_L)\leq \lambda_s/(16L)$; hence,
\Cref{rel_lamL} implies that $\Delta(x^0)\leq \|x_\star\|_2$. Thus, by \Cref{prop_alpha_and_tk}(b), we have $t_0 = \min\{g_0\Delta(x^0),L^{-1}\}$ for some $g_0\in[g_L,g_H]$. In addition, \eqref{proof_thm47_eq000} implies that $g_0\Delta(x^0)\leq g_H\Delta(x^0)\leq L^{-1}$, which leads to \eqref{proof_thm47_eq00} for $k=0$. Thus, the base case ($k=0$) for induction holds. Next, we assume that \eqref{eq:induction-HAC} holds for some $k\in\mathbb{N}$, and we prove that it also holds for $k+1$. The induction hypothesis on \eqref{proof_thm47_eq0} implies that $\Delta(x^k)\leq \Delta(x^0)$. Thus, $\Delta(x^k)\leq \min\{\lambda_s/(4L),~\norm{x_\star}_2/\sqrt{M_0},~1/(g_H L)\}$, and $t_k\in[2\Delta(x^k)/\lambda_s,~L^{-1}]$ since $g_L\geq 2/\lambda_s$ and $\Delta(x^0)\leq \frac{1}{g_H L}$. Therefore, \Cref{lemma_path_traj_high} together with \eqref{proof_thm47_eq00}
implies that
$\lambda_s\Delta(x^{k+1})\leq (4\rho_h/g_k+4L\Delta(x^{k}))\Delta(x^{k}).$ By $g_k\geq g_L$ and $\Delta(x^{k})\leq \rho_h/(2Lg_L)$, we further have
$$\lambda_s\Delta(x^{k+1})\leq (4\rho_h/g_L+4L\rho_h/(2Lg_L))\Delta(x^{k}) = 6\rho_h\Delta(x^k)/g_L.$$
Together with \eqref{proof_thm47_eq0} for $k$, this proves \eqref{proof_thm47_eq0} for $k+1$. Moreover, observing that the rate coefficient in \eqref{proof_thm47_eq0} satisfies
$\frac{6\rho_h}{\lambda_sg_L}\leq \frac{3}{4}$ as $g_L\geq 2/\lambda_s$ and $\rho_h\in (0,1/4)$, we also get $\Delta(x^{k+1})\leq 3\Delta(x^k)/4$. Thus, \eqref{proof_thm47_eq000} for $k$ immediately implies that \eqref{proof_thm47_eq000} also holds for $k+1$. Finally,
since $\Delta(x^{k+1})\leq \Delta(x^0)\leq \lambda_s/(4L)\leq \|x_\star\|_2$, according to \Cref{prop_alpha_and_tk}(b), we have $t_{k+1} = \min\{g_{k+1}\Delta(x^{k+1}),L^{-1}\}$ for some $g_{k+1}\in[g_L,g_H]$, and $\Delta(x^{k+1})\leq (g_H L)^{-1}$ due to \eqref{proof_thm47_eq000} shows that \eqref{proof_thm47_eq00} holds for $k+1$, establishing the induction.
Since \eqref{proof_thm47_eq00} holds for all $k\in\mathbb{N}$, we have $2\Delta(x^k)/\lambda_s\leq g_L\Delta(x^k)\leq t_k\leq g_H \Delta(x^k)$ for $k\geq 0$ since $g_L\geq 2/\lambda_s$. Using these inequalities within the bound for $N_k$ in
\Cref{iteration_high_sub}, we get
    $\sup_{k\in\mathbb{N}} N_k\leq M_2g_H\sqrt{\frac{1+\rho_h}{\rho_h}}$.
As $\Delta(x^{K_\epsilon - 1})> \epsilon$, \eqref{proof_thm47_eq0} implies
$\left(\frac{6\rho_h}{\lambda_s g_L}\right)^{K_\epsilon - 1}\geq \frac{\epsilon}{\Delta(x^0)}$; thus,
$K_\epsilon\leq \log \Big(\frac{\lambda_s\Delta(x^0)g_L}{6\rho_h}\cdot\frac{1}{\epsilon}\Big)/\log\Big(\frac{\lambda_sg_L}{6\rho_h}\Big)$.
Since $N(\epsilon)\leq K_\epsilon \sup_{k\in\mathbb{N}} N_k$,
we get the desired result.

\section{Numerical Experiments}\label{sec:main_numerical}
In this section, we conduct numerical experiments on the RPR problem in \eqref{duchi_l1_ori}. We tested the following algorithms:
\begin{itemize}
	\item[(i)] {\textbf{\gsub}}: The subgradient method with geometrically decaying step sizes is proposed in \cite{davis2018subgradient}
 for solving \eqref{duchi_l1_ori}---\gsub{} updates are stated in~\eqref{descent_subgradient_geome}. We picked the best performing contraction coefficient $q\in(0,1)$ among all the values tested in \cite{davis2018subgradient}, i.e., $q\in\{0.983,\allowbreak 0.989,\allowbreak 0.993,\allowbreak 0.996,\allowbreak 0.997\}$; on the other hand, \cite{davis2018subgradient} does not specify how to choose $\lambda_0$, and we set $\lambda_0 = 0.1\|x^0\|_2$.
 \item[(ii)] \ipl{}\texttt{-LAC}, \ipl{}\texttt{-HAC}:
 \ipl{} algorithm with fixed step sizes $t_k = 1/L$ and the subproblems in \eqref{PL} are solved using \apg{}. ``\emph{\texttt{-LAC}}" and ``\emph{\texttt{-HAC}}"
    indicate whether (\textbf{LAC}) or (\textbf{HAC}) is used.
  \item[(iii)] \adaipl{}\texttt{-LAC}, \adaipl{}\texttt{-HAC}:
 These methods, stated in~Algorithm \ref{alg:adaptive-IPL}, use adaptive step sizes given in \eqref{choice_t_diminish}. The subproblems \eqref{PL} are solved using \apg{}, which performed better
 compared to \apd{} in our tests.
 ``\emph{\texttt{-LAC}}" and ``\emph{\texttt{-HAC}}"
    indicate whether (\textbf{LAC}) or (\textbf{HAC}) is used.
 \item[(iv)] \textbf{\adasub{}}: It is the subgradient method, stated in Algorithm \ref{alg:adaptive-sub}, with adaptive step sizes given in \eqref{choice_alpha_subg} and \eqref{iter_adasub}.
\end{itemize}
Rather than the original proximal linear algorithm proposed in \cite{duchi2019solving}, we tested our proposed algorithms against \ipl{}\texttt{-LAC} and \ipl{}\texttt{-HAC} from \cite{zheng2023new} as these methods have already demonstrated better performances when compared to the prox-linear method~\cite{duchi2019solving}. We conduct some preliminary tests with \texttt{Robust-AM}~\cite{kim2024robust} which iteratively uses ADMM to inexactly solve subproblems of the form $x^{k+1} = \argmin_{x\in\mathbb{R}^n} \frac{1}{m}\sum_{i=1}^m |(a_i^\top x) - \mbox{sign}(a_i^\top x^k)\sqrt{b_i}|$; and on these tests we found \texttt{Robust-AM} inefficient for computing a high-accuracy solution---it was an order of magnitude slower than the other competing methods; that is why we did not include it in our tests reported below. We did not test \psub{} either because according to~\cite{davis2020nonsmooth}, it only works for the noiseless setting. We will always let $\tilde{p} = 1/2$ for \adasub{} and \adaipl{}. All the methods are initialized from the same point $x^0$, which is generated by \cite[Algorithm 3]{duchi2019solving}.
\subsection{\texorpdfstring{Selecting $G$ for \adasub{} and \adaipl{}}{Selecting G for \adasub{} and \adaipl{}}}
\label{sec:G-selection}
We provide practical guidance on choosing $G$.

For \adasub{}, our goal is to approximate $F(x^k) - F(x_\star)$ used in \psub{} from \cite{davis2020nonsmooth} with $Gr^{\tilde{p}}(x^k)$ for $\tilde p=1/2$. Since $F(x^k) - F(x_\star)$ is equal to the mean of $\{r_i(x^k) - r_i(x_\star)\}_{i=1}^m$, $r^{\tilde{p}}(x^k)$ for $\tilde p = 1/2$ represents the median of
$\{r_i(x^k)\}_{i=1}^m$, and $r_i(x_\star) = 0$ holds for most $i\in [m]$, it is highly likely that $(F(x^k) - F(x_\star))/r^{\tilde{p}}(x^k)$ is not far from $1$. Thus, we can let $G$ be close to 1, relying on this observation. We also choose a slightly smaller value for conservativeness, since overly large $G$ might lead to divergence, whereas overly small $G$ will not. For \adaipl{}, \cref{overall_complexity_diminish_stepsize_gen,overall_complexity_diminish_stepsize} imply that the ideal choice for $t_k$ is $2\Delta(x^k)/\lambda_s$ when $\Delta(x^k)$ is small. Thus, we wish $Gr^{\tilde{p}}(x^k)$ to be close to $2\Delta(x^k)/\lambda_s$, i.e., $G$ should be close to $2\Delta(x^k)/(\lambda_s r^{\tilde{p}}(x^k))$. Under the same belief that we discussed for \adasub{}, we wish that $G$ is close to $2\Delta(x^k)/(\lambda_s (F(x^k) - F(x_\star)))$.

Next, we discuss the conservativeness in choosing $G$ for \adaipl{}. Since the definition of $t_k$ given in \eqref{choice_t_diminish} provides an upper bound $L^{-1}$, it is relatively safe to use a large $G$ since the total iteration complexity under $G = \infty$ will degrade to that of \ipl{}, which still guarantees a sublinear convergence rate. On the contrary, as shown in Table \ref{tab:imperfect_situation}, the efficiency can be arbitrarily worse if we use an overly small $G$. Thus, we prefer a relatively large $G$ to a potentially small one when uncertainty arises.

Thus, noticing that $2\Delta(x^k)/(\lambda_s (F(x^k) - F(x_\star)))\leq 2/\lambda_s^2$ due to Assumption \ref{ass:sharpness} (sharpness) and \eqref{eq:kappa0} implies that $\lambda_s = L\|x_\star\|_2/(2\kappa_0)$, following the conservativeness, we wish that $G$ is close to $8\kappa_0^2/(L^2\|x_\star\|_2^2)$. Here, $\kappa_0\geq 1$ is the condition number of the robust phase retrieval problem.
$L = 2\|A\|_2^2/m$ is known. $\|x_\star\|_2$ can be approximated by $\|x^0\|_2$ where $x^0$ is the initializer discussed above --see \cite[Section 4.2]{duchi2019solving} for the validity of this approximation. Thus, we can choose $G$ as $8\tilde{G}/(L^2\|x^0\|_2^2)$. Here, $\tilde{G}>0$ reflects our belief on $\kappa_0^2$. It should be large when we expect severe ill-conditioning ($\kappa_0$ is much larger than 1), which occurs under a relatively large $p_{\mathrm{fail}}$ or a large condition number of $A^\top A$. Otherwise, it should be close to $1$. As discussed above, we can use a relatively large $\tilde{G}$ for conservativeness. 
\subsection{Synthetic Data}
\label{sec:synthetic}
We generate synthetic data as in~\cite{duchi2019solving} and \cite{zheng2023new}. Specifically, $a_i$'s are drawn randomly from the normal distribution $\mathcal{N}(0,\mbox{diag}([s_1,s_2\ldots s_n]))$
where $s_i=1-0.75\frac{i-1}{n-1}$ for $i\in[n]$. Throughout the experiments in this subsection,  we set $n = 1500$ and we choose $m$ such that $m/n\in\{5,6,7,8\}$. The entries of $x_\star \in\reals^n$ are drawn uniformly at random from $\{-1,1\}$.
For each value of $m$ tested, the index set $\mathcal{I}_2$ for corrupted measurements is generated by random sampling $\lceil m~p_{\rm fail} \rceil$ elements from $\{1,2,\ldots,m\}$ without replacement, where $p_{\rm fail}\in{\{0.1,0.2\}}$. Corrupted measurements $b_i=\xi_i$ for $i\in\cI_2$
are, conditional on the noiseless measurements, independently drawn from a half-Cauchy distribution with scale $\tilde{M}$, i.e.,
$b_i=\xi_i=\tilde{M}\tan \left(\frac{\pi}{2}U_i\right)$ and $U_i\sim U(0,1)$ for all $i\in \mathcal{I}_{2}$, where $\tilde{M}$ is the sample median of $\{(a_i^\top x_\star)^2\}_{i=1}^m$.
For a given threshold $\epsilon>0$, we call an algorithm successful if it returns some $x_\epsilon\in\reals^n$ such that the relative error
    $\Delta(x_\epsilon)/\|x_\star\|_2\leq \epsilon$.

All algorithms are terminated after such an $x_\epsilon$ is computed. For each $m/n\in\{5,6,7,8\}$, we randomly generate $10$ instances according to the above procedure. For both \ipl{} and \adaipl{}, the parameters are set to $\rho_l = \rho_h = 0.24$ as in~\cite{zheng2023new}; and we choose $G = 1.0$ for \adasub{} and $\tilde{G} = 100$ for \adaipl{}. In these experiments, the success rate is 1 for \adasub{} and close to 1 for \ipl{} and \adaipl{}\footnote{In only one replication for $p_{\rm fail} =0.2,m/n = 6$, all of \gsub{}, \ipl{} and \adaipl{} fail.}, and we find that $q = 0.983$ for \gsub{} consistently yields the best performance in terms of CPU time\footnote{In a small number of replications, \gsub{} fails under $q = 0.983$.}; thus, we compare them based on this parameter choice.
\begin{figure}
    \centering
    \begin{subfigure}[t]{0.49\textwidth}
        \centering
        \includegraphics[width=\linewidth]{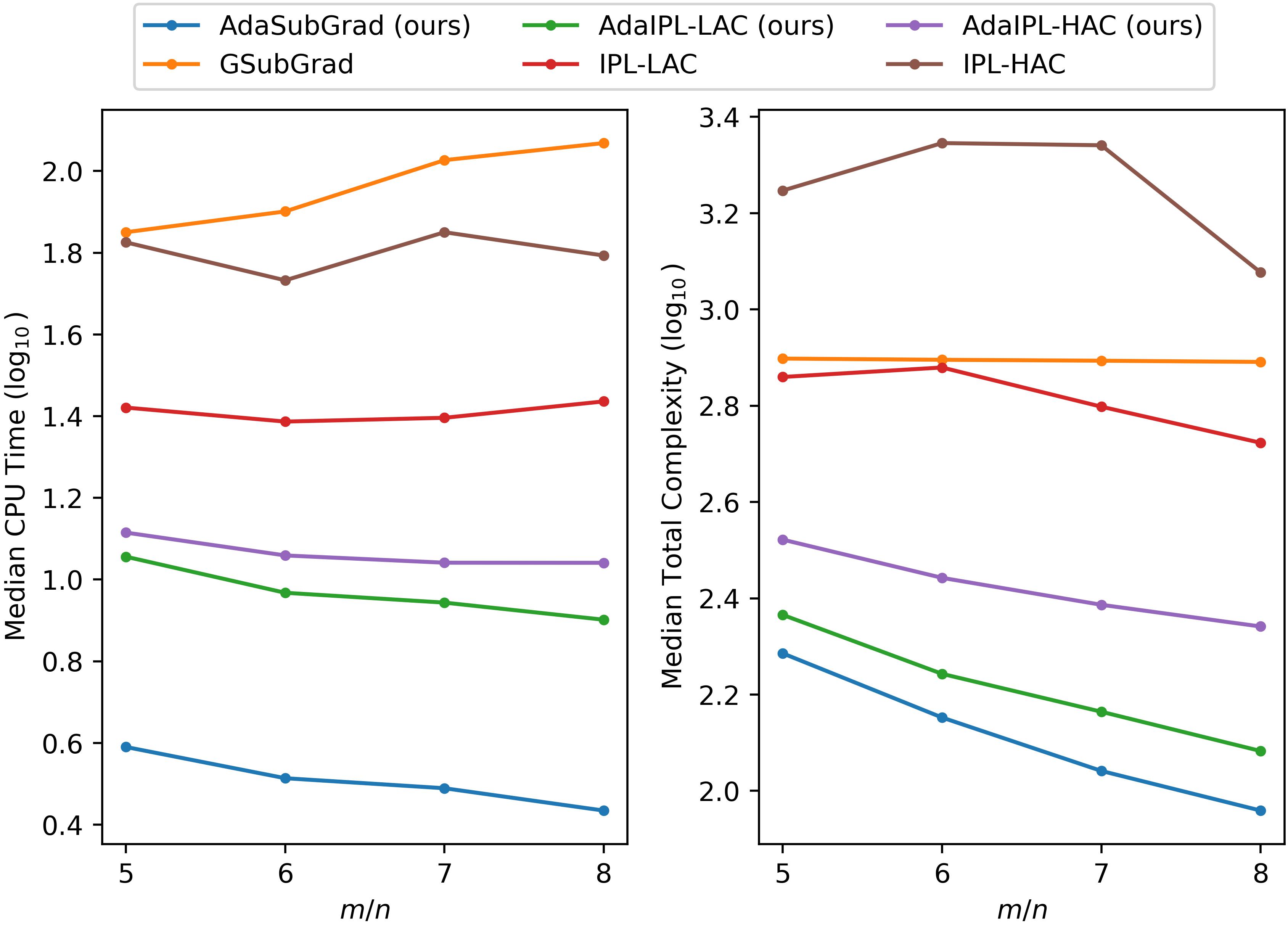}
    \end{subfigure}
    \hfill
    \begin{subfigure}[t]{0.49\textwidth}
        \centering
        \includegraphics[width=\linewidth]{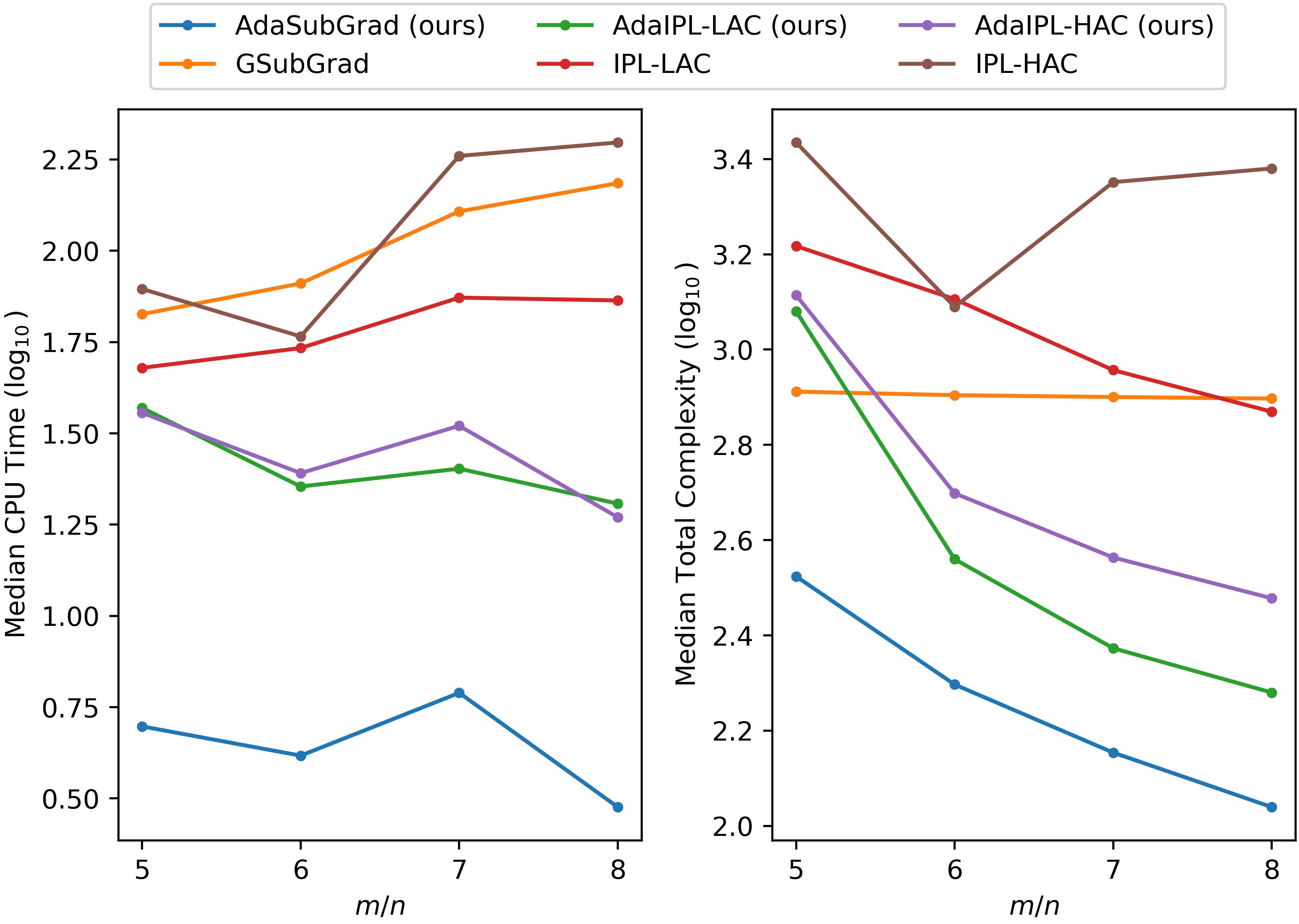}
    \end{subfigure}
    \caption{{Comparisons on Synthetic Datasets. The left and right subfigures show the results for $p_{\rm fail} = 0.1$ and $0.2$, respectively.}}
    \label{comparison_gen2}
\end{figure}
In \Cref{comparison_gen2}, we provide the experimental results
with $\epsilon = 10^{-7}$. The left panel of each subfigure shows the median CPU time in seconds\footnote{The CPU time for initialization and approximately calculating $L$ is negligible in this example.} for all the successful replications. The right panel of each subfigure shows the median of total iteration numbers, i.e., total inner iteration numbers for \ipl{}, and \adaipl{}, and total iteration numbers for \gsub{} and \adasub{}. \Cref{comparison_gen2} shows that both \adasub{} and \adaipl{} perform better than the other algorithms.

Next, Table \ref{tab:detail_comp_gen} provides results for the setting in Figure \ref{comparison_gen2} with $m/n = 8$ and $p_{\rm fail} = 0.1$ to show the performance under different choices of $G$. In this table, median values for CPU time, total subproblem iterations, and main iterations are reported. We test $\tilde{G} \in \{1000, 100, 10, 1\}$ for \adaipl{}, and $G \in \{ 4.0, 3.0, 1.0, 0.1\}$ for \adasub{}. We find that using overly large or overly small $G$ for \adasub{} leads to large numbers of iterations and low efficiency. $G = 4.0$ even leads to divergence. For \adaipl{}, using overly large or overly small $\tilde{G}$ leads to large numbers of total iterations and low efficiency.
In addition, using a larger $\tilde{G}$ leads to a larger ratio of total iterations to main iterations and a smaller main iteration number, which is lower bounded by the \ipl{} counterparts. In addition, when $\tilde{G}$ is very small, the ratio of total iterations to main iterations is less than 3, which means that very few iterations are needed to solve a subproblem.
\begin{table}[htbp]
\centering
{\footnotesize
\begin{tabular}{|c|c|c|c|}
\hline
                     & CPU Time (sec) & Total Iterations &  Main Iterations \\ \hline
\ipl{}\texttt{-LAC}              & 26.37           & 514                   & 11                     \\ \hline
\adaipl{}\texttt{-LAC}: 1000       & 10.22            & 144                     & 11                     \\ \hline
\adaipl{}\texttt{-LAC}: 100        & 9.13            & 121                   & 11                   \\ \hline
\adaipl{}\texttt{-LAC}: 10        & 11.47            & 150                   & 15                   \\ \hline
\adaipl{}\texttt{-LAC}: 1        & 75.19           & 476                     & 209                  \\ \hline
                     & CPU Time (sec) & Total Iterations & Main Iterations \\ \hline
\ipl{}\texttt{-HAC}             & 73.70           & 1548                    & 6                      \\ \hline
\adaipl{}\texttt{-HAC}: 1000      & 21.52           & 430                     & 7                      \\ \hline
\adaipl{}\texttt{-HAC}: 100       & 12.49            & 219                   & 7                      \\ \hline
\adaipl{}\texttt{-HAC}: 10       & 13.11            & 176                     & 16                     \\ \hline
\adaipl{}\texttt{-HAC}: 1       & 73.40           & 458                   & 209                  \\ \hline
                     & CPU Time (sec) & Total Iterations & Main Iterations \\ \hline
\adasub{}: 4.0 & diverge            & diverge                     & diverge                    \\ \hline
\adasub{}: 3.0 & 5.99            & 191                     & 191                    \\ \hline
\adasub{}: 1.0 & 2.98            & 91                     & 91                    \\ \hline
\adasub{}: 0.1 & 15.22           & 471                     & 471                    \\ \hline
\end{tabular}}\caption{{Tests for different $G$ or $\tilde{G}$ values: median values for 10 replications are reported.}}
\label{tab:detail_comp_gen}
\end{table}

\subsection{Image Recovery}\label{sec:image}
We conduct numerical tests on images as in~\cite{duchi2019solving}. In particular, given an RGB image array $X_\star\in \mathbb{R}^{n_1\times n_2\times 3}$, we construct the signal as $x_\star=[\mbox{vec}(X_{\star})^\top \mathbf{0}^\top]^\top\in \mathbb{R}^{n}$ where
$n=\min\{2^s\mid s\in {\mathbb{N}},~2^s\geq 3n_1n_2\}$ and $\mathbf{0}\in\reals^{n-3n_1n_2}$---here, $\mbox{vec}(\cdot)$ vectorizes its argument. Let $H_n\in\frac{1}{\sqrt{n}}\{-1,1\}^{n\times n}$ be the Hadamard matrix. We generate \textit{diagonal} matrices $S_j\in\reals^{n\times n}$ for $j=1,2,\ldots,6$
such that all the diagonal entries are chosen uniformly at random from $\{-1,1\}$. Next, for $m=6n$, we set
$A\triangleq\sqrt{n}[(H_nS_1)^\top (H_nS_2)^\top\ldots (H_nS_6)^\top]^\top\in\reals^{m\times n}$---for this setting, it can be shown that $L=2$. The advantage of such a mapping is that it mimics the fast Fourier transform, and calculating $Ax$
requires $O(m\log n)$ work for $x\in\reals^n$.

\begin{table}[htbp]
  \centering
  {\small
\begin{tabular}{|c|c|}
\hline
\textbf{Method} & \textbf{CPU Time}  \\ \hline
\gsub{}     & 44.28 (7.52)                       \\ \hline
\ipl{}\texttt{-LAC}         & 209.24 (74.91)                     \\ \hline
\ipl{}\texttt{-HAC}        & 237.13 (59.59)                    \\ \hline
\adaipl{}\texttt{-LAC}     & 15.70 (4.90)                       \\ \hline
\adaipl{}\texttt{-HAC}    & 34.29 (8.14)                      \\ \hline
\adasub{} & 3.01 (0.91)                        \\ \hline
\end{tabular}}\caption{{Comparison of CPU time (in minutes) for image recovery problem. Median (interquartile range) values are reported over $10$ replications.}}
  \label{tab:image_rna}
\end{table}

We conduct the test on an RNA image\footnote{https://visualsonline.cancer.gov/details.cfm?imageid=11167}
of size $n = 2^{22}$. We remark that \texttt{Robust-AM} \cite{kim2024robust} cannot be applied here since its ADMM subproblem solver requires calculating the pseudo-inverse of a matrix of size $m\times n$ (see \cite[Eq.(9a)]{kim2024robust}). In the experiment we considered $p_{\textrm{fail}} = 0.1$, and we generate corrupted measurements as in the synthetic datasets. The algorithm parameters are set to $G = 1.0$ for \adasub{}, $\tilde{G} = 10$ for \adaipl{}, and $q = 0.983$ for \gsub{} while keeping the other hyperparameters the same as in \Cref{sec:synthetic}.

All the algorithms are terminated whenever $x_\epsilon$ with a relative error at most $\epsilon=10^{-7}$ is computed. The results are reported in \Cref{tab:image_rna}, where we report the median and interquartile range of CPU times in \textit{minutes} based on $10$ replications.
The results show that both \adasub{} and \adaipl{} perform better than the other algorithms.

\section{Conclusion}\label{sec:conclusions}
In this paper, we propose two adaptive algorithms for solving the robust phase retrieval
problem. Our contribution lies in designing new adaptive step size rules that are based on the quantiles of absolute residuals and are robust to sparse corruptions. Employing adaptive step sizes, both methods show local linear convergence and are robust to hyperparameter selection. Numerical results demonstrate that both \adasub{} and \adaipl{} perform significantly better than the existing state-of-the-art methods tested on both synthetic-data and real-data RPR problems.

\section*{Acknowledgments}
Zhong Zheng and Lingzhou Xue are supported in part by NIH Grant 1R01GM152812.
N. Serhat Aybat is supported in part by ONR Grant N00014-24-1-2666.
Shiqian Ma is supported in part by ONR Grant N00014-24-1-2705, and NSF Grants CCF-2626366 and ECCS-2639042.


\bibliographystyle{plain}
\bibliography{references}

\end{document}